\documentclass{article}
\usepackage{graphicx} 
\usepackage{amsmath}
\usepackage{amssymb}
\usepackage{amsfonts}
\usepackage{amsthm}
\usepackage{xcolor}
\usepackage{stmaryrd}
\usepackage[all,cmtip]{xy}
\usepackage{tikz}
\usetikzlibrary{positioning, calc, cd}
\usepackage[hidelinks]{hyperref}

\theoremstyle{plain}
\newtheorem{thm}{Theorem}[section]
\newtheorem{prop}[thm]{Proposition}
\newtheorem{cor}[thm]{Corollary}

\newtheorem{lem}[thm]{Lemma}

\theoremstyle{definition}
\newtheorem{defn}[thm]{Definition}
\newtheorem{ex}[thm]{Example}
\newtheorem{rem}[thm]{Remark}
\newtheorem{con}[thm]{Construction}

\newcommand{\Z}{\mathbb{Z}}

\newcommand{\Q}{\mathbb{Q}}
\newcommand{\Ss}{\mathbb{S}}

\newcommand{\Oc}{\mathcal{O}}
\newcommand{\Id}{\text{Id}}

\DeclareMathOperator{\Sp}{Sp}
\DeclareMathOperator{\Or}{Or}
\DeclareMathOperator{\op}{op}
\DeclareMathOperator{\res}{res}
\DeclareMathOperator{\tr}{tr}
\DeclareMathOperator{\colim}{colim}
\DeclareMathOperator{\Ind}{Ind}
\DeclareMathOperator{\Cat}{Cat}
\DeclareMathOperator{\rex}{rex}

\DeclareMathOperator{\incl}{incl}

\DeclareMathOperator{\Fun}{Fun}
\DeclareMathOperator{\orb}{orb}

\DeclareMathOperator{\Cl}{Cl}
\DeclareMathOperator{\SL}{SL}
\DeclareMathOperator{\GL}{GL}
\DeclareMathOperator{\Gen}{Gen}

\title{Chromatic Euler characteristics and duality for infinite groups}
\author{Gijs Heuts and Irakli Patchkoria}

\date{September 2026}

\begin{document}

\maketitle

\begin{abstract}
    We study a family of generalizations of the notion of Euler characteristic of discrete groups (or of orbifolds, depending on one's perspective) indexed on the natural numbers. For $n=0$, this is the classical orbifold Euler characteristic as studied by Wall and Serre, whereas for $n \geq 1$ and finite groups, this is the chromatic cardinality as studied by Ben-Moshe--Carmeli--Schlank--Yanovski. For general $n$, we show that our generalized Euler characteristic admits a natural interpretation in terms of the Morava $E$-theories. Our work involves showing that the generalized cohomology of infinite groups $G$ with finite universal space for proper actions $\underline{E}G$ has a good theory of duality, as expressed by a new duality functor on the category of proper $G$-equivariant spectra. In particular, for such groups we prove the vanishing of Klein's generalized Farrell--Tate cohomology with $T(n)$-local coefficients. We compute our generalized orbifold Euler characteristics in a large number of examples. This includes many mapping class groups, where the classical calculation is a result of Harer--Zagier, and many arithmetic groups, whose classical orbifold Euler characteristics were computed by Harder.
\end{abstract}

\setcounter{tocdepth}{1}
\tableofcontents

\section{Introduction}

The aim of this paper is to study generalizations of the notion of \emph{Euler characteristic} of a group (or of an orbifold) arising from the cohomology theories used in chromatic homotopy theory (e.g., $p$-completed complex $K$-theory) and compute them in a number of interesting examples. Our primary interest is in groups $G$ for which there exists a finite \emph{classifying space for proper actions} $\underline{E}G$. In more detail, this $\underline{E}G$ should be a finite $G$-CW complex where each cell has finite stabilizer and for each finite subgroup $H \leq G$, the fixed points $\underline{E}G^H$ are contractible. Important examples include arithmetic groups and mapping class groups. The cohomology of such a group has strong finiteness properties (see e.g., \cite[Theorem VIII.11.1]{Brown}). 

If $G$ admits a torsion-free subgroup $\Gamma \leq G$ of finite index, then following Wall and Serre \cite{Wall, Serre1} one can define its \emph{orbifold Euler characteristic} by 
\[
\chi_{\mathrm{orb}}(G) := \frac{\chi^{\mathbb{Q}}(\underline{E}G/\Gamma)}{[G:\Gamma]},
\]
where $\chi^{\mathbb{Q}}(\underline{E}G/\Gamma)$ denotes the usual (rational homology) Euler characteristic of the finite CW complex $\underline{E}G/\Gamma \simeq B\Gamma$. By \cite[Proposition IX.7.3]{Brown82} it can alternatively be computed by Quillen's formula
\[
\chi_{\mathrm{orb}}(G) = \sum_{G\sigma} (-1)^{d_\sigma} \frac{1}{|H_\sigma|},
\]
where the sum is over the $G$-orbits of cells, $d_\sigma$ is the dimension of such a cell $\sigma$, and $H_{\sigma}$ denotes the stabilizer of $\sigma$ in $X=\underline{E}G$. This formula is the origin of the term orbifold Euler characteristic: if $X$ is a manifold, then $X\sslash G$ is an orbifold and the formula above computes its Euler characteristic. Note that this is a rational number and is usually different from the classical Euler characteristic of $X/G$, which is always an integer. 


In this paper we will demonstrate how to associate to each Morava $E$-theory $E_n$ the  \emph{$E$-theory orbifold Euler characteristic} $\chi_{\mathrm{orb}}^{E_n}(G)$ (see Section \ref{sec: cardinalities}). This notion is a common generalization of the classical orbifold Euler characteristic of groups (being the case $n=0$) and the chromatic cardinalities of finite groups studied by Ben-Moshe--Carmeli--Schlank--Yanovski \cite{BenMosche, CSYambidexterity, CSYambheight, Yan}. The paper consists of two parts. In the first part we develop the theory of duality between generalized homology and cohomology for groups $G$ with finite $\underline{E}G$. This can be seen as an extension of Bieri--Eckmann duality to extraordinary cohomology theories; it also extends Klein's work on the dualizing complex of groups (at least in the case of discrete groups). In the second part we apply this general theory to construct our generalized orbifold Euler characteristics and compute them in many examples. We will now describe these two parts in some more detail and in opposite order.

\subsubsection*{Part II: Chromatic Euler characteristics}

Let $E_n$ be a Morava $E$-theory at height $n$ and a prime $p$. (For $n=1$, an example would be $p$-completed complex $K$-theory.) Suppose $G$ is a discrete group that admits a finite $\underline{E}G$. Then the $E$-theory orbifold Euler characteristic $\chi_{\mathrm{orb}}^{E_n}(G)$ we will define satisfies the following version of Quillen's formula (see Corollary \ref{cor: Quillen}), where, as before, the sum is over $G$-orbits of cells $\sigma$ of $\underline{E}G$ and $|BH|_{E_n}$ denotes the $E_n$-cardinality of $BH$ (with $H$ a finite group):
\[
\chi_{\mathrm{orb}}^{E_n}(G) = \sum_{G\sigma} (-1)^{d_\sigma}|BH_{\sigma}|_{E_n}.
\]
The $E_n$-cardinality $|BH|_{E_n}$ was introduced by Ben-Moshe--Carmeli--Schlank--Yanovski \cite{BenMosche, CSYambidexterity, CSYambheight, Yan} and admits an elementary formula in terms of properties of the group $H$, see Proposition \ref{prop: card for finite}. For now, let us state some of the most important properties of our generalized Euler characteristic:

\begin{thm} 
\label{thm:introEulerchar} Let $G$ be a discrete group that admits a finite $\underline{E}G$ and $p$ a prime. Then:
\begin{itemize}
\item[(1)] For $n=0$, the Euler characteristic $\chi_{\mathrm{orb}}^{E_n}(G)$ recovers the usual orbifold Euler characteristic $\chi_{\mathrm{orb}}(G)$. (Remark \ref{remark: n=0 special case})
\item[(2)] For $n > 0$, the Euler characteristic $\chi_{\mathrm{orb}}^{E_n}(G)$ is a $p$-local integer. (Corollary \ref{cor: Quillen})
\item[(3)] There is an equality
\[
\chi_{\mathrm{orb}}^{E_n}(G) = \sum_{[g_1, \ldots, g_n] \in G_{n,p}/G} \chi_{\mathrm{orb}}(C\langle g_1, \ldots, g_n\rangle),
\]
where $G_{n,p} = \mathrm{Hom}(\mathbb{Z}_p^n, G)$ is the set of $n$-tuples of commuting elements of $G$ all of which are of $p$-power order, on which $G$ acts by conjugation, and $C\langle g_1, \ldots, g_n\rangle$ denotes the centralizer of the (abelian) subgroup $\langle g_1, \ldots, g_n\rangle$ generated by $g_1,\dots,g_n$. (Theorem \ref{thm: centraliser})
\end{itemize}
\end{thm}

Combining items (2) and (3) yields $p$-integrality statements involving the rational numbers $\chi_{\mathrm{orb}}(C\langle g_1, \ldots, g_n\rangle)$ that reproduce interesting congruences from number theory. This was already observed by elementary methods in \cite{P24}, where it is shown amongst other things that these congruences recover Kummer's regularity criterion \cite{Kumregular}, Kummer's congruences for Bernoulli numbers, and further congruences due to Carlitz and Cohen \cite{Carl, Cohen}.

As we will see below, the invariant $\chi_{\mathrm{orb}}^{E_n}(G)$ is not in general an integer and hence differs from the Morava $K$-theory Euler characteristic $\chi^{K(n-1)}(BG) \in \Z$ studied in \cite{HKR} for finite groups and in \cite{LPS2024} for infinite groups. However, for groups which do not contain any elements of finite order other than a power of $p$, these invariants agree. This is explained in Remark \ref{rem: compare}. 

Theorem \ref{thm:introEulerchar} shows that one can compute $\chi_{\mathrm{orb}}^{E_n}(G)$ from the classical orbifold Euler characteristics if one has sufficient control over the set $G_{n,p}$ and the associated centralizers. We will use this strategy to compute the $E$-theory orbifold Euler characteristic in many examples. Harder \cite{Hard} computed the orbifold Euler characteristic of arithmetic groups in terms of special values of $\zeta$-functions. A basic example is $\chi_{\mathrm{orb}}(\mathrm{SL}_2(\mathbb{Z})) =-\frac{1}{12} = \zeta(-1)$ or, more generally, for $\mathcal{O}_K$ the ring of integers in a totally real field $K$, one has  $\chi_{\mathrm{orb}}(\mathrm{SL}_2(\mathcal{O}_K)) = \zeta_K(-1)$. For the symplectic groups Harder established the formula
\[
\chi_{\mathrm{orb}}(\mathrm{Sp}_{2g}(\mathbb{Z})) = \zeta(-1)\zeta(-3) \cdots \zeta(1-2g).
\]
Examples \ref{ex: symplectic}, \ref{ex: general linear}, \ref{ex: totally real} of this paper combine to give the following generalization of some of these formulas:

\begin{thm} \label{thm:chiarithmetic} Let $n \geq 0$ and let $E_n$ be a Morava $E$-theory at height $n$ and a prime $p$.
\begin{itemize}
\item[(1)] For $K$ a totally real field, $\Oc_K$ its ring of integers, and $p \geq 3$, we have  
\[\chi_{\orb}^{E_n} (\SL_2(\Oc_K))= \zeta_K(-1)+\sum_{(H)} \frac{{(\vert H_{(p)} \vert}^n-1)}{\vert H \vert},
\]
where $(H)$ runs over the conjugacy classes of the maximal finite subgroups. 

\item[(2)] For $p \geq 3$, we have
\[\chi_{\orb}^{E_n}(\Sp_{p-1}(\Z))=\zeta(-1) \zeta(-3) \cdots \zeta(2-p)+\frac{p^n-1}{p(p-1)} \cdot 2^{\frac{p-3}{2}}h_p^{-},\]
where $h_p^{-}$ is the class number of the cyclotomic field $\Q(\zeta_p)$ divided by the class number of the maximal real subfield.
 
\item[(3)] For $p \geq 5$, we have
\[\chi_{\orb}^{E_n}(\GL_{p-1}(\Z))=\chi_{\orb}^{E_n}(\SL_{p-1}(\Z))=0.\]
\end{itemize}
\end{thm}
\begin{rem}
In the context of item (1), we note that \cite[Theorem, p.198]{Hirz} and \cite{Prestel} provide explicit formulas for computing the numbers and types of conjugacy classes of maximal finite subgroups of $SL_2(\Oc_K)$ in the case where $K=\mathbb{Q}(\sqrt{d})$ and $d$ is a square-free positive integer.
\end{rem}

We refer the reader to Section \ref{sec: examples} for more examples, such as right-angled Coxeter groups. Still, our work leaves open the calculation of the $E$-theory orbifold Euler characteristics of $\mathrm{SL}_n(\mathbb{Z})$ and $\mathrm{Sp}_{2n}(\mathbb{Z})$ in general. It would be interesting to investigate these invariants further. 

Another famous calculation is the determination of the orbifold Euler characteristic of mapping class groups by Harer--Zagier \cite{HZ}, who give formulas for $\chi_{\orb}(\Gamma_g^s)$, where $\Gamma_g^s$ is the mapping class group of a genus $g$ surface with $s$ marked points. Equivalently, $\chi_{\orb}(\Gamma_g^s)$ can be thought of as the orbifold Euler characteristic of the moduli space $\mathcal{M}_g^s$ of genus $g$ surfaces with $s$ marked points.  We recall precise formulas in Section \ref{section: MCG and Harer-Zagier}. For now let us state the well-known case $s=1$:
\[\chi_{\orb}(\mathcal{M}_g^1)=\chi_{\orb}(\Gamma_g^1)=\zeta(1-2g)=-\frac{B_{2g}}{2g},\]
where $B_{2g}$ is the $2g$-th Bernoulli number. In Section \ref{section: MCG and Harer-Zagier} we extend this result to our generalized $E$-theory orbifold Euler characteristics. The following is our `chromatic' version of the Harer--Zagier theorem (see Theorem \ref{thm height n HZ}):

\begin{thm}
\label{thm:introHZ}
Let $\Sigma_g$ be a closed Riemann surface of genus $g$, equipped with a marked point $x_0$, and let $\Gamma^1_g$ be its mapping class group. Then the $E$-theory orbifold Euler characteristic of $\Gamma^1_g$ is given by
\[
\chi^{E_n}_{\orb}(\Gamma_g^1) = \chi_{\mathrm{orb}}(\Gamma_g^1) + \sum_{m \geq 1, h \geq 0} \frac{|\mathrm{Sur}(\mathbb{Z}^n, \mathbb{Z}/p^m)|}{p^m} \chi_{\mathrm{orb}}(\mathcal{H}_{h}^{1}(\mathbb{Z}/p^m,g)),
\]
where $\mathrm{Sur}(-,-)$ denotes the set of surjective homomorphisms and $\mathcal{H}_{h}^{1}(\mathbb{Z}/p^m,g)$ denotes the Hurwitz space of connected branched covers $\Sigma_g \to B$, with $B$ of genus $h$ and deck group cyclic of order $p^m$, so that the distinguished point $x_0$ is fixed by all deck transformations. (All of the numbers appearing in the formula can be made completely explicit and we will do so in Section \ref{section: MCG and Harer-Zagier}.)
\end{thm}


The computation uses the fact that any finite subgroup of $\Gamma_g^1$ is cyclic. It is more complicated to compute the $E$-theory orbifold Euler characteristic of the full mapping class group $\Gamma_g$. This is because the finite subgroup structure of $\Gamma_g$ is much more complicated than that of $\Gamma_g^1$. In fact any finite group occurs as a subgroup of $\Gamma_g$ for a large enough genus $g$ (see e.g., \cite[Theorem 7.12]{FarbMarg}). Harer and Zagier \cite{HZ} use the group extension
\[1 \to \pi_1(\Sigma_g) \to \Gamma^1_g \to \Gamma_g \to 1\]
to deduce that $\chi_{\orb}(\Gamma_g)=\frac{\zeta(1-2g)}{2-2g}$. There is no analogous argument to compute $\chi^{E_n}_{\orb}(\Gamma_g)$ for heights $n \geq 1$, because already for finite groups $\chi^{E_n}_{\orb}(G)=\vert BG \vert_{E_n}$ does not behave well with respect to group extensions (see \cite[Section 1.2]{CSYambheight}). Yet we do manage to obtain a general formula for the height one case $\chi^{E_1}_{\orb}(\Gamma_g)$ in Example \ref{ex: height one full mcg}. This uses that at $n=1$ we only need the classification of cyclic subgroups of $\Gamma_g$, which are well understood using \cite{HZ} and the Nielsen realization \cite{Nielsen}. Finally, when the prime $p$ is not too far away from the genus $g$, we obtain formulas for an arbitrary height $n$. For example, we will work out $\chi_{\orb}^{E_n}(\Gamma_{\frac{p-1}{2}})$ and $\chi_{\orb}^{E_n}(\Gamma_{\frac{(p-1)(p-2)}{2}})$ in Section \ref{section: MCG and Harer-Zagier}.


\subsubsection*{Part I: Duality and norm maps for infinite groups}

The key ingredient in our construction of the $E$-theory orbifold Euler characteristic of a group will be a good understanding of \emph{duality} for the generalized homology and cohomology of a group $G$ with finite $\underline{E}G$. Let $\mathcal{C}$ be a presentably symmetric monoidal stable $\infty$-category. (For us, this will typically be some Bousfield localization of the $\infty$-category of spectra.) If $X$ is an object of $\mathcal{C}$ with $G$-action, then homology and cohomology of $G$ with coefficients in $X$ are related through the \emph{norm map}
\[
\mathrm{Nm}_{BG}\colon (X \otimes \mathbf{D}_{BG})_{hG} \to X^{hG},
\]
where $\mathbf{D}_{BG}$ denotes an object of $\mathrm{Fun}(BG,\mathcal{C})$ called the \emph{dualizing complex}. Such norm maps feature in the classical Bieri--Eckmann theory of duality groups \cite{BieEck} and have been generalized, in increasing levels of abstraction, by Klein \cite{Klein,Klein02}, Nikolaus--Scholze \cite{NS}, and Cnossen \cite{Cno23}. We will give a review of the norm and the dualizing complex in Section \ref{sec: norm and Tate}, in which we also include some material on the naturality of the norm that we were unable to find in the literature.

A question of particular interest is when the norm map is invertible. Equivalently, this is the question of when Farrell--Tate cohomology of $G$ (defined to be the cofiber of the norm) with coefficients in $\mathcal{C}$ vanishes. A simple example is where $G$ is finite and $\mathcal{C}$ is the derived category $\mathcal{D}(\mathbb{Q})$ of the rational numbers, or any commutative ring in which the order of $G$ is invertible. Many more examples are given by the rational duality groups of Bieri--Eckmann (cf. Example \ref{ex:rationaldualitygroup}). However, there exist more sophisticated settings where Farrell--Tate cohomology vanishes. The ambidexterity results of Hopkins--Lurie \cite{HL13} and Carmeli--Schlank--Yanovski \cite{CSYambidexterity} imply that the norm map is always an equivalence if the group $G$ is finite (in which case the dualizing complex $\mathbf{D}_{BG}$ turns out to be just the unit) and the $\infty$-category $\mathcal{C}$ is that of $K(n)$-local or $T(n)$-local spectra. We extend this result to a larger class of groups as follows (see Theorem \ref{thm:Tatevanishing} and Proposition \ref{prop:generalizednorm}):

\begin{thm}
\label{thm:Tatevanishingintro}
Let $G$ be a group with finite stable $\underline{E}G$ and take $\mathcal{C}$ to be the $\infty$-category of $T(n)$- or $K(n)$-local spectra. Then for any $X \in \mathrm{Fun}(BG,\mathcal{C})$, the norm map $(X \otimes \mathbf{D}_{BG})_{hG} \to X^{hG}$ is an equivalence in $\mathcal{C}$.
\end{thm}

Our next step is a closer look at the dualizing complex $\mathbf{D}_{BG}$. In its usual definition, this is an object of $\mathrm{Fun}(BG,\Sp)$, i.e., a $G$-spectrum in the naivest sense of the word. However, we observe that it arises naturally as the underlying object of a \emph{genuine proper $G$-spectrum} that we construct in Section \ref{sec: genuine dualizing complex}. In fact, we construct a novel \emph{duality functor}
\[
D\colon \Sp_{\mathrm{prop}}^G \to (\Sp_{\mathrm{prop}}^G)^{\mathrm{op}}
\]
on the $\infty$-category of proper $G$-spectra. (The reader should be warned that this is \emph{not} duality with respect to the symmetric monoidal structure.) Then $\mathbf{D}_{BG}$ turns out to be the object underlying the dual of the unit $D(\mathbb{S}_G)$. One important fact about this duality is that any finite proper $G$-CW complex $X$ defines a \emph{fundamental class} 
\[
{[X]}\colon D(\mathbb{S}_G) \to \mathbb{S}_G.
\]
Putting the ingredients together, we can now define the $E$-theory $G$-equivariant Euler characteristic of $X$ to be the following composite (see Definition \ref{def: E-theory euler char}):
\[
E \xrightarrow{\Delta} E^{hG} \xrightarrow{\mathrm{Nm}_{BG}^{-1}} (E \otimes D(\mathbb{S}_G))_{hG} \xrightarrow{(E \otimes {[X]})_{hG}} E_{hG} \xrightarrow{\nabla} E.
\]
When $G$ is a finite group and $X$ is the unit, this simplifies to the composite
\[
E \xrightarrow{\Delta} E^{hG} \xrightarrow{\mathrm{Nm}_{BG}^{-1}} E_{hG} \xrightarrow{\nabla} E
\]
and we retrieve the chromatic cardinality of $BG$ in the sense of Ben-Moshe--Carmeli--Schlank--Yanovski \cite{BMCSY}.
 

\subsubsection*{Acknowledgments}
GH acknowledges the support of the ERC through a Starting Grant 
(no. 950048) and the NWO through a Vidi grant (no. 233.093) and the XL grant \emph{Symmetry on the interface of topology and higher algebra}. IP acknowledges the support of the EPSRC through a NI grant EP/X038424/1 \emph{Classifying spaces, proper actions and stable homotopy theory}. Both authors thank the Isaac Newton Institute for Mathematical Sciences, where part of this work was carried out, for its hospitality and excellent working conditions. This work was also supported by the EPSRC grant no EP/R014604/1.

We thank Shachar Carmeli, Sil Linskens, Akhil Mathew and Oscar Randal-Williams for helpful conversations. We also thank Christian Kremer for pointing out that our duality functor $D$ should be related to \cite{HKK}, see Appendix \ref{app:parametrized}.



\part{Duality and norm maps for infinite groups}

\section{The classifying space for proper actions} 

In this section we review and collect well-known material on the \emph{universal space for proper $G$-actions} $\underline{E}G$, which we will need in later sections. Along the way we take the opportunity to point out some important examples. 

For a discrete group $G$, there are many flavors of $G$-equivariant homotopy theory. The most basic is the $\infty$-category $\mathrm{Fun}(BG,\mathcal{S})$ of spaces with a $G$-action. On the other hand, there is the $\infty$-category of \emph{genuine} $G$-spaces 
\[
\mathcal{S}^G := \mathrm{Fun}((\mathrm{Or}^G)^{\mathrm{op}}, \mathcal{S}),
\]
where $\mathrm{Or}^G$ denotes the orbit category of $G$. By Elmendorf's theorem \cite{Elm}, this $\infty$-category is equivalent to the localization of the category of $G$-CW complexes with respect to maps inducing homotopy equivalences on $H$-fixed points for every subgroup $H \leq G$. 

In this paper we are mostly concerned with the intermediate notion of \emph{proper} $G$-equivariant homotopy theory, which is only sensitive to the data of fixed points with respect to \emph{finite} subgroups $H \leq G$. Write $\mathrm{Or}^G_{\mathrm{fin}}$ for the full subcategory of $\mathrm{Or}^G$ on orbits with finite isotropy, i.e., on the objects $G/H$ with $H$ a finite subgroup of $G$. We will write $\mathcal{S}^G_{\mathrm{prop}}$ for the $\infty$-category of proper $G$-spaces, which can be described in the following three equivalent ways:

\begin{prop}
\label{prop:properGspaces}
    Let $G$ be a discrete group. The following three $\infty$-categories are equivalent:
    \begin{itemize}
        \item[(1)] The localization of the category of $G$-CW complexes where each cell has a finite isotropy group with respect to the maps inducing homotopy equivalences on $H$-fixed points for all subgroups $H \leq G$.
        \item[(2)] The localization of the category of $G$-CW complexes with respect to the maps inducing homotopy equivalences on $H$-fixed points for all finite subgroups $H \leq G$.
        \item[(3)] The presheaf $\infty$-category $\mathrm{PSh}(\mathrm{Or}^G_{\mathrm{fin}}) = \mathrm{Fun}((\mathrm{Or}^G_{\mathrm{fin}})^{\mathrm{op}},\mathcal{S})$.
    \end{itemize}
\end{prop}
\begin{proof} This is essentially contained in \cite[Section 7]{DL98}. See also \cite[Remark 3.1.12]{DHLPS}. 
\end{proof}

For a presheaf $X \in \mathrm{PSh}(\mathrm{Or}^G_{\mathrm{fin}})$ we will, as is customary, use the notation $X^H$ for the evaluation of $X$ at the orbit $G/H$ and refer to it as the $H$-fixed points of $X$. Note that the object $G/e \in \mathrm{Or}^G_{\mathrm{fin}}$ has automorphism group $G$. In particular, evaluation at $G/e$ determines a functor
\[
\mathcal{S}^G_{\mathrm{prop}} \to \mathrm{Fun}(BG,\mathcal{S})\colon X \mapsto X^e.
\]
We will sometimes refer to $X^e$ as the \emph{underlying space} of $X$.

By definition, the space $\underline{E}G$ is a $G$-CW complex such that the fixed points $\underline{E}G^H$ are contractible when $H$ is finite and empty otherwise. It is naturally thought of as an object of $\mathcal{S}^G_{\mathrm{prop}}$ through perspective (1) of Proposition \ref{prop:properGspaces}. In terms of perspectives (2) and (3) it corresponds simply to the one-point space. Said more invariantly, $\underline{E}G$ is the terminal object of $\mathcal{S}^G_{\mathrm{prop}}$. More details and examples of $\underline{E}G$ can be found for example in \cite{LuckSurvey}. We record the following description of it (see e.g., \cite[Lemma 2.2]{LO01}):

\begin{lem}
\label{lem:underbarEGcolim}
The unique map of proper $G$-spaces
\[
\colim_{G/H \in \mathrm{Or}^G_{\mathrm{fin}}} G/H \to \underline{E}G
\]
is an equivalence.
\end{lem}
\begin{proof}
It is standard (and easy to show) that in any presheaf category, the terminal object can be written as the colimit of the Yoneda embedding.
\end{proof}

\begin{cor}
\label{cor:BGcolim}
The canonical map
\[
\colim_{G/H \in \mathrm{Or}^G_{\mathrm{fin}}} BH \to BG
\]
is an equivalence of spaces.
\end{cor}
\begin{proof}
This follows by passing to underlying spaces of the map of Lemma \ref{lem:underbarEGcolim} and applying the homotopy orbits functor $(-)_{hG}$. We use that the underlying space of $\underline{E}G$ is contractible and that $(G/H)_{hG} \simeq BH$.
\end{proof}

The colimit formula for $\underline{E}G$ of Lemma \ref{lem:underbarEGcolim} will be useful to us in the paper, but because of its universal nature it is also inefficient. Often there exist much smaller cell decompositions of $\underline{E}G$; having such small decompositions is a crucial tool in studying the (co)homology of $G$. We will say that $\underline{E}G$ is finite if as a $G$-CW complex it can be built from finitely many cells of the form $G/H \times D^n$, with $H$ ranging over finite subgroups of $G$. To be more precise, one says $G$ has \emph{finite $\underline{E}G$} if any of the equivalent conditions of the following lemma is satisfied:


\begin{lem} \label{lem: finite underline}
Let $G$ be a discrete group. Then the following are equivalent:
\begin{itemize}
\item[(1)] In terms of perspective (1) of Proposition \ref{prop:properGspaces}, $\underline{E}G$ is a retract of a finite $G$-CW complex with finite isotropy groups.
\item[(2)] The terminal object of $\mathrm{PSh}(\mathrm{Or}^G_{\mathrm{fin}})$ can be written as a retract of a finite colimit of representables $G/H$.
\item[(3)] The terminal object of $\mathcal{S}^{G}_{\mathrm{prop}}$ is compact.
\end{itemize}
\end{lem}

\begin{rem} \label{rem:finitevsfinitelydom} Classically, the condition (1) in Lemma \ref{lem: finite underline} is referred to by saying that $G$ admits a \emph{finitely dominated} $\underline{E}G$. This is potentially slightly weaker than having a finite $\underline{E}G$; however, we are not aware of any example of a group which admits a finitely dominated $\underline{E}G$ but does not admit a finite $G$-CW complex model for $\underline{E}G$ (see e.g., \cite[Problem 7.13]{Luckfinite}). For all of our theoretical results, the equivalent conditions of Lemma \ref{lem: finite underline} will suffice. When we do actually need a finite model of $\underline{E}G$ for calculation, we will announce this explicitly.
\end{rem}

\begin{ex} Most interesting examples of groups in geometric group theory and number theory admit a finite $\underline{E}G$. We list here examples that are relevant in this paper. For any finite group, the one-point space with the trivial action is a model for a finite $\underline{E}G$. For an amalgamated product of finite groups, the Bass-Serre tree gives a finite $\underline{E}G$ \cite{Serre} (Example \ref{ex: amalgamated1}). In particular, since $SL_2(\Z)$ is isomorphic to $C_4 \ast_{C_2} C_6$, one sees that $SL_2(\Z)$ admits a finite $\underline{E}G$. Right-angled Coxeter groups admit a finite $\underline{E}G$ given by the Davis complex \cite[Chapter 7]{Dav}. The Borel--Serre compactification is a model of a finite $\underline{E}G$ for arithmetic groups \cite{Ji2007}. Finally, Harer's spine of Teichm\"uller space is a model for a finite $\underline{E}G$ for mapping class groups \cite{Har, Broughton90, JiWolpert2010, Mislin2010}.  
\end{ex}

Let us record the following simple observation explicitly for later use:

\begin{lem}
\label{lem:finiteEGandBG}
If $G$ has finite $\underline{E}G$, then the classifying space $BG$ is a retract of a finite colimit of classifying spaces $BH$ of finite groups $H$.
\end{lem}
\begin{proof}
Apply $(-)_{hG}$ to a (retract of a) finite decomposition of $\underline{E}G$ into cells with finite stabilizers $H$.
\end{proof}

\begin{ex} \label{ex: amalgamated1}
Let $H \leq K,L$ be finite groups and  $G=K \ast_H L$ the amalgamated product. By Bass--Serre theory \cite{Serre}, a model of $\underline{E}G$ is given by the pushout square of $G$-spaces
\[\xymatrix{G/H \times S^0 \ar[r] \ar[d] & G/K \coprod G/L \ar[d] \\ G/H \times D^1 \ar[r] &\underline{E}G,}\]
where the top horizontal map is the disjoint union of restrictions. In particular, the classifying space $BG$ is given by the pushout:
\[\xymatrix{BH \times S^0 \ar[r] \ar[d] & BK \coprod BL \ar[d] \\ BH \times D^1 \ar[r] &BG.}\]

\end{ex}

To conclude this section we will give yet another description of the $\infty$-category $\mathcal{S}^G_{\mathrm{prop}}$ of proper $G$-spaces and introduce the $\infty$-category $\mathrm{Sp}^G_{\mathrm{prop}}$ of proper $G$-spectra. Recall that for a group $H$ we write $\mathcal{S}^H$ for the $\infty$-category of genuine $H$-spaces.

\begin{prop}
\label{prop:S^Gaslimit}
    The restriction functors induce an equivalence of $\infty$-categories
    \[
    \mathcal{S}^{G}_{\mathrm{prop}} \xrightarrow{\simeq} \lim_{G/H \in (\mathrm{Or}^G_{\mathrm{fin}})^{\mathrm{op}}} \mathcal{S}^H.
    \]
\end{prop}
\begin{proof}
Generally, for an $\infty$-category $\mathcal{C}$, there is a canonical equivalence
\[
\mathcal{C} \simeq \mathrm{colim}_{x \in \mathcal{C}} \mathcal{C}_{/x}
\]
in $\mathrm{Cat}_\infty$. Indeed, this colimit can be computed by forming the Grothendieck construction of $x \mapsto \mathcal{C}_{/x}$ and localizing it at the coCartesian arrows. This Grothendieck construction is the arrow category $\mathrm{Ar}(\mathcal{C})$ of $\mathcal{C}$ and the coCartesian arrows are the maps going to equivalences under the source projection $\mathrm{Ar}(\mathcal{C}) \to \mathcal{C}$. After performing this localization, the source projection becomes an equivalence: indeed, it has a fully faithful left adjoint by assigning to an object its identity arrow and the counit becomes an equivalence in the stated localization. We now apply this abstract colimit decomposition to the case $\mathcal{C} = \mathrm{Or}^G_{\mathrm{fin}}$ and observe that the slice over an orbit $G/H$ is canonically equivalent to the orbit category $\mathrm{Or}^H$. Taking presheaves on these orbit categories gives the result.
\end{proof}

The category of proper $G$-spectra $\Sp^G_{\mathrm{prop}}$ was defined in \cite{DHLPS}. In this paper we use the $\infty$-categorical description of this category by Linskens--Nardin--Pol. By \cite[Theorem 12.11]{LNP}, the restriction functors induce an equivalence
\[\Sp^G_{\mathrm{prop}} \xrightarrow{\simeq} \lim_{G/H \in (\Or^G_{\mathrm{fin}})^{\op}} \Sp^H,\]
where $\Sp^H$ is the category of genuine $H$-spectra. Since this is a limit of presentable $\infty$-categories along right adjoint functors, \cite[Corollary 5.5.3.4 and Theorem 5.5.3.18]{LurieHTT} show that we may also describe $\Sp_{\mathrm{prop}}^G$ as the colimit in $\Pr^L$ along the induction functors (being left adjoint to restriction) as follows:
\[\colim_{\Or^G_{\mathrm{fin}}} \Sp^H \xrightarrow{\simeq} \Sp^G_{\mathrm{prop}}. \]
In particular, we have for each finite subgroup $H \leq G$ a left adjoint functor $\Sp^H \to \Sp^G_{\mathrm{prop}}$ that we denote by $\mathrm{ind}_H^G$ and refer to as induction.

The first description in terms of restrictions has the advantage that it equips $\Sp^G_{\mathrm{prop}}$ with a symmetric monoidal structure so that for any finite subgroup $H \leq G$, the restriction $\Sp^G_{\mathrm{prop}} \to \Sp^H$ has a symmetric monoidal structure. As in the case of proper $G$-spaces, restricting to the free orbit $G/e$ defines a `forgetful functor'
\[
\mathrm{Sp}^G_{\mathrm{prop}} \to \mathrm{Fun}(BG,\Sp)\colon X \mapsto X^e.
\]
We record the following description of the underlying spectrum for later use, where we write $\mathbb{S}[G]$ for the induced proper $G$-spectrum $\mathrm{ind}_{e}^G \mathbb{S}$:

\begin{lem}
\label{lem:underlyingspectrum}
For $X \in \mathrm{Sp}^G_{\mathrm{prop}}$, there is a natural equivalence
\[
X^e \simeq \mathrm{map}^G(\mathbb{S}[G], X)
\]
in $\mathrm{Fun}(BG,\Sp)$, where the $\mathrm{map}^G$ on the right-hand side denotes the mapping spectrum associated with the stable $\infty$-category $\mathrm{Sp}^G_{\mathrm{prop}}$ and the (left) $G$-action arises from the evident (right) $G$-action on $\mathbb{S}[G]$.
\end{lem}
\begin{proof}
We have natural equivalences
\[
X^e \simeq \mathrm{map}(\mathbb{S},\mathrm{res}^G_e(X)) \simeq \mathrm{map}^G(\mathbb{S}[G],X).
\]
\end{proof}

Recall that for any finite group $H$ there is an equivariant suspension spectrum functor
\[
\Sigma^\infty_+\colon \mathcal{S}^H \to \Sp^H,
\]
which admits a symmetric monoidal structure. These functors are compatible with restriction along subgroup inclusions $K \leq H$. Therefore we can assemble them into a symmetric monoidal functor
\[
 \lim_{G/H \in (\mathrm{Or}^G_{\mathrm{fin}})^{\mathrm{op}}} \mathcal{S}^H \to  \lim_{G/H \in (\mathrm{Or}^G_{\mathrm{fin}})^{\mathrm{op}}} \Sp^H
\]
or equivalently a symmetric monoidal functor
\[
\Sigma^\infty_+\colon \mathcal{S}^G_{\mathrm{prop}} \to \mathrm{Sp}^G_{\mathrm{prop}}.
\]
Thus the tensor unit $\mathbb{S}_G \in \mathrm{Sp}^G_{\mathrm{prop}}$ is equivalent to $\Sigma^\infty_+ \underline{E}G$. In particular, if $G$ has finite $\underline{E}G$, then the unit of $\mathrm{Sp}^G_{\mathrm{prop}}$ is compact. (The latter condition is often expressed as $G$ having finite \emph{stable} $\underline{E}G$ \cite{BDP}.) We record the following consequence of Lemma \ref{lem:underbarEGcolim} for future reference:

\begin{lem}
\label{lem:SGcolim}
There is an equivalence of proper $G$-spectra
\[
\colim_{G/H \in \mathrm{Or}^G_{\mathrm{fin}}} \Sigma^\infty_+ G/H \xrightarrow{\simeq} \mathbb{S}_G.
\]
\end{lem}


\section{The norm map and Tate vanishing for infinite groups} \label{sec: norm and Tate}

Let $X$ be a space and $\mathcal{C}$ a presentably symmetric monoidal $\infty$-category. Consider a $\mathcal{C}$-valued local system $F \in \mathrm{Fun}(X,\mathcal{C})$. In this section we will review the construction of the \emph{norm map}
\[
\mathrm{Nm}_X\colon \mathrm{colim}_X (F \otimes \mathbf{D}_X) \to \mathrm{lim}_X F.
\]
Here $\mathbf{D}_X$ denotes the \emph{dualizing complex} of $X$ (with coefficients in $\mathcal{C}$), which we will describe below. For most of this paper we will be interested in the case $X = BG$ for some discrete group $G$. In this case $F$ encodes an object of $\mathcal{C}$ with $G$-action; the domain of the norm map may be interpreted as the homology of $G$ with coefficients in $F \otimes \mathbf{D}_X$, whereas the codomain can be thought of as the cohomology of $G$ with coefficients in $F$. 

When $\mathcal{C}$ is the $\infty$-category of spectra and $X = BG$, the dualizing complex and the associated norm map were constructed and studied by Klein \cite{Klein}; the case of general $X$ is treated by Nikolaus--Scholze \cite{NS}. The case of general $\mathcal{C}$ is developed by Cnossen \cite{Cno23} under the name \emph{twisted ambidexterity}. The first part of this section is a review of those constructions. After that we discuss a certain naturality of the norm map, which appears to be new.

Recall that an $\infty$-category $\mathcal{C}$ is called \emph{1-semiadditive} if it is pointed, finite products exist and are equivalent to finite coproducts, and for any finite group $G$ the homotopy orbits and fixed point functors $(-)_{hG}$ and $(-)^{hG}$ exist and are equivalent. (This last condition can equivalently be phrased as saying that  $\mathcal{C}$ has \emph{Tate vanishing} for finite groups.) The main examples of 1-semiadditive $\infty$-categories for us will be the derived $\infty$-category of the rationals $\mathcal{D}(\mathbb{Q})$ and the $\infty$-categories $\mathrm{Sp}_{K(n)}$ and $\mathrm{Sp}_{T(n)}$ of $K(n)$- and $T(n)$-local spectra, respectively. Over the rational numbers Tate vanishing is simple to prove, for $\mathrm{Sp}_{K(n)}$ it is a result of Greenlees--Sadofsky \cite{GS96} and Hovey--Sadofsky \cite{HS96}, and for $\mathrm{Sp}_{T(n)}$ a result of Kuhn \cite{kuhntate}.

The main purpose of Section \ref{sec:dualitymaps} is to show the following:

\begin{thm}
\label{thm:Tatevanishing}
Let $G$ be a group with finite $\underline{E}G$ and suppose that $\mathcal{C}$ is stable and 1-semiadditive. Then the norm map 
\[
\mathrm{Nm}_{BG}\colon (F \otimes \mathbf{D}_{BG})_{hG} \to F^{hG}
\]
is an equivalence for any $F \in \mathrm{Fun}(BG,\mathcal{C})$.
\end{thm}

\begin{rem} More generally, given the norm map 
\[
\mathrm{Nm}_{BG}\colon (F \otimes \mathbf{D}_{BG})_{hG} \to F^{hG}
\]
in a pointed $\infty$-category $\mathcal{C}$, the cofiber is called \emph{the generalized Farrell--Tate construction} or simply \emph{the Tate construction} and is denoted by $F^{tG}$. If $G$ is finite, and $\mathcal{C}=\mathcal{D}(\mathbb{Z})$, and $F$ is concentrated in degree zero, then the homotopy groups of $F^{tG}$ recover the classical Tate cohomology groups. More generally if $\mathcal{C}=\Sp$, and $G$ is finite, then $F^{tG}$ coincides with the generalized Tate construction of Greenlees--May \cite{GM95} (also featuring prominently in \cite{NS}). For a general $G$, the spectrum $F^{tG}$ was defined by Klein \cite{Klein02, Klein}, who showed that this construction generalizes the classical Farrell--Tate cohomology for virtually torsion-free groups. For a stable and 1-semiadditive $\infty$-category $\mathcal{C}$, Theorem \ref{thm:Tatevanishing} is equivalent to saying that the generalized Farrell--Tate construction $F^{tG}$ vanishes. \end{rem}

To prove the theorem it will be convenient to work in the relative setting of a map $f\colon X \to Y$ between spaces. Such a map induces a symmetric monoidal pullback functor
\[
f^*\colon \mathrm{Fun}(Y,\mathcal{C}) \to \mathrm{Fun}(X,\mathcal{C})
\]
that admits both a left adjoint $f_!$ (taking colimits along fibers of $f$) and a right adjoint $f_*$ (taking limits along fibers of $f$). Consider the following commutative diagram:
\[
\begin{tikzcd}
X \ar{dr}{\Delta}\ar[bend left]{drr}{1_X}\ar[bend right]{ddr}[swap]{1_X} && \\
& X \times_Y X \ar{r}{\pi_1}\ar{d}{\pi_2} & X \ar{d}{f} \\
& X \ar{r}{f} & Y.
\end{tikzcd}
\]

\begin{defn}
\label{def:dualizingcomplex}
We define the \emph{dualizing complex} $\mathbf{D}_f$ of $f$ to be $(\pi_1)_*\Delta_! \mathbf{1}_X$, where $\mathbf{1}_X$ denotes the monoidal unit of $\mathrm{Fun}(X,\mathcal{C})$. When $Y = *$, we denote the dualizing complex $\mathbf{D}_f$ by $\mathbf{D}_X$.
\end{defn}

\begin{ex}
\label{ex:DBG}
Suppose $\mathcal{C} = \Sp$, let $G$ be a discrete group, and $X = BG$. Then $\Delta_! \mathbf{1}_{BG} \simeq \mathbb{S}[G]$, interpreted as an object of $\mathrm{Fun}(BG \times BG,\Sp)$ by letting $G$ act on itself on both sides, and 
\[
\mathbf{D}_{BG} = (\pi_1)_*\Delta_! \mathbf{1}_{BG} \simeq \mathbb{S}[G]^{hG},
\]
with the fixed points being taken with respect to the first copy of $G$. This is Klein's description of the dualizing complex in \cite{Klein}. 

If the discrete group $G$ is finite, then $\mathbf{D}_{BG}$ is equivalent to $\mathbb{S}$ with the trivial action. This follows since the Tate construction for finite groups vanishes on induced spectra from the trivial subgroup \cite[Lemma I.3.8]{NS} and hence the norm map
\[\mathbb{S} \simeq \mathbb{S}[G]_{hG} \xrightarrow{\mathrm{Nm}_{BG}} \mathbb{S}[G]^{hG} \]
is an equivalence in $\Sp$. By the naturality of the norm, this is in fact an equivalence in $\Fun(BG, \Sp)$. 
\end{ex}

\begin{rem}
Consider the case $X = BG$ with $G$ a discrete group. In Section \ref{sec: genuine dualizing complex} we will upgrade the dualizing spectrum $\mathbf{D}_X$ to a genuine proper $G$-spectrum, rather than just an object of $\mathrm{Fun}(BG,\mathrm{Sp})$.
\end{rem}

To construct the norm map, we first recall the following properties of the functors involved:
\begin{itemize}
\item[(1)] The $\infty$-category $\mathrm{Fun}(X,\mathcal{C})$ is presentably symmetric monoidal, so it is in particular \emph{closed}: for a given local system $A$ the tensor product functor $A \otimes -$ admits a right adjoint, which we denote $\mathbf{hom}_{X}(A,-)$.
\item[(2)] For $A \in \mathrm{Fun}(X,\mathcal{C})$ and $B \in \mathrm{Fun}(Y,\mathcal{C})$ the natural map
\[
f_!(A \otimes f^*B) \to f_!A \otimes B
\]
is an equivalence. This is the \emph{projection formula}.
\item[(3)] An equivalent formulation of (2) is that for local systems $B,C \in \mathrm{Fun}(Y,\mathcal{C})$ the natural map
\[
f^*\mathbf{hom}_Y(B,C) \to \mathbf{hom}_X(f^*B,f^*C)
\]
is an equivalence.
\item[(4)] For a pullback square of spaces
\[
\begin{tikzcd}
X' \ar{r}{g}\ar{d}{q} & Y' \ar{d}{p} \\
X \ar{r}{f} & Y,
\end{tikzcd}
\]
the associated pullback/pushforward functors satisfy \emph{base change}, meaning that the evident natural transformation $g_!q^* \to p^*f_!$ is an equivalence. Passing to right  adjoints, the natural transformation $f^*p_* \to q_*g^*$ is an equivalence as well.
\end{itemize}

\begin{rem}
All of these facts are standard. Let us just briefly remark on their proofs. For (2), consider any $y \in Y$ and write $X_y$ for the fiber of $f$ at $y$. After evaluating at $y$ the map under consideration becomes the equivalence
\[
\mathrm{colim}_{X_y}(A \otimes B(y))|_{X_y} \to (\mathrm{colim}_{X_y}A|_{X_y}) \otimes B(y).
\]
To see the equivalence between items (2) and (3), note that for any $A \in \mathrm{Fun}(X,\mathcal{C})$ we have a natural commutative square
\[
\begin{tikzcd}
\mathrm{Map}(A, f^*\mathbf{hom}_Y(B,C)) \ar{r}\ar{d}{\simeq} & \mathrm{Map}(A,\mathbf{hom}_X(f^*B,f^*C)) \ar{d}{\simeq} \\
\mathrm{Map}(f_!A \otimes B, C) \ar{r} & \mathrm{Map}(f_!(A \otimes f^*B), C),
\end{tikzcd}
\]
where the vertical arrows arise from adjunction and the horizontal maps are the ones arising from items (2) and (3). For the base change equivalence, note that for $y \in Y'$ and $A \in \mathrm{Fun}(X,\mathcal{C})$, the evaluation of the map under consideration at $y$ is the map
\[
\mathrm{colim}_{w \in X'_y}A(q(w)) \to \mathrm{colim}_{z \in X_{p(y)}} A(z). 
\]
Since the square is a pullback, the induced map of fibers $X'_y \to X_{p(y)}$ is an equivalence and hence the same is true of the map above.
\end{rem}

\begin{con}
\label{con:norm}
Let $F \in \mathrm{Fun}(X,\mathcal{C})$ be a local system. Consider the composition
\[
\mathbf{1}_X \to \mathbf{hom}_X(F,F) \xrightarrow{\simeq} \Delta^*\mathbf{hom}_{X \times_Y X}(\pi_1^* F, \pi_2^*F) 
\]
where the first is the unit and the second is the equivalence associated with $\Delta$ (item (3) on our list above). Through the adjunctions $(\Delta_!,\Delta^*)$ and tensor-hom we find a morphism
\[
\pi_1^*F \otimes \Delta_!\mathbf{1}_X \to \pi_2^* F.
\]
We now form the commutative diagram
\begin{equation}
\label{diag:defnorm}
\begin{tikzcd}
F \otimes (\pi_1)_*\Delta_!\mathbf{1}_X \ar{d}\ar{r}{\widetilde{\mathrm{Nm}}_f} & f^*f_* F\\
(\pi_1)_*(\pi_1^*F \otimes \Delta_!\mathbf{1}_X) \ar{r} & (\pi_1)_*\pi_2^*F, \ar{u}{\simeq}
\end{tikzcd}
\end{equation}
where the bottom arrow is $(\pi_1)_*$ applied to the map we just constructed, the left vertical arrow composes the unit $F \to (\pi_1)_*\pi_1^*F$ and the lax monoidal structure of $(\pi_1)_*$, and the right vertical arrow is the base change equivalence associated with the pullback square defining $X \times_Y X$. The top horizontal arrow is then defined to be the composite of the other three.

\end{con}

The following general version of the norm was introduced by Cnossen, \cite[Definition 3.3]{Cno23}:

\begin{defn} \label{def: relative norm}
The \emph{norm map}
    \[
\mathrm{Nm}_f\colon f_!(F \otimes \mathbf{D}_f) \to f_* F
\]
is the adjoint of the morphism $\widetilde{\mathrm{Nm}}_f$ constructed above.
\end{defn}


Further on we will need a certain naturality property of the norm map, to which we devote the remainder of this section. Consider a map of spaces $i\colon A \to X$ and let us write $f\colon X \to *$ for the unique map to the point. We also write $g = fi$ for the unique map from $A$ to the point. (There is also a relative version of what we do, replacing the point by a space $Y$. The ideas are identical to the absolute case we treat here, but the notation would be more involved.) For $F \in \mathrm{Fun}(X,\mathcal{C})$ we will construct a natural commutative diagram
\begin{equation}
\label{diag:mixednorm}
\begin{tikzcd}
f_!(F \otimes \mathbf{D}_X) \ar{r}{\mathrm{Nm}_X}\ar{d} & f_* F \ar{d} \\
f_!(F \otimes \mathbf{D}_{(X,A)}) \ar{r}{\mathrm{Nm}_{(X,A)}} & f_*i_*i^*F \ar[equal]{d} \\
g_!(i^*F \otimes \mathbf{D}_A) \ar{u}\ar{r}{\mathrm{Nm}_A} & g_*i^*F.
\end{tikzcd}
\end{equation}
The top map is the norm map for $X$ we have constructed, the bottom map is the norm map of $A$, applied to the restricted local system $i^*F$. The middle map, which we construct below, is a \emph{mixed norm map} and $\mathbf{D}_{(X,A)}$ the \emph{mixed dualizing complex} associated with the map $i\colon A \to X$. Of course it would be better to make the map $i$ part of the notation, but we wish to avoid confusion with the dualizing complex of the map $i$, which is different from this mixed dualizing complex $\mathbf{D}_{(X,A)}$. 

In many cases of practical interest the lower left map in the diagram is an equivalence, thus identifying the mixed norm $\mathrm{Nm}_{(X,A)}$ with the norm map $\mathrm{Nm}_A$ of $A$. (See Lemma \ref{lem:mixednormvsAnorm} for a particular case relevant to us.) In that case, Diagram (\ref{diag:mixednorm}) simplifies to a direct comparison between $\mathrm{Nm}_X$ (applied to $F$) and $\mathrm{Nm}_A$ (applied to $i^*F$).


To construct the mixed norm, we will make use of the following commutative diagrams of spaces:
\[
\begin{tikzcd}
A \ar[equal]{r}\ar{d}{\Delta_A} & A \ar{r}{i}\ar{d}{(i,1)} & X \ar{d}{\Delta_X} \ar{d} && X \times A \ar{r}{1 \times i}\ar{d}{q} & X \times X \ar{d}{\pi_2} \\
A \times A \ar{d}{\pi_1}\ar{r}{i \times 1} & X \times A \ar{d}{p} \ar{r}{1 \times i} & X \times X \ar{d}{\pi_1} && A \ar{r}{i} & X. \\
A \ar{r}{i} & X \ar[equal]{r} & X, &&
\end{tikzcd}
\]
For convenience we have introduced the notations $p$ and $q$ for the projections of $X \times A$ onto the first and second coordinate respectively. Observe that in the left diagram, the bottom left and top right squares are pullbacks in the $\infty$-category of spaces. The square on the right is a pullback as well.

\begin{defn}
The \emph{mixed dualizing complex} of $i\colon A \to X$ is the object
\[
\mathbf{D}_{(X,A)} := p_*(i,1)_!\mathbf{1}_A \in \mathrm{Fun}(X,\mathcal{C}).
\]
\end{defn}

\begin{rem} The mixed dualizing complex $\mathbf{D}_{(X,A)}$ recovers the dualizing complex $\mathbf{D}_{X}$ when $A=X$ and $i=\mathrm{id}_X$. This illustrates the difference between the relative and mixed dualizing complexes, since the relative dualizing complex of the identity of $X$ is the monoidal unit $\mathbf{1}_X$. On the other hand, the relative dualizing complex of the map $X \to *$ recovers $\mathbf{D}_X$.
\end{rem}

\begin{con}
\label{con:mixednorm}
To start constructing Diagram (\ref{diag:mixednorm}), let us first explain how to define the mixed norm map
\[
f_!(F \otimes p_*(i,1)_!\mathbf{1}_A) \xrightarrow{\mathrm{Nm}_{(X,A)}} f_*i_*i^*F.
\]
First of all, we have a map
\[
(i,1)_!\mathbf{1}_A \to \mathbf{hom}_{X \times A}(p^*F, q^*i^*F)  
\]
that is adjoint to the unit map
\[
\mathbf{1}_A \to \mathbf{hom}_A(i^*F,i^*F) \xrightarrow{\simeq} (i,1)^*\mathbf{hom}_{X \times A}(p^*F,q^*i^*F).
\]
Here we have used the identifications $p \circ (i,1) = i$ and $q \circ (i,1) = \mathrm{id}_A$. By further adjunction, we find a map
\[
p^*F \otimes (i,1)_!\mathbf{1}_A \to q^*i^*F.
\]
We now use this map to construct the bottom left vertical arrow in the following commutative diagram, where the top left vertical arrow uses the unit $F \to p_*p^*F$ and the lax monoidal structure of $p_*$:
\[
\begin{tikzcd}
F \otimes p_*(i,1)_!\mathbf{1}_A \ar{d}\ar{r}{\widetilde{\mathrm{Nm}}_{(X,A)}} & f^*f_*i_*i^* F\\
p_*(p^*F \otimes (i,1)_!\mathbf{1}_A) \ar{d} & (\pi_1)_*\pi_2^*i_*i^* F \ar{u}{\simeq} \\
p_*q^*i^* F \ar[equal]{r} & (\pi_1)_*(1 \times i)_*q^*i^* F. \ar{u}{\simeq}
\end{tikzcd}
\]
We now define the mixed norm map to be the adjoint of the top horizontal map in the diagram.
\end{con}

\begin{con}
\label{con:mixednorm1}
Let us now explain the top commutative square in Diagram (\ref{diag:mixednorm}). We have already defined the relevant norm maps and the right vertical map arises from the unit $F \to i_*i^*F$. For the left vertical map we use the morphism
\[
\mathbf{D}_X = (\pi_1)_*(\Delta_X)_!\mathbf{1}_X \to (\pi_1)_*(1 \times i)_*(1 \times i)^*(\Delta_X)_!\mathbf{1}_X \simeq p_*(i,1)_!i^*\mathbf{1}_X \simeq \mathbf{D}_{(X,A)}
\]
arising from the unit $\mathrm{id} \to (1 \times i)_*(1 \times i)^*$ and base change. Arguing that the square commutes is a very formal matter; unraveling the definitions, this boils down to the observation that applying the functor $(1 \times i)^*$ to the map
\[
(\Delta_X)_!\mathbf{1}_X \to \mathrm{hom}_{X \times X}(\pi_1^* F, \pi_2^* F)
\]
used to define $\mathrm{Nm}_X$ yields the map
\[
(i,1)_!\mathbf{1}_A \to \mathrm{hom}_{X \times A}(p^* F, q^*i^* F)
\]
used to define $\mathrm{Nm}_{(X,A)}$. 

\end{con}

\begin{con}
\label{con:mixednorm2}
We will now explain how to construct the bottom square in Diagram (\ref{diag:mixednorm}). To construct the left vertical map in that square, it suffices to specify a map
\[
i_!(i^*F \otimes \mathbf{D}_A) \to F \otimes \mathbf{D}_{(X,A)}.
\]
The projection formula allows us to rewrite the domain as $F \otimes i_!\mathbf{D}_A$, so it will suffice to provide a map $i_!\mathbf{D}_A \to \mathbf{D}_{(X,A)}$. For this we use the adjoint of the morphism
\[
\mathbf{D}_A = (\pi_1)_*(\Delta_A)_!\mathbf{1}_A \to (\pi_1)_*(i\times 1)^*(i \times 1)_!(\Delta_A)_!\mathbf{1}_A \simeq i^*p_*(i,1)_!\mathbf{1}_A = i^*\mathbf{D}_{(X,A)}
\]
combining the unit $\mathrm{id} \to (i\times 1)^*(i \times 1)_!$ and base-change $(\pi_1)_*(i \times 1)^* \simeq i^*p_*$.

Arguing that the bottom square of Diagram (\ref{diag:mixednorm}) commutes is again a formal exercise. The crucial point this time is that there is a commutative square
\[
\begin{tikzcd}
(\Delta_A)_!\mathbf{1}_A \ar{r}\ar{d} &  \mathbf{hom}_{A \times A}(\pi_1^*i^*F, \pi_2^*i^*F) \ar{d}{\simeq} \\
(i\times 1)^*(i,1)_!\mathbf{1}_A \ar{r} & (i\times 1)^*\mathbf{hom}_{X \times A}(p^*F, q^*i^*F)
\end{tikzcd}
\]
in which the top horizontal map is the one used to define $\mathrm{Nm}_A$, the bottom horizontal map is $(i
\times 1)^*$ applied to the one used to define $\mathrm{Nm}_{(X,A)}$, the left-hand vertical map combines the unit $\mathrm{id} \to (i\times 1)^*(i \times 1)_!$ and the identification $(i \times 1)_!(\Delta_A)_! = (i,1)_!$, and the right-hand vertical map is the (adjoint of) the projection formula.
\end{con}

As we alluded to above, in many cases one can identify the mixed norm map $\mathrm{Nm}_{(X,A)}$ with the norm map of $A$. We record the following special instance for later use:

\begin{lem}
\label{lem:mixednormvsAnorm}
Let $\mathcal{C}$ be the $\infty$-category of spectra, $G$ a group, and $H \leq G$ a finite subgroup. Set $X = BG$ and $A = BH$. Then the map
\[
g_!(i^*F \otimes \mathbf{D}_A) \to f_!(F \otimes \mathbf{D}_{(X,A)})
\]
appearing in Diagram (\ref{diag:mixednorm}) is an equivalence. In particular, in this case the mixed norm map $\mathrm{Nm}_{(X,A)}$ for $F$ is equivalent to the norm map $\mathrm{Nm}_A$ for $i^*F$.
\end{lem}
\begin{proof}
It suffices to check that the map $i_!\mathbf{D}_A \to \mathbf{D}_{(X,A)}$ of Construction \ref{con:mixednorm2} is an equivalence. Since $H$ is finite, the associated dualizing complex $\mathbf{D}_{BH} = \mathbf{D}_A$ is the unit $\mathbb{S}$ (see Example \ref{ex:DBG}). Also, we have $(i,1)_!\mathbb{S} \simeq \mathbb{S}[G]$, where $G$ is to be interpreted as a $G \times H$-set in the usual way. For the mixed dualizing complex we then find
\[
\mathbf{D}_{(X,A)} = p_*(i,1)_!\mathbb{S} \simeq \mathbb{S}[G]^{hH}.
\]
By chasing through the definition, the map $\mathbf{D}_A \to i^*\mathbf{D}_{(X,A)}$ from Construction \ref{con:mixednorm2} can be identified with the map
\[\xymatrix{\mathbb{S} \simeq \mathbb{S}[H]_{hH} \ar[rr]^{\mathrm{Nm}_{BH}}_{\simeq} & & \mathbb{S}[H]^{hH} \ar[r]^-{\incl} & i^*(\mathbb{S}[G]^{hH})}\]
in $\Fun(BH, \Sp)$, where the last map is induced from the subgroup inclusion. Since the norm is a natural transformation, we have a commutative diagram in $\Fun(BH, \Sp)$:
\[\xymatrix{\mathbb{S} & \mathbb{S}[H]_{hH} \ar[l]_-{\simeq}\ar[r]^-{\incl} \ar[d]^{\mathrm{Nm_{BH}}} & i^*(\mathbb{S}[G]_{hH}) \ar[d]^{i^*(\mathrm{Nm_{BH}})}  \ar[r]^\simeq & \Sigma^{\infty} G/H_+ & \\  & \mathbb{S}[H]^{hH} \ar[r]^-{\incl} & i^*(\mathbb{S}[G]^{hH}). & }\]
Passing to the adjoint diagram we obtain a commutative square in $\Fun(BG, \Sp)$:
\[\xymatrix{i_!\mathbb{S} \ar[d]^{i_!(\mathrm{Nm_{BH}})} \ar[r]^-{\simeq} & \mathbb{S}[G]_{hH} \ar[d]^{\mathrm{Nm_{BH}}} \\ i_!(\mathbb{S}[H]^{hH}) \ar[r] & \mathbb{S}[G]^{hH}.}\]
\noindent This commutative diagram identifies the map $i_!\mathbf{D}_A \to \mathbf{D}_{(X,A)}$ of Construction \ref{con:mixednorm2} with the norm map
\[
i_!\mathbb{S} \simeq \Sigma^\infty_+ G/H \simeq (\mathbb{S}[G])_{hH} \xrightarrow{\mathrm{Nm}_{BH}} \mathbb{S}[G]^{hH}
\]
associated to $BH$. This map is an equivalence because $\mathbb{S}[G]$ is, as an $H$-object, induced from the trivial group.
\end{proof}

\begin{rem}
Let $\mathcal{C}$ be a presentably symmetric monoidal $\infty$-category. In the proof above we used the following simple general observation: if $E$ is an object of $\mathcal{C}$ with a $G \times H$-action, then the norm map $\mathrm{Nm}_{BH} \colon (E \otimes \mathbf{D}_{BH})_{hH} \to E^{hH}$ lives in $\Fun(BG,\mathcal{C})$. This is because $\mathrm{Nm}_{BH}$ is the underlying non-equivariant map of the relative norm map of the projection $BG \times BH \to BG$. 
\end{rem}

\section{Duality maps and duality spaces}
\label{sec:dualitymaps}

Recall that a group $\Gamma$ satisfying the $FP$-finiteness property is called a \emph{Bieri--Eckmann duality group} if there is a dualizing $\Gamma$-module $D_{\Gamma}$ and a number $n$ such that for any $\mathbb{Z}[\Gamma]$-module $M$, there is an isomorphism
\[H_{n-*}(\Gamma; D_{\Gamma} \otimes_{\mathbb{Z}} M) \cong H^*(\Gamma; M). \]
There are many equivalent characterizations of this condition \cite[Section VIII.10]{Brown}. In particular it implies that the number $n$ is the cohomological dimension of $\Gamma$, the module $D_{\Gamma}$ is isomorphic to $H^n(\Gamma; \Z[\Gamma])$, and 
\[H^i(\Gamma; \Z[\Gamma])=0\; \text{for all}\; i \neq n.\]
In this case it is known that $D_{\Gamma}$ is a torsion-free (possibly infinitely generated) abelian group.

Any duality group is torsion-free. For groups which contain torsion, there is a generalization of the notion of duality. We recall the following definition from \cite[Section VIII.11]{Brown}: \begin{defn} A group $G$ is called a \emph{virtual duality group} if it has a finite index subgroup which is a Bieri--Eckmann duality group \cite{BieEck}. 
\end{defn}

For $G$ a virtual duality group, $H^n(G; \Z[G])$ is torsion-free and $H^i(G; \Z[G])=0$ for all $i \neq n$. The module $H^n(G; \Z[G])$ serves as the dualizing module for the finite index duality subgroup. The number $n$ is the virtual cohomological dimension of $G$ \cite[Section VIII.11]{Brown}. Examples of virtual duality groups include arithmetic groups \cite{BorelSerre} and mapping class groups \cite{Har}.

It turns out that if a group $\Gamma$ is a duality group in the classical sense of Bieri--Eckmann, then the dualizing complex $\mathbf{D}_{B\Gamma}$ in the derived category $\mathcal{C} = \mathcal{D}(\mathbb{Z})$ is concentrated in a single cohomological degree, which equals the cohomological dimension of $\Gamma$. Furthermore, the norm map $\mathrm{Nm}_{B\Gamma}$ is an equivalence and induces the above isomorphism between cohomology and homology with coefficients in the dualizing module.

One of the main themes of this paper is that upon changing the coefficients from $\mathcal{D}(\mathbb{Z})$ to $\infty$-categories such as $\mathrm{Sp}_{T(n)}$ or $\mathrm{Sp}_{K(n)}$, there are many more groups that behave like `duality groups' in the above sense. We extend the classical terminology as follows:


\begin{defn}
A map $f\colon X \to Y$ is a {$\mathcal{C}$-duality map} if the norm map $\mathrm{Nm}_f$ is an equivalence for every local system $F \in \mathrm{Fun}(X,\mathcal{C})$. A space $X$ is a {$\mathcal{C}$-duality space} if $X \to *$ is a $\mathcal{C}$-duality map.
\end{defn}

\begin{rem}
The notion of a $\mathcal{C}$-duality map is referred to as a \emph{twisted $\mathcal{C}$-ambidextrous map} in \cite{Cno23}.
\end{rem}

The purpose of this section is to describe several closure properties of the class of duality spaces (and duality maps), which together will cover all of the examples in this paper. In particular we will prove Theorem \ref{thm:Tatevanishing} at the end of this section.

\begin{ex}
\label{ex:1semiadditive}
The $\infty$-category $\mathcal{C}$ is 1-semiadditive if and only if every 1-truncated $\pi$-finite space is a $\mathcal{C}$-duality space. Note that for a \emph{finite} group $G$, the dualizing complex $\mathbf{D}_{BG}$ is simply the unit, and the norm map reduces to a natural transformation $(-)_{hG} \to (-)^{hG}$, which is the same as the norm considered by Hopkins--Lurie \cite{HL13}.
\end{ex}




We will now review some of the basic properties of duality maps. The following is proved by Cnossen \cite[Proposition 3.13]{Cno23}, but can also be deduced directly from Proposition \ref{prop:f*Clinear} below.

\begin{lem}
\label{lem:Cdualityclosure}
The class of $\mathcal{C}$-duality maps is closed under composition, pullback, and retracts.
\end{lem}

The following is easily deduced from this, see also \cite[Proposition 3.13]{Cno23}.

\begin{lem}
\label{lem:dualityfiberwise}
A map $f$ is a $\mathcal{C}$-duality map if and only if each of its fibers is a $\mathcal{C}$-duality space.
\end{lem}

\begin{cor} \label{cor: extensions}
The class of $\mathcal{C}$-duality spaces is closed under extensions. To be precise, if $f\colon E \to B$ is a map so that $B$ is a $\mathcal{C}$-duality space and every fiber of $f$ is a $\mathcal{C}$-duality space, then $E$ is a $\mathcal{C}$-duality space.
\end{cor}
\begin{proof}
By Lemma \ref{lem:dualityfiberwise} the map $f$ is a $\mathcal{C}$-duality map. Lemma \ref{lem:Cdualityclosure} then implies that the composition of maps $E \to B \to *$ is a $\mathcal{C}$-duality map, i.e., that $E$ is a $\mathcal{C}$-duality space.
\end{proof}
 
The following example shows that in the $1$-semiadditive case the distinction between duality groups and virtual duality groups disappears:

\begin{ex}
Suppose that $\mathcal{C}$ is 1-semiadditive. Let $G$ be a group with a finite index normal subgroup $\Gamma$ that is a $\mathcal{C}$-duality group, meaning $B\Gamma$ is a $\mathcal{C}$-duality space. Then $G$ itself is a $\mathcal{C}$-duality group. This follows from Corollary \ref{cor: extensions} by considering the fiber sequence
\[
B\Gamma \to BG \to BG/\Gamma
\]
and noting that the last term is a $\mathcal{C}$-duality space by the assumption that $\mathcal{C}$ is 1-semiadditive (see Example \ref{ex:1semiadditive}).
\end{ex}

A useful way to think of the norm map is as the universal approximation of $f_*$ by a $\mathcal{C}$-linear functor (cf. \cite[Proposition 3.7]{Cno23}). To be precise, let us say that $f_*$ is $\mathcal{C}$-linear if it preserves colimits and for any $M \in \mathcal{C}$ and $F \in \mathrm{Fun}(X,\mathcal{C})$, the natural map
\[
M \otimes f_*F \to f_*(M \otimes F)
\]
is an equivalence. (This map is part of the lax monoidal structure of $f_*$, which it has by virtue of being right adjoint to the monoidal functor $f^*$.)

\begin{prop}
    \label{prop:f*Clinear}
    A space $X$ is a $\mathcal{C}$-duality space if and only if the functor $f_*$ is $\mathcal{C}$-linear (with $f$ the map $X \to *$).
\end{prop}
\begin{proof}
Suppose that $X$ is a $\mathcal{C}$-duality space. To see that $f_*$ is $\mathcal{C}$-linear, it suffices to observe that $f_!$ is $\mathcal{C}$-linear. For the converse, suppose that $f_*$ is a $\mathcal{C}$-linear functor. To check that the norm map is an equivalence, it now suffices to do so on `induced' objects of the kind $F = x_!(\mathbf{1})$, where $x\colon * \to X$ is the inclusion of a point of $X$ and $\mathbf{1} \in \mathcal{C}$ is the unit. (Indeed, varying $x$ over the components of $X$, such $F$ generate $\mathrm{Fun}(X,\mathcal{C})$ as a $\mathcal{C}$-linear $\infty$-category.) In this case, the norm map may be identified with a map
\[
\mathrm{Nm}_X\colon x^*\mathbf{D}_X \to f_*x_!\mathbf{1},
\]
where we have used the projection formula to rewrite the domain. To unravel this further, consider the pullback squares
\[
\begin{tikzcd}
* \ar{r}{x} \ar{d}{x} & X \ar{r}{f} \ar{d}{x \times 1_X} & * \ar{d}{x} \\
X \ar{r}{\Delta} & X \times X \ar{r}{\pi_1} & X.
\end{tikzcd}
\]
Applying base change twice we find equivalences
\[
x^*\mathbf{D}_X = x^*(\pi_1)_*\Delta_! \mathbf{1}_X \simeq f_*(x \times 1_X)^*\Delta_! \mathbf{1}_X \simeq f_*x_!x^*\mathbf{1}_X = f_*x_!\mathbf{1},
\]
and unpacking the definitions one sees that this composite may be identified with $\mathrm{Nm}_X$.
\end{proof}

The characterization of duality maps of Proposition \ref{prop:f*Clinear} yields the following very useful result (this result has already appeared in \cite[Lemma 4.21]{CCRY} with a different proof): 

\begin{prop}
\label{prop:Cdualityfinitecolim}
Suppose that $\mathcal{C}$ is stable. Then a finite colimit of $\mathcal{C}$-duality spaces is a $\mathcal{C}$-duality space. In particular, any finite space is a $\mathcal{C}$-duality space.
\end{prop}
\begin{proof}
The final claim about a finite space follows from the preceding by noting that the one-point space is a $\mathcal{C}$-duality space for any $\mathcal{C}$. For the first claim, suppose that $X = \mathrm{colim}_{i 
\in I} X_i$ for some finite diagram $I$, where each $X_i$ is a $\mathcal{C}$-duality space. Write $g_i$ for the map $X_i \to X$ and $f$ for the map $X \to *$. Then $\mathrm{Fun}(X,\mathcal{C}) \simeq \mathrm{lim}_{i \in I^{\mathrm{op}}} \mathrm{Fun}(X_i,\mathcal{C})$, from which it follows that for $F$ a local system on $X$ the natural map 
\[
F \to \mathrm{lim}_{i \in I^{\mathrm{op}}} (g_i)_*g_i^* F
\]
is an equivalence. Furthermore, the hypothesis that each $X_i$ is a $\mathcal{C}$-duality space means that each of the functors $f_*(g_i)_*$ is $\mathcal{C}$-linear. The pullback functors $g_i^*$ are evidently $\mathcal{C}$-linear. Since $\mathcal{C}$ is stable, finite limits commute with colimits and tensor products, so that any finite limit of $\mathcal{C}$-linear functors is $\mathcal{C}$-linear. We conclude that $f_* \simeq \mathrm{lim}_{i \in I^{\mathrm{op}}} f_*(g_i)_*g_i^*$ is $\mathcal{C}$-linear. Proposition \ref{prop:f*Clinear} implies that $X$ is a $\mathcal{C}$-duality space.
\end{proof}

Our first main result on Tate vanishing now follows straightforwardly:

\begin{proof}[Proof of Theorem \ref{thm:Tatevanishing}]
Suppose $G$ has finite $\underline{E}G$. Then $BG$ is a (retract of a) finite colimit of classifying spaces of finite groups (see Lemma \ref{lem:finiteEGandBG}). By the assumption that $\mathcal{C}$ is 1-semiadditive, those classifying spaces are $\mathcal{C}$-duality spaces (Example \ref{ex:1semiadditive}). Now Proposition \ref{prop:Cdualityfinitecolim} implies that $BG$ is a $\mathcal{C}$-duality space, as desired.
\end{proof}

\begin{ex}
\label{ex:rationaldualitygroup}
Let $G$ be a virtual duality group with finite virtual cohomological dimension $n$. Then in fact $G$ is a rational duality group \cite[Theorem 9.2]{Bieri} (see also \cite[Section X.3]{Brown}), i.e., for any $\Q[G]$-module $V$, there is an isomorphism
\[H^*(G; V) \cong H_{n-*}(G; D_{G}^{\Q} \otimes_{\mathbb{\Q}} V),\]
where $D_{G}^{\Q}\cong H^{n}(G; \Q[G])$. 

This result can be thought of as a special case of Theorem \ref{thm:Tatevanishing} when $G$ admits a finite $\underline{E}G$. Indeed the derived category $\mathcal{D}(\mathbb{Q})$ is $1$-semiadditive and Theorem \ref{thm:Tatevanishing} implies that the norm map
\begin{align}\label{eq: rational norm}\mathrm{Nm}_{BG} \colon (V \otimes \mathbf{D}_{BG})_{hG} \to V^{hG}\end{align}
is an equivalence. To obtain the rational duality statement for (co)homology, one observes that since $G$ is a virtual duality group, we have $D_{G}^{\Q} \simeq \Sigma^{n} \Q \otimes \mathbf{D}_{BG}$ as $\Q[G]$-modules (cf. \cite[Proposition VIII.11.3]{Brown}). In fact Theorem \ref{thm:Tatevanishing} says that for the rational vanishing of the Tate construction it is not necessary to assume that $G$ is a virtual duality group. In other words, the norm map (\ref{eq: rational norm})
is an equivalence for any group $G$ with finite $\underline{E}G$ and any $V$ in $\mathcal{D}(\mathbb{Q}[G])$. We believe this is familiar to experts, since it is well known that the Farrell--Tate cohomology of $VFP$ groups is torsion \cite[Proposition 8]{Far78}. 
\end{ex}

\begin{ex} The previous example can be generalized to higher chromatic heights. As mentioned before it follows from \cite{HS96,GS96} that $\mathrm{Sp}_{K(n)}$ is $1$-semiadditive and \cite{kuhntate} implies that $\mathrm{Sp}_{T(n)}$ is $1$-semiadditive. Theorem \ref{thm:Tatevanishing} specializes to the following: let $G$ be a group with finite $\underline{E}G$ and $X$ a $K(n)$-local (or $T(n)$-local) spectrum with a $G$-action. Then the norm map
\[\mathrm{Nm}_{BG} \colon (X \otimes \mathbf{D}_{BG})_{hG} \to X^{hG}\]
is a $K(n)$-local (resp. $T(n)$-local) equivalence, i.e., we have
\[L_{K(n)} X^{tG}=0 \quad (\text{resp.} \quad L_{T(n)} X^{tG}=0).\]

\end{ex}

\begin{rem} For us, the most relevant
example of a $K(n)$-local spectrum with $G$-action will be $E_n$, meaning Morava $E$-theory at height $n$, equipped with the trivial $G$-action. For $G$ a finite group, Hopkins--Kuhn--Ravenel character theory \cite{HKR} and Stapleton's transchromatic character theory \cite{Stap} can be used to study the Tate construction $E_n^{tG}$. We recall some details from \cite{HKR}. Let $L$ denote the Hopkins--Kuhn--Ravenel character ring. This is a faithfully flat rational graded algebra over 
\[E_n^{*} \cong W(\mathbb{F}_{p^n})\llbracket u_1, \dots, u_{n-1} \rrbracket[u^{\pm 1}],\]
where $ W(\mathbb{F}_{p^n})$ is the ring of Witt vectors of $\mathbb{F}_{p^n}$, the variables $u_1, \dots \dots, u_{n-1}$ have degree $0$ and $u$ has degree $-2$. Further, given a discrete group $G$, let $G_{n,p}$ denote the set of $n$-tuples $(g_1, \dots, g_n)$ of commuting elements ($g_i g_j=g_jg_i$ for all $i$ and $j$), each of which has $p$-power order. The group $G$ acts on $G_{n,p}$ by conjugation in each variable. Hopkins--Kuhn--Ravenel construct a character isomorphism
\[ \chi_{n,p} \ : \ E_n^*(B G) \otimes_{E_n^*} L \xrightarrow{\cong}
    \Fun(G_{n,p}/G, L),\]
generalizing classical character theory at height $n=1$. By definition of $\chi_{n,p}$ (see \cite[Section 6.3]{HKR}) and $1$-semiadditivity of $\mathcal{D}(\mathbb{Q})$, we see that after tensoring with $L$, the norm map
\[N_G \colon (E_n)_{-*}(B G) \otimes_{E_n^*} L \cong L \to E_n^{*}(B G) \otimes_{E_n^*} L\]
becomes split injective. The splitting is given by the map $\Fun(G_{n,p}/G, L) \to L$ which evaluates at the class $[1, \dots 1]$ and divides by the order of $G$. This implies that the character map induces an isomorphism
\[\pi_{-*}(E_n^{tG}) \otimes_{E_n^{*}} L \cong \Fun(G_{n,p}-\{(1, \dots, 1)\}/G, L).\]
The character theory of \cite{HKR} was generalized in \cite{LPS2024} from finite groups to discrete groups with finite $\underline{E}G$. Using this one can deduce that if $G$ is a discrete group with finite $\underline{E}G$, then
\[\pi_{-*}(E_n^{tG}) \otimes_{E_n^{*}} L \cong  \prod_{[g_1,\dots,g_n] \in G_{n,p}-\{(1, \dots, 1)\}/G} H^*( B C\langle g_1, \dots, g_n \rangle; L ),\]
where 
$C\langle g_1, \dots, g_n \rangle$ denotes the centralizer of the (abelian) subgroup generated by the elements $g_1, \dots, g_n $. See \cite{AndrewPatchkoria} for explicit computations of the homotopy groups of the Farrell--Tate construction $E_n^{tG}$ at height $n=1$. 
\end{rem}

\section{The dualizing complex as a proper $G$-spectrum} \label{sec: genuine dualizing complex} 

Let $G$ be a discrete group. Throughout this section we take $\mathcal{C}$ to be the $\infty$-category $\Sp$ of spectra, although one can work at the level of a general presentably symmetric monoidal stable $\mathcal{C}$ with little change. The highlight of this section is the construction of a `duality functor' $D$ on the $\infty$-category $\Sp^G_{\mathrm{prop}}$ of proper $G$-spectra (not to be confused with duality with respect to the symmetric monoidal structure). We will show that the dual of the unit $D(\mathbb{S}_G)$ provides a lift of the dualizing complex $\mathbf{D}_{BG} \in \mathrm{Fun}(BG,\mathrm{Sp})$ of Definition \ref{def:dualizingcomplex} to a genuine proper $G$-spectrum (see Proposition \ref{prop: dualizing module}). In particular, a cell decomposition of $\Sigma^\infty_+\underline{E}G$ gives an explicit description of $\mathbf{D}_{BG}$, see Lemma \ref{lem:DStransfer}. 

To start, recall that there is an equivalence
\[
\colim_{G/H \in \mathrm{Or}^G_{\mathrm{fin}}} \Sp^{H} \xrightarrow{\simeq} \Sp^G_{\mathrm{prop}},
\]
where the colimit is computed along induction functors in the $\infty$-category $\Pr^L$. Each of the $\infty$-categories $\Sp^{H}$ is compactly generated, i.e., it is of the form $\mathrm{Ind}((\Sp^H)^{\omega})$ with $(\Sp^H)^{\omega}$ the full subcategory of compact $H$-spectra. These compact objects are precisely the retracts of finite $H$-spectra or, in other words, $(\Sp^H)^{\omega}$ is the thick subcategory generated by orbits $\Sigma^\infty_+ H/K$ with $K$ ranging over subgroups of $H$.

Let $\mathrm{Cat}_\infty^{\rex}$ denote the $\infty$-category of small $\infty$-categories admitting finite colimits and functors between them preserving finite colimits. Then the functor $\Ind \colon \Cat_{\infty}^{\rex} \to \Pr^L$ preserves colimits (see \cite[Propositions 5.5.7.6, 5.5.7.10]{LurieHTT}). Hence we find (see also \cite[Corollary 1.3.11]{DHLPS} and \cite[Section 3]{BDP} for analogous statements in terms of triangulated categories):

\begin{lem}
\label{lem:SpGcompactgen}
The functor
\[
\mathrm{Ind}\bigl(\colim_{G/H \in \mathrm{Or}^G_{\mathrm{fin}}} (\Sp^{H})^{\omega}\bigr) \rightarrow \Sp^G_{\mathrm{prop}} 
\]
is an equivalence, with the colimit $\colim_{G/H \in \mathrm{Or}^G_{\mathrm{fin}}} (\Sp^{H})^{\omega}$ computed in $\mathrm{Cat}_\infty^{\rex}$. In particular, $\Sp^G_{\mathrm{prop}}$ is compactly generated and the thick subcategory of compact objects is generated by the orbits $\Sigma^\infty_+ G/H$ with $H$ ranging over finite subgroups of $G$.
\end{lem}
\begin{proof}
The first part of the lemma is immediate from the discussion above. It follows that the compact objects of $\Sp^G_{\mathrm{prop}}$ are retracts of induced objects of the form $\mathrm{ind}_{H}^G X$ with $H \leq G$ finite and $X \in \Sp^H$ a compact object. Such $X$ are generated by objects of the form $\Sigma^{\infty}_+ H/K$. Now observe $\mathrm{ind}_{H}^G \Sigma^{\infty}_+ H/K \simeq \Sigma^{\infty}_+ G/K$, so that we find the desired family of generators for the compact objects of $\Sp^G_{\mathrm{prop}}$.
\end{proof}

For a finite group $H$, the compact objects coincide with the dualizable objects in the symmetric monoidal $\infty$-category of genuine $H$-spectra $\Sp^H$. In particular, Spanier--Whitehead duality gives an equivalence of $\infty$-categories
\[
D\colon (\Sp^H)^{\omega} \xrightarrow{\simeq} ((\Sp^H)^{\omega})^{\op}.
\]
This equivalence is natural in $H$, in the following sense. Consider the diagram
\[
\mathrm{res}\colon (\mathrm{Or}^G_{\mathrm{fin}})^{\mathrm{op}} \to \mathrm{Cat}_\infty^{\rex} \colon G/H \mapsto (\Sp^{H})^{\omega}
\]
built from the restriction functors; e.g., an inclusion of subgroups $H \leq K$, giving a map of orbits $G/H \to G/K$, is sent to $\mathrm{res}^K_H\colon (\Sp^{K})^{\omega} \to (\Sp^{H})^{\omega}$. The restriction functors admit a natural symmetric monoidal structure, so that Spanier--Whitehead duality can be upgraded to an equivalence of diagrams $D\colon \mathrm{res} \to \mathrm{res}^{\mathrm{op}}$. The restriction functors admit left adjoints (namely induction) and right adjoints (coinduction), and we will temporarily write
\[
\mathrm{ind}, \, \mathrm{coind}\colon \mathrm{Or}^G_{\mathrm{fin}} \to \mathrm{Cat}_\infty^{\rex} \colon G/H \mapsto (\Sp^{H})^{\omega}
\]
for the resulting two diagrams. Passing to adjoints, the equivalence of diagrams $D$ we just established therefore also gives an equivalence (denoted by the same letter)
\[
D\colon \mathrm{ind} \to \mathrm{coind}^{\mathrm{op}}.
\]
The Wirthm\"uller isomorphism \cite{Wir} shows that induction and coinduction along inclusions of finite groups agree. In fact, this statement can be made natural, in the following sense:

\begin{lem}
\label{lem:Wirthmuller}
The Wirthm\"uller isomorphisms $\mathrm{ind}_H^K \simeq \mathrm{coind}_H^K$ can be assembled to an equivalence of functors
\[
\mathrm{ind} \simeq \mathrm{coind}\colon \mathrm{Or}^G_{\mathrm{fin}} \to \mathrm{Cat}_\infty^{\rex}.
\]
\end{lem} 

The proof of this lemma can be found in the appendix. We can now complete the promised construction of the duality functor $D$ for proper $G$-spectra:

\begin{con}
\label{con:duality}
Lemma \ref{lem:Wirthmuller} and the equivalence $D\colon \mathrm{ind} \to \mathrm{coind}^{\mathrm{op}}$ constructed above together give an equivalence of diagrams
\[
\mathrm{ind} \simeq \mathrm{ind}^{\mathrm{op}}\colon \mathrm{Or}^G_{\mathrm{fin}} \to \mathrm{Cat}_\infty^{\rex}.
\]
Passing to the colimit over $\mathrm{Or}^G_{\mathrm{fin}}$ then gives an equivalence of $\infty$-categories
\[
D\colon (\Sp^G_{\mathrm{prop}})^{\omega} \xrightarrow{\simeq} ((\Sp^G_{\mathrm{prop}})^{\omega})^{\op}.
\]
The composite functor $(\Sp^G_{\mathrm{prop}})^{\omega} \xrightarrow{D}((\Sp^G_{\mathrm{prop}})^{\omega})^{\op} \xrightarrow{\incl} (\Sp^G_{\mathrm{prop}})^{\op}$ is exact and therefore induces a colimit-preserving functor 
\[
D \colon \Sp^G_{\mathrm{prop}} \to (\Sp^G_{\mathrm{prop}})^{\op}
\]
by \cite[Proposition 5.3.5.10]{LurieHTT}.
\end{con}

Let us now record some basic properties of this construction.

\begin{prop}
    \label{prop:dualitySpG}
    The duality functor $D$ satisfies the following:
    \begin{itemize}
        \item[(1)] For a finite subgroup $H \leq G$, there is a preferred equivalence $D(\Sigma^\infty_+ G/H) \simeq \Sigma^\infty_+ G/H$.
        \item[(2)] For $X \in \Sp^G_{\mathrm{prop}}$, there is a natural map $\eta\colon X \to D^{\op}(D(X))$ which is an equivalence whenever $X$ is compact.
        \item[(3)] The map of (2) is the unit of an adjunction making $D^{\op}$ right adjoint to $D$:
        \[
        \begin{tikzcd}
        \Sp^G_{\mathrm{prop}} \ar[shift left]{r}{D} & (\Sp^G_{\mathrm{prop}})^{\op}. \ar[shift left]{l}{D^{\op}}
        \end{tikzcd}
        \]
        In particular, for $X, Y \in \Sp^G_{\mathrm{prop}}$ there is a natural equivalence 
        \[
        \mathrm{Map}^G(X, DY) \simeq \mathrm{Map}^G(Y, DX),
        \]
        where $\mathrm{Map}^G$ is shorthand for the mapping space in the $\infty$-category $\Sp^G_{\mathrm{prop}}$.
    \end{itemize}
\end{prop}
\begin{proof}
By Construction \ref{con:duality}, the diagram 
\[
\begin{tikzcd}
(\Sp^H)^{\omega} \ar{d}{\mathrm{ind}_H^G}\ar{r}{D} & ((\Sp^H)^{\omega})^{\mathrm{op}} \ar{d}{(\mathrm{ind}_H^G)^{\mathrm{op}}} \\
(\Sp^G_{\mathrm{prop}})^{\omega} \ar{r}{D} & ((\Sp^G_{\mathrm{prop}})^{\omega})^{\mathrm{op}}
\end{tikzcd}
\]
commutes for any finite subgroup $H \leq G$. Claim (1) follows by evaluating both composites on the unit $\mathbb{S}_H \in (\Sp^H)^{\omega}$, relying on the fact that $\mathrm{ind}_H^G \mathbb{S}_H = \Sigma^\infty_+ G/H$. 

We will now prove (3). Consider again the two diagrams $\mathrm{ind}, \mathrm{ind}^{\mathrm{op}}\colon \mathrm{Or}^G_{\mathrm{fin}} \to \mathrm{Cat}_\infty^{\rex}$. We have seen that Spanier--Whitehead duality defines a natural transformation $D\colon \mathrm{ind} \to \mathrm{ind}^{\mathrm{op}}$. Moreover, for every morphism $G/H \to G/K$ in $\mathrm{Or}^{G}_{\mathrm{fin}}$, the left of the following two squares
\[
\begin{tikzcd}
(\Sp^H)^{\omega} \ar{d}{\mathrm{ind}_H^K}\ar{r}{D} &  ((\Sp^H)^{\omega})^{\mathrm{op}} \ar{d}{(\mathrm{ind}_H^K)^{\mathrm{op}}} && (\Sp^H)^{\omega} \ar{d}{\mathrm{ind}_H^K} & ((\Sp^H)^{\omega})^{\mathrm{op}} \ar{d}{(\mathrm{ind}_H^K)^{\mathrm{op}}} \ar{l}[swap]{D^{\mathrm{op}}} \\
(\Sp^K)^{\omega} \ar{r}{D} & ((\Sp^K)^{\omega})^{\mathrm{op}} && (\Sp^K)^{\omega} & ((\Sp^K)^{\omega})^{\mathrm{op}} \ar{l}{D^{\mathrm{op}}}
\end{tikzcd}
\]
is \emph{right adjointable}, in the sense that the associated Beck--Chevalley transformation $\mathrm{ind}_H^K \circ D^{\mathrm{op}} \to D^{\mathrm{op}} \circ (\mathrm{ind}_H^K)^{\mathrm{op}}$ is an equivalence, thus making the right square commute as well. By \cite[Proposition 2.1.5]{elmantohaugseng}, this implies that the functors $D^{\mathrm{op}}$ assemble into a right adjoint $D^{\mathrm{op}}\colon \mathrm{ind}^{\mathrm{op}} \to \mathrm{ind}$ to $D$ in the $(\infty,2)$-category $\mathrm{Fun}(\mathrm{Or}^G_{\mathrm{fin}},\mathrm{Cat}_\infty^{\rex})$. Now applying the functor of $(\infty,2)$-categories
\[
\mathrm{colim}\colon \mathrm{Fun}(\mathrm{Or}^G_{\mathrm{fin}},\mathrm{Cat}_\infty^{\rex}) \to \mathrm{Cat}_\infty^{\rex},
\]
we find an adjunction
\[
\begin{tikzcd}
(\Sp^G_{\mathrm{prop}})^{\omega} \ar[shift left]{r}{D} & ((\Sp^G_{\mathrm{prop}})^{\omega})^{\mathrm{op}}. \ar[shift left]{l}{D^{\mathrm{op}}}
\end{tikzcd}
\]
By Kan extension from compact objects we also find the adjunction of item (3) in the proposition. 

For claim (2), we take the unit map of the adjunction just constructed. To verify that it is an equivalence for compact $X$, we may reduce to checking this for a family of compact generators, for which we may take $\mathrm{ind}_H^G \mathbb{S}$ with $H$ ranging over finite subgroups of $G$. By construction we have a commutative diagram
\[
\begin{tikzcd}
\mathrm{ind}_H^G \mathbb{S} \ar{r}{\eta}\ar{dr}[swap]{\mathrm{ind}_H^G(\eta)} & D^{\mathrm{op}}D(\mathrm{ind}_H^G \mathbb{S}) \\
& \mathrm{ind}_H^G(D^{\mathrm{op}}D(\mathbb{S})). \ar{u}{\simeq}
\end{tikzcd}
\]
The slanted map is evidently an equivalence, completing the proof.
\end{proof}

\begin{rem}
The functor $D$ constructed above is generally \emph{different} from duality with respect to the symmetric monoidal structure of $\Sp^G_{\mathrm{prop}}$. For example, Proposition \ref{prop:dualitySpG}(1) shows that the orbits $\Sigma^\infty_+ G/H$ are `self-dual' with respect to $D$. However, if $G$ is an infinite group then these objects are never dualizable with respect to the symmetric monoidal structure: indeed, in that case the underlying spectrum of $\Sigma^\infty_+ G/H$ is an infinite sum of copies of the sphere spectrum.
\end{rem}

Consider the functor
\[
\mathrm{Or}^G_{\mathrm{fin}} \to \Sp^G_{\mathrm{prop}}\colon G/H \mapsto \Sigma^\infty_+ G/H.
\]
Applying $D$, we obtain another diagram
\[
(\mathrm{Or}^G_{\mathrm{fin}})^{\op} \to \Sp^G_{\mathrm{prop}}\colon G/H \mapsto D(\Sigma^\infty_+ G/H)
\]
of which the values may be identified with $\Sigma^\infty_+ G/H$ again by Proposition \ref{prop:dualitySpG}(1). In particular, for every subgroup inclusion $K \leq H$ we obtain a `wrong-way map'
\[
\mathrm{tr}_K^H\colon \Sigma^\infty_+ G/H \to \Sigma^\infty_+ G/K
\]
that we will refer to as the \emph{transfer}. This map coincides with the transfer map constructed in \cite[Section 2.2]{DHLPS}. 

\begin{lem}
\label{lem:DStransfer}
There is an equivalence
\[
D(\mathbb{S}_G) \xrightarrow{\simeq} \lim_{G/H \in (\mathrm{Or}^G_{\mathrm{fin}})^{\op}} \Sigma^\infty_+ G/H,
\]
where the limit is along the diagram of transfer maps described above.
\end{lem}
\begin{proof}
This is immediate from the colimit description of $\mathbb{S}_G = \Sigma^\infty_+ \underline{E}G$ of Lemma \ref{lem:SGcolim} and the fact that $D$ takes colimits to limits.
\end{proof}

Of course, if we have a more efficient description of $\Sigma^\infty_+ \underline{E}G$ (for example if $\underline{E}G$ is finite) then we find a corresponding description of $D(\mathbb{S}_G)$:

\begin{ex} \label{ex: amalgamated} Let $G =K \ast_H L$ be an amalgamated product of finite groups. Using Example \ref{ex: amalgamated1} and the proof of Lemma \ref{lem:DStransfer}, we obtain the following fiber sequence in $\Sp^G_{\mathrm{prop}}$:
\[\xymatrix{D(\Ss_G) \ar[r] & \Sigma^{\infty} G/K_+ \oplus \Sigma^{\infty} G/L_+ \ar[rr]^-{(\tr_H^K, -\tr_H^L)}   & & \Sigma^{\infty} G/H_+ .}\]
\end{ex}




We have already remarked that the functor $D$ is not the same as duality with respect to the symmetric monoidal structure of $\Sp^G_{\mathrm{prop}}$. Lemma \ref{lem:DStransfer} is another witness of this fact: indeed, it shows that the dual of the unit $D(\mathbb{S}_G)$ is generally not equivalent to $\mathbb{S}_G$ itself. In fact the following is true:

\begin{prop}  \label{prop: dualizing module} Let $G$ be a discrete group and $X \in \Sp^G_{\mathrm{prop}}$. Then the underlying spectrum $D(X)^e \in \mathrm{Fun}(BG,\Sp)$ is described by natural equivalences
\[
D(X)^e \simeq \mathrm{map}^G(X, \mathbb{S}[G]) \simeq \mathrm{map}^G(X \otimes \Sigma^{\infty}_+ EG, \mathbb{S}[G]).
\]
In particular, the underlying spectrum $D(\mathbb{S}_G)^e \in \mathrm{Fun}(BG,\Sp)$ is equivalent to $\mathbb{S}[G]^{hG}$, which is the dualizing complex $\mathbf{D}_{BG} \in \mathrm{Fun}(BG,\Sp)$ of Definition \ref{def:dualizingcomplex}. 
\end{prop}

\begin{proof} We can write
\[
D(X)^e \simeq \mathrm{map}^G(\mathbb{S}[G], D(X)) \simeq \mathrm{map}^G(X,D(\mathbb{S}[G])) \simeq \mathrm{map}^G(X,\mathbb{S}[G])
\]
using Lemma \ref{lem:underlyingspectrum} for the first equivalence, Proposition \ref{prop:dualitySpG}(3) for the second, and Proposition \ref{prop:dualitySpG}(1) for the third. Next consider the natural transformation
\[\mathrm{map}^G(X,\mathbb{S}[G]) \to \mathrm{map}^G(X \otimes \Sigma^{\infty}_+ EG,\mathbb{S}[G]),\]
induced from the projection $\Sigma^{\infty}_+ EG \to \mathbb{S}_G$. We claim that this map is an equivalence. To see this, it will suffice to prove the claim for $X=\Sigma^\infty_+ G/H$, with $H$ a finite subgroup, since such objects generate the $\infty$-category of proper $G$-spectra under colimits and desuspensions. In this particular case, the map above is equivalent to the map
\[\mathrm{map}^H(\mathbb{S}_H, \mathrm{res}^G_H \mathbb{S}[G]) \to \mathbb{S}[G]^{hH} \]
induced by the projection $\Sigma^{\infty}_+ EH \to \mathbb{S}_H$. Since the restriction $\mathrm{res}^G_H \mathbb{S}[G]$ is an induced $H$-object from the trivial subgroup, the left-hand side may be identified with the homotopy orbits $\mathbb{S}[G]_{hH}$ (this is essentially the Adams isomorphism \cite{LMS, RV16, CLL}) and the map becomes the norm. This is an equivalence in this case, because the Tate construction for finite groups vanishes on induced spectra from the trivial subgroup \cite[Lemma I.3.8]{NS}.
\end{proof}

\begin{rem} \label{rem: comparison of limits} It follows from Proposition \ref{prop: dualizing module} and Lemma \ref{lem:SGcolim} that one has an equivalence:
\[D(\mathbb{S}_G)^e \simeq \mathbb{S}[G]^{hG} \simeq  \lim_{G/H \in (\mathrm{Or}^G_{\mathrm{fin}})^{\op}} \mathbb{S}[G]^{hH}.\]
This yields a map $D(\mathbb{S}_G)^e \to \mathbb{S}[G]^{hH}$ for any finite subgroup $H \leq G$. On the other hand, for any finite subgroup $H \leq G$, the inverse limit of Lemma \ref{lem:DStransfer} gives a morphism $D(\mathbb{S}_G) \to \Sigma^\infty_+ G/H$. These maps are compatible in the sense that the following diagram in $\Fun(BG, \Sp)$ commutes:
\[\xymatrix{D(\mathbb{S}_G)^e \ar[d] \ar[rr]^{\simeq} & & \mathbb{S}[G]^{hG} \ar[d] \\ (\Sigma^{\infty}_+ G/H)^e \ar[rr]^{\mathrm{Nm}_{BH}} & & \mathbb{S}[G]^{hH}. }\]
To check the commutativity of this diagram, we first observe that by naturality of the equivalence in Proposition \ref{prop: dualizing module}, the diagram in $\Fun(BG, \Sp)$
\[\xymatrix{D(\mathbb{S}_G)^e \ar[d]^{D(\res)} \ar[rr]^{\simeq} & & \mathbb{S}[G]^{hG} \ar[d]^{\res} \\ (D(\Sigma^{\infty}_+ G/H))^e \ar[rr]^{\simeq} & & \mathbb{S}[G]^{hH} }\]
commutes. We need to identify the composite 
\[(\Sigma^{\infty}_+ G/H)^e \simeq (D(\Sigma^{\infty}_+ G/H))^e \xrightarrow{\simeq} \mathbb{S}[G]^{hH}  \]
with the norm $\mathrm{Nm}_{BH}$ in $\Fun(BG,\Sp)$, where the first equivalence is the preferred identification from Proposition \ref{prop:dualitySpG}(1) and the second map is the equivalence of Proposition \ref{prop: dualizing module}. Using the proof of Lemma  \ref{lem:mixednormvsAnorm}, it further suffices to show that the adjoint composite in $\Fun(BH, \Sp)$ 
\[\mathbb{S} \xrightarrow{u} i^*(\Sigma^{\infty}_+ G/H)^e \to i^*(\mathbb{S}[G]^{hH}), \]
where $u$ is the unit of the adjunction and $i \colon BH \to BG$ is the inclusion, is equal to the composite 
\[\xymatrix{\mathbb{S} \simeq \mathbb{S}[H]_{hH} \ar[rr]^{\mathrm{Nm}_{BH}}_{\simeq} & & \mathbb{S}[H]^{hH} \ar[r]^-{\incl} & i^*(\mathbb{S}[G]^{hH})}.\]
This now follows from the commutative diagram in $\Fun(BH, \Sp)$:
\[\xymatrix{\mathbb{S} \ar@{=}[d]\ar[r]^-u_-{\simeq} & \mathrm{map}^H(\mathbb{S}[H], \mathbb{S}_H) \ar[d]^{\mathrm{ind}_H^G} \ar[r]^-{D} & \mathrm{map}^H(\mathbb{S}_H, \mathbb{S}[H]) \ar[r]^-{\simeq} \ar[d]^{\mathrm{ind}_H^G} & \mathbb{S}[H]^{hH} \ar@{=}[dd] \\ \mathbb{S} \ar@{=}[d] \ar[r]^-u & i^*\mathrm{map}^G(\mathbb{S}[G], \Sigma^{\infty}_+ G/H) \ar[r]^-{D} & i^*\mathrm{map}^G(\Sigma^{\infty}_+ G/H, \mathbb{S}[G]) \ar[d]^{\simeq} & \\ \mathbb{S} \ar[r]^-u & i^*(\Sigma^{\infty}_+ G/H)^e \ar[u]_\simeq \ar[r] & i^*(\mathbb{S}[G]^{hH}) & \mathbb{S}[H]^{hH}, \ar[l]_{\incl} } \]
where we implicitly used the preferred equivalence of Proposition \ref{prop:dualitySpG}(1) to simplify the notation. The upper square commutes by Construction \ref{con:duality}. The lower middle square is the definition of the map we are trying to identify. The right rectangle commutes by the definition of induction. The composite on the top row is the norm map by the Wirthm\"uller isomorphism and Spanier--Whitehead duality for finite groups. The horizontal maps on the left are appropriate unit maps and the squares on the left commute by definition. 
    
\end{rem}

To conclude this section, we outline a useful generalization of the norm map to a natural map
\[
N_{X,Y}\colon (X \otimes D(Y))_{hG} \to \mathrm{map}(Y,X)^{hG},
\]
for $X, Y \in \Sp_{\mathrm{prop}}^G$. This map will essentially reduce to our previous norm $\mathrm{Nm}_{BG}$ in case $Y = \mathbb{S}$, but the added naturality can for example be used to prove Corollary \ref{cor: stable e underbar duality} below. Since we will not need these results, we only indicate the necessary ingredients and leave a detailed development for future work.

\begin{con}
Suppose $X, Y \in \Sp_{\mathrm{prop}}^G$. First observe that
\[
X \otimes D(Y) = \mathrm{map}^G(\mathbb{S}[G], X) \otimes \mathrm{map}^G(Y, \mathbb{S}[G]),
\]
using Proposition \ref{prop: dualizing module} for the second factor. Consider the composition of maps
\[
\mathrm{map}^G(\mathbb{S}[G], X) \otimes \mathrm{map}^G(Y, \mathbb{S}[G]) \to \mathrm{map}^G(Y,X) \to \mathrm{map}(Y,X)^{hG}.
\]
The first arrow is composition (at the level of mapping spectra) in the $\infty$-category $\Sp_{\mathrm{prop}}^G$ and the second arises from the forgetful functor $\Sp_{\mathrm{prop}}^G \to \mathrm{Fun}(BG,\Sp)$. Thinking of $\mathbb{S}[G]$ as a $G \times G$-object, the domain of this map carries a residual $G$-action and the map is $G$-equivariant if we equip the target with the trivial $G$-action. Hence we find a factorization through the homotopy orbits of the domain, which is the promised map
\[
(X \otimes D(Y))_{hG} \xrightarrow{N_{X,Y}} \mathrm{map}(Y,X)^{hG}.
\]
\end{con}

If $H \leq G$ is a finite subgroup and $Y = \Sigma^\infty_+ G/H$, then $(X \otimes D(Y))_{hG} \simeq X_{hH}$ and $\mathrm{map}(Y,X)^{hG} \simeq X^{hH}$. In fact, with a little bit of diagram chasing as in Remark \ref{rem: comparison of limits} one can show:

\begin{lem}
\label{lem:generalizednorm}
For a finite subgroup $H \leq G$ and $Y = \Sigma^\infty_+ G/H$, the map $N_{X,Y}$ may be identified with the norm
\[
\mathrm{Nm}_{BH}\colon X_{hH} \to X^{hH}.
\]
\end{lem}

\begin{rem} \label{rem: norm limit}
In the case $Y = \mathbb{S}_G$, the map $N_{X,Y}$ should reproduce our earlier norm map
\[
\mathrm{Nm}_{BG}\colon (X \otimes D(\mathbb{S}_G))_{hG} \to X^{hG}.
\]
To establish this, one would need a naturality of norm maps with more coherence than what we established in Section \ref{sec: norm and Tate}. Indeed, if one can make the diagrams (shown to commute in the proof of Proposition \ref{prop: Quillen formula}) 
\[
\begin{tikzcd}
(X \otimes D(\mathbb{S}_G))_{hG} \ar{d}{\mathrm{tr}^G_H} \ar{r}{\mathrm{Nm}_{BG}} & X^{hG} \ar{dd}{\mathrm{res}^G_H} \\
(X \otimes \Sigma^\infty_+ G/H)_{hG} \ar{d}{\simeq} & \\
X_{hH} \ar{r}{\mathrm{Nm}_{BH}} & X^{hH}
\end{tikzcd}
\]
functorial in $G/H \in \mathrm{Or}^G_{\mathrm{fin}}$, then taking the limit over $(\mathrm{Or}^G_{\mathrm{fin}})^{\mathrm{op}}$ and observing $X^{hG} \simeq \varprojlim_{G/H} X^{hH}$ would provide a factorization of $\mathrm{Nm}_{BG}$ through a limit of norm maps for finite groups, which one could then combine with Lemma \ref{lem:generalizednorm} to compare $\mathrm{Nm}_{BG}$ with $N_{X,\mathbb{S}_G}$. We will not need such a comparison here and therefore leave it to future work.

\end{rem}

If we work in a setting where norm maps for finite groups are equivalences, we obtain the following:

\begin{prop}
\label{prop:generalizednorm}
If $X$ is a $T(n)$-local proper $G$-spectrum, then $N_{X,Y}$ is a $T(n)$-equivalence for any compact object $Y$ of $\Sp_{\mathrm{prop}}^G$.
\end{prop}
\begin{proof}
Note that both the domain and codomain of $N_{X,Y}$ are exact functors of $Y$, so it suffices to check this for a family of compact generators, e.g., $\Sigma^\infty_+ G/H$ with $H$ ranging over finite subgroups of $G$. In these cases we may apply Lemma \ref{lem:generalizednorm} to reduce to norm maps for finite groups and we conclude by the fact that $T(n)$-local spectra form a 1-semiadditive $\infty$-category.
\end{proof}

In Theorem \ref{thm:Tatevanishing} we showed that if $G$ admits a finite $\underline{E}G$, then the norm map for $BG$ with coefficients in a 1-semiadditive $\infty$-category is an equivalence. Recall that finiteness of $\underline{E}G$ implies that the unit $\mathbb{S}_G \in \Sp_{\mathrm{prop}}^G$ is a compact object. This second property is often quoted as saying that $G$ has a finite \emph{stable} $\underline{E}G$ (see \cite{BDP} for details). 
The following is a strengthening of the conclusion of Theorem \ref{thm:Tatevanishing}:

\begin{cor} \label{cor: stable e underbar duality}
If $G$ has finite stable $\underline{E}G$ and $X \in \mathrm{Fun}(BG,\Sp_{T(n)})$, then the norm map
\[
\mathrm{Nm}_{BG}\colon (X \otimes \mathbf{D}_{BG})_{hG} \to X^{hG}
\]
is an equivalence in $\Sp_{T(n)}$.
\end{cor}
\begin{proof}
By Proposition \ref{prop:f*Clinear}, it will suffice to show that the functor 
\[
\mathrm{Fun}(BG,\Sp_{T(n)}) \to \Sp_{T(n)}\colon X \mapsto X^{hG}
\]
preserves colimits. This follows immediately by applying Proposition \ref{prop:generalizednorm} to $Y = \mathbb{S}_G$.
\end{proof}

\begin{ex} There are groups with finite stable $\underline{E}G$ which do not admit a finite unstable $\underline{E}G$. This implies that Corollary \ref{cor: stable e underbar duality} applies to more groups than Theorem \ref{thm:Tatevanishing}. Torsion-free examples are provided by the celebrated Bestvina--Brady groups \cite{BestvinaBrady}, which are $FP_{\infty}$ but not finitely presented (see \cite[Remark 5.2]{BDP}). A generalization with interesting torsion can be provided by semidirect products of Bestvina-Brady groups with finite groups as considered in \cite{LearyNucinkis}. Indeed, let $L$ be a finite flag complex with a simplicial admissible action of a finite group $K$. Suppose that $L^{K'}$ is acyclic for any subgroup $K' \leq K$ and assume additionally that $L$ is not simply-connected. Now consider the Bestvina--Brady group $H_L$ which is defined to be the kernel of the group homomorphism $G_L \to \Z$, where $G_L$ is the right-angled  Artin group associated to $L$, and the homomorphism sends generators to $1$. The group $K$ acts on $G_L$ and hence on $H_L$ via the given action on $L$. Consider the semidirect product $G=H_L \rtimes K$. By \cite[Theorem 3]{LearyNucinkis} and \cite{BestvinaBrady} this group only has finitely many conjugacy classes of finite subgroups and for any finite subgroup $F \leq G$, the Weyl group $N(F)/F$ is $FP_{\infty}$. Additionally, it has finite virtual cohomological dimension \cite[Section 6]{LearyNucinkis}. It then follows from \cite[Theorem 1.2 and Theorem 1.3]{BDP} that $G$ admits a finite stable $\underline{E}G$. However, since $L$ is not simply-connected, we know that $H_L$ and hence $G=H_L \rtimes K$ are not finitely presented \cite{BestvinaBrady} and hence $G$ does not admit a finite $\underline{E}G$ by \cite[Theorem 5]{LearyNucinkis} (see also \cite[Theorem 5.1]{Luckfinite}). 

An example of such an acyclic flag complex $L$ is given by the barycentric subdivision of the Floyd--Richardson complex \cite{FR}, which admits an admissible simplicial action of $A_5$. The details can be found for example in \cite[Section 9]{LearyNucinkis} and \cite[Section 6]{BDP}. Though $L^{A_5}=\emptyset$, for any proper subgroup $K \leq A_5$, the fixed subcomplex $L^K$ is acyclic. Hence for any proper subgroup $K \leq A_5$, the group $G=H_L \rtimes K$ is an example with finite stable $\underline{E}G$, but no finite $\underline{E}G$.
\end{ex}

\part{Chromatic Euler characteristics}

\section{Applications to chromatic cardinalities} \label{sec: cardinalities}

Let $G$ be a discrete group. According to \cite[Section 6.6]{LL2} an \emph{additive invariant} for $G$ is a pair $(A, \chi)$, where $A$ is an abelian group and $\chi$ assigns to every finite proper $G$-CW complex $X$ a number $\chi(X) \in A$ in such a way that the following properties hold: 

\begin{itemize}  

\item[(1)] If $X$ and $Y$ are $G$-homotopy equivalent, then $\chi(X)=\chi(Y)$.

\item[(2)] Given a (homotopy) pushout 
\[\xymatrix{X \ar[r] \ar[d] & Y \ar[d]  \\ Z \ar[r] & T}\]
of finite proper $G$-CW complexes, one has 
\[\chi(T) +\chi(X)=\chi(Z)+\chi(Y).\]

\item[(3)] One has $\chi(\emptyset)=0.$

\end{itemize}
By \cite[Lemma 6.85]{LL2}, there is a universal additive invariant which can be constructed as follows: let $A(G)$ denote the colimit
\[\colim_{G/H \in \Or^G_{\mathrm{fin}}} A(H), \]
where $A(H)$ is the Burnside ring for a finite subgroup $H \leq G$ and the colimit is taken along inductions and conjugations. Note that despite the fact that $A(H)$ has a natural ring structure for any finite $H \leq G$, the colimit $A(G)$ is not a ring in general since the inductions are not multiplicative. Equivalently, $A(G)$ is the Grothendieck group of cofinite $G$-sets with finite stabilizers. The group $A(G)$ is free abelian with basis $\{[G/H] \; \vert \; H \leq G\; \text{finite} \}$ \cite[Section 6.6.2]{LL2}.

The universal additive invariant $[-]$ is now defined as follows:

\[[X]=\sum_{G\sigma} (-1)^{d_\sigma} [G/H_{\sigma}] \in A(G),\]
    where the sum runs over all $G$-orbits of cells $\sigma$ of $X$, the number $d_\sigma$ is the dimension of $\sigma$ and $H_{\sigma}$ is the stabilizer of $\sigma$. 

It follows from the universal property that an additive invariant $(A, \chi)$ amounts to a homomorphism $\chi \colon A(G) \to A$. Since $A(G)$ is free abelian with the above specified basis, in order to define $\chi$, it suffices to prescribe the values $\chi(G/H)$ for any finite $H$. 

The following general fact will be useful for identifying the set of connected components of $\mathrm{map}^G(D(\Ss_G), \Ss_G)$
with $A(G)$:

\begin{lem} \label{lem: Kan ext on psh} Let $\mathcal{C}$ be a small $\infty$-category, $\mathcal{D}$ a cocomplete $\infty$-category and $F \colon \mathcal{C} \to \mathcal{D}$ a functor. Then the value of the left Kan extension (along the Yoneda embedding) $\hat{F} \colon \mathrm{PSh}(\mathcal{C})\to \mathcal{D}$ at the terminal presheaf is equivalent to the colimit $\colim_{\mathcal{C}}F$. 

\end{lem}

\begin{proof} This follows from the pointwise formula for Kan extensions and from the observation that the overcategory $Y_{{/\ast}}$ is equivalent to $\mathcal{C}$, where $Y \colon \mathcal{C} \to \mathrm{PSh}(\mathcal{C})$ is the Yoneda embedding and $\ast$ is the terminal presheaf. \end{proof}

\begin{lem} \label{lem: dulaising maps} Let $G$ be a discrete group with finite $\underline{E}G$. Then the set of homotopy classes of maps $[D(\Ss_G), \Ss]^G$ in the homotopy category of $\Sp^G_{\mathrm{prop}}$ is in natural bijection with $A(G)$. 
\end{lem}

\begin{proof} Let 
\[F \colon \Or_{\mathrm{fin}}^G \to \Sp \]
denote the restriction of
the functor 
\[\mathrm{map}^G(D(\Sigma^{\infty}_+(-)), \Ss_G) \colon \mathcal{S}^G_{\mathrm{prop}} \to \mathrm{Sp}\]
to the orbit category. By Proposition \ref{prop:dualitySpG}(1), the functor $F$ sends the coset $G/H$ to the genuine fixed points $\mathbb{S}^H$, restrictions to transfers, and conjugations by $g$ to conjugations by $g^{-1}$. Let $\hat{F}$ denote the left Kan extension of $F$ along the Yoneda embedding $\Or_{\mathrm{fin}}^G \to \mathcal{S}^G_{\mathrm{prop}}$ (see Proposition \ref{prop:properGspaces}).  The counit
\[\hat{F} \to \mathrm{map}^G(D(\Sigma^{\infty}_+(-)), \Ss_G)\]
is a natural transformation of functors defined on $\mathcal{S}^G_{\mathrm{prop}}$ and restricts to the identity on  $\Or_{\mathrm{fin}}^G$. The functor $\hat{F}$ preserves colimits as a left Kan extension. The functor $\mathrm{map}^G(D(\Sigma^{\infty}_+(-)), \Ss_G)$ preserves finite colimits by construction of $D$ and stability. Hence for any finite proper $G$-CW complex $X$, the map 
\[\hat{F}(X) \to \mathrm{map}^G(D(\Sigma^{\infty}_+(X)), \Ss_G)\]
is an equivalence. Since $G$ admits a finite $\underline{E}G$, we can take $X=\underline{E}G$ and conclude that 
\[\hat{F}(\underline{E}G) \to \mathrm{map}^G(D(\mathbb{S_G}), \Ss_G)\]
is an equivalence. The proper $G$-space $\underline{E}G$ is equivalent to the terminal presheaf (see Lemma \ref{lem: finite underline}) and hence by Lemma \ref{lem: Kan ext on psh}, we obtain an equivalence
\[\colim_{\Or_{\mathrm{fin}}^G} \mathbb{S}^H \simeq \hat{F}(\ast) \simeq \hat{F}(\underline{E}G) \simeq \mathrm{map}^G(D(\mathbb{S_G}), \Ss_G).\] 
Since $\mathbb{S}^H$ is connective for any finite $H \leq G$, the result now follows just by applying $\pi_0$ to the latter equivalence and using the isomorphism $\pi_0(\Ss^H) \cong A(H)$ (see e.g., \cite{tomD}). \end{proof}


This lemma implies that given a finite proper $G$-CW complex $X$, we have a map $[X] \colon D(\Ss_G) \to \Ss_G$ corresponding to the universal element $[X] \in A(G)$. In particular, when $G$ admits a finite $\underline{E}G$, there is a canonical map 
\[e(G)=[\underline{E}G] \colon D(\Ss_G) \to \Ss_G.\]

Now let $E$ be a $K(n)$-local spectrum equipped with the trivial $G$-action. If $G$ is finite, then one can use ambidexterity \cite{GS96, HS96, HL13} to define the $E$-cardinality of $BG$, denoted by $\vert BG \vert_E$, as the following composite in the $K(n)$-local stable homotopy category (see \cite[Section 2]{CSYambheight}):
\[\xymatrix{E \ar[r]^-{\Delta} & E^{hG} \ar[r]^{\mathrm{Nm}_{BG}^{-1}} & E_{hG} \ar[r]^{\nabla} & E.}\]
If $E$ is a $K(n)$-local $\mathbb{E}_1$-algebra, then the map above is one of $E$-modules and one might as well interpret $\vert BG \vert_E$ as an element of $\pi_0 E$. Using Theorem \ref{thm:Tatevanishing}, we generalize this invariant to infinite discrete groups with finite $\underline{E}G$. 

\begin{defn} \label{def: E-theory euler char} Suppose $G$ admits a finite $\underline{E}G$, let $E$ be a $K(n)$-local spectrum with trivial $G$-action, and $X$ a finite proper $G$-CW complex. The \emph{$E$-theory $G$-equivariant Euler characteristic} $\chi^E_{G}(X)$ is defined to be the following composite in the $K(n)$-local stable homotopy category:
\[\xymatrix{E \ar[r]^-{\Delta} & E^{hG} \ar[r]^-{\mathrm{Nm}_{BG}^{-1}} &  (E \otimes D(\Ss_G))_{hG} \ar[rr]^-{(E \otimes [X])_{hG}} & & E_{hG} \ar[r]^{\nabla} & E.}\]
The map $\mathrm{Nm}_{BG}^{-1} \colon E^{hG} \to (E \otimes D(\Ss_G))_{hG}$ is the inverse of the equivalence given by Theorem \ref{thm:Tatevanishing}. In case $E$ is an $\mathbb{E}_1$-algebra, we will think of $\chi^E_{G}(X)$ as an element of $\pi_0 E$. 
The \emph{$E$-theory orbifold Euler characteristic} of $G$ is then defined to be $\chi^E_{G}(\underline{E}G)$ and will be denoted by $\chi^E_{\orb}(G)$. 
\end{defn}

\begin{rem}
In the definition above, we do actually need a finite $G$-CW model for $\underline{E}G$, rather than a finitely dominated one. Our previous results on norm maps only needed $\underline{E}G$ to be compact (i.e., finitely dominated) as a proper $G$-space, but in the remainder of this paper we need the stronger finiteness assumption. As already noted in Remark \ref{rem:finitevsfinitelydom}, this presents no essential loss of generality as far as examples are concerned.
\end{rem}

\begin{prop} \label{prop: Quillen formula} Suppose $G$ admits a finite $\underline{E}G$ and let $E$ be a $K(n)$-local spectrum. Then for any finite proper $G$-CW complex $X$, we have
\[\chi^E_{G}(X)=\sum_{G\sigma} (-1)^{d_\sigma} \vert BH_{\sigma} \vert_E,\]
where the sum runs over all $G$-orbits of cells $\sigma$ of $X$, the number $d_\sigma$ is the dimension of $\sigma$ and $H_{\sigma}$ is the stabilizer of $\sigma$. 
    
\end{prop}

\begin{proof} To prove the formula, we need to check the following two claims:
\rm 1) $\chi^E_{G}$ is an additive invariant;
\rm 2) For any finite subgroup $H \leq G$, $\chi_{G}^E(G/H)=\vert BH \vert_E$. The first claim follows immediately from the definition and from the fact that $[X] \in A(G)$ is an additive invariant of $X$. 

To prove the second claim we first observe that the morphism $[G/H]$ factors as
\[\xymatrix{D(\Ss_G) \ar[rr]^-{[G/H]} \ar[dr]_{\tr}  & &  \Ss_G \\ & \Sigma^{\infty}_+ G/H, \ar[ur]_{\res} &   }\]
where $\tr$ is the dual of the restriction $\res \colon \Sigma^{\infty}_+ G/H \to \Ss_G$. The following diagram commutes by construction of the mixed dualizing complex and mixed norm map from Section \ref{sec: norm and Tate}, (the proof of) Lemma \ref{lem:mixednormvsAnorm}, and Remark \ref{rem: comparison of limits}:
\[\xymatrix{(E \otimes D(\mathbb{S}_G))_{hG} \ar[d]^{\tr_{hG}} \ar[r]^\simeq & (E \otimes \mathbb{S}[G]^{hG})_{hG} \ar[d]^{\res_{hG}} \ar[rr]^-{\mathrm{Nm}_{BG}} & & E^{hG} \ar[d]^{\res} \\ (E \otimes \Sigma^{\infty}_+ G/H )_{hG} \ar[r]^{\simeq}  & (E \otimes \mathbb{S}[G]^{hH} )_{hG} \ar[rr]^-{\mathrm{Nm}_{(BG,BH)}} & & E^{hH} \ar@{=}[d] \\ E_{hH} \ar[r]^-{\simeq} \ar[u]^{\simeq} & (E \otimes \mathbb{S}[H]^{hH})_{hH}  \ar[u]^-{\simeq} \ar[rr]^-{\mathrm{Nm}_{BH}} & & E^{hH}.}\]
Now the claim follows from the commutative diagram in the $K(n)$-local stable homotopy category: 
\[\xymatrix{E \ar[r]^{\Delta} \ar[dr]_{\Delta} & E^{hG} \ar[d]^{\res} \ar[r]^-{\mathrm{Nm}_{BG}^{-1}} &  (E \otimes D(\Ss_G))_{hG} \ar[d]_{\tr_{hG}} \ar[rr]^-{(E \otimes [G/H])_{hG}} & & E_{hG} \ar[r]^{\nabla} & E. \\ & E^{hH} \ar[r]^{\mathrm{Nm}_{BH}^{-1}} & E_{hH} \ar[urr]^{\res} \ar[urrr]_{\nabla} }\]
We observe that the triangles on the left and right commute by definition. \end{proof}

The typical examples to which we want to apply Proposition \ref{prop: Quillen formula} are the $K(n)$-local sphere $L_{K(n)}\Ss$ and Morava $E$-theory $E_n$. The following is a consequence of \cite[Theorem 5.8]{BenMosche} (see also \cite[Lemma 4.13]{HKR} and \cite[Section 3.4]{LurieEl}):

\begin{prop} \label{prop: card for finite} Let $G$ be a finite group. Then
\[\vert BG \vert_{E_{n}}=\frac{\vert G_{n,p} \vert}{\vert G \vert}.\]
\end{prop}

It is a consequence of the Frobenius theorem (see e.g., \cite[Lemma 3.2]{P24}) that if $n \geq 1$, then $\vert BG \vert_{E_{n}}$ is $p$-integral, i.e., it is an element of $\Z_{(p)} \subset \Z_p \subset \pi_0(E_n)$. 

\begin{cor} \label{cor: Quillen} Suppose $G$ admits a finite $\underline{E}G$. Then for any finite proper $G$-CW complex $X$, we have 
\[\chi^{E_n}_{G}(X)=\sum_{G\sigma} (-1)^{d_\sigma} \frac{\vert (H_\sigma)_{n,p} \vert}{\vert H_\sigma \vert} ,\]
where the sum runs over all $G$-orbits of cells $\sigma$ of $X$, the number $d_\sigma$ is the dimension of $\sigma$, and $H_{\sigma}$ is the stabilizer of $\sigma$. In particular, $\chi^{E_n}_{G}(X)$ is a $p$-local integer whenever $n \geq 1$. 
\end{cor}

\begin{rem} \label{remark: n=0 special case} It follows from Corollary \ref{cor: Quillen} and \cite[Proposition IX.7.3]{Brown} that at height $n=0$, the invariant $\chi^{E_n}_{G}(X)$ recovers the classical equivariant orbifold Euler characteristic of Wall and Serre \cite{Wall, Serre1}. For $G$ a virtually torsion-free group and $X$ a finite proper $G$-CW complex, this is classically defined by 
\[ \chi_{G}(X)=\frac{\chi^{\Q}(X/\Gamma)}{[G:\Gamma]},\]
where $\chi^{\Q}$ is the rational cohomology Euler characteristic and $\Gamma \leq G$ is a finite index torsion-free subgroup. More generally, $\chi_{G}(X)$ is an additive invariant valued in $\Q$, which can be computed using Quillen's formula (see \cite[Proposition IX.7.3]{Brown82})
\[\chi_G(X)=\sum_{G\sigma} (-1)^{d_\sigma} \frac{1}{\vert H_\sigma \vert}.\]
In particular, if $G$ admits a finite $G$-CW complex model for $\underline{E}G$, we recover the orbifold Euler characteristic of the group $G$
\[\chi_{G}(\underline{E}G)=\frac{\chi^{\Q}(B\Gamma)}{[G:\Gamma]}=\chi_{\orb}(G)  \in \Q.\]
\end{rem} 

The formula in Corollary \ref{cor: Quillen} is difficult to use in general to compute the invariant $\chi_{\orb}^{E_n}(G)=\chi_G^{E_n}(\underline{E}G)$, since the cell structure of $\underline{E}G$ could be hard to make explicit. The following can be thought of as a generalization of \cite[Theorem 5.8]{BenMosche} and is more useful in practice to compute the $E_n$-theory orbifold Euler characteristic:

\begin{thm} \label{thm: centraliser} Let $G$ be a group with finite $\underline{E}G$. Then for any finite proper $G$-CW complex $X$ and $n \geq 1$, we have
\[\chi_G^{E_n}(X)=\sum_{[g_1,\dots,g_n] \in G_{n,p}/G} \chi_{C\langle g_1, \dots, g_n \rangle}(X^{\langle g_1, \dots, g_n \rangle}).\]
In particular,
\[\chi_{\orb}^{E_n}(G)=\sum_{[g_1,\dots,g_n] \in G_{n,p}/G} \chi_{\orb}(C\langle g_1, \dots, g_n \rangle).\]
\end{thm}

\begin{proof} Since both sides of the equation are additive invariants, it suffices to prove the result for $X=G/H$, where $H$ is a finite subgroup. First we observe that for any subgroup $H \leq G$, there is a bijection
\[\coprod_{ ( g_1,\dots,g_n ) \in G_{n,p}^{\mathrm{rep}}} (G/H)^{ \langle g_1,\dots,g_n \rangle}/C\langle g_1,\dots,g_n \rangle \to H_{n,p}/H,\]
sending $[xH] \in (G/H)^{\langle g_1,\dots,g_n \rangle }/C\langle g_1,\dots,g_n \rangle $ to $[x^{-1}g_1x, \dots x^{-1}g_nx]$, where $G_{n,p}^{\mathrm{rep}}$ is a fixed set of representatives. Next, we have (below $C=C \langle g_1,\dots, g_n \rangle $ for short)
\begin{align*}
& \sum_{[g_1,\dots,g_n] \in G_{n,p}/G} \chi_{C\langle g_1,\dots,g_n \rangle }((G/H)^{\langle g_1,\dots,g_n \rangle }) \\ & = \sum_{[g_1,\dots,g_n] \in G_{n,p}/G} \;\;\sum_{[x] \in (G/H)^{ \langle g_1,\dots,g_n \rangle}/C } \frac{1}{\vert C\cap xHx^{-1} \vert} \\ & = \sum_{[h_1,\dots,h_n] \in H_{n,p}/H} \frac{1}{\vert C_H \langle h_1,\dots,h_n \rangle \vert} =  \frac{\vert H_{n,p}\vert} {\vert H \vert}=\chi^{E_n}_G(G/H),\end{align*} 
where we used that $C\cap xHx^{-1}$ is the stabiliser of $[x] \in (G/H)^{ \langle g_1,\dots,g_n \rangle}$ with respect to the action of $C=C \langle g_1,\dots, g_n \rangle $. Here $C_H \langle h_1,\dots,h_n \rangle$ is the centralizer in $H$ and the penultimate identity is the class formula. The last identity is a consequence of Corollary \ref{cor: Quillen}. 
\end{proof}

\begin{rem} \label{rem: congruence} An interesting consequence of Theorem \ref{thm: centraliser} is that for $n \geq 1$, we obtain a congruence
\[\sum_{[g_1,\dots,g_n] \in G_{n,p}/G} \chi_{\orb}(C\langle g_1, \dots, g_n \rangle) \equiv 0 \mod \Z_{(p)}.\]
This is because $\chi_{\orb}^{E_n}(G) \in \Z_{(p)}$. The numbers $\chi_{\orb}(C\langle g_1, \dots, g_n \rangle)$ are in general only rational numbers and this congruence can be used to deduce explicit results in number theory. For an elementary proof of this congruence and applications in number theory see \cite{P24}. 
    
\end{rem}

\begin{rem}

One can examine $\chi^{E}_{G}(X)$ for a general $K(n)$-local $\mathbb{E}_1$-algebra $E$. The following is a consequence of the naturality of the norm and Definition \ref{def: E-theory euler char}. Let $G$ be a group with finite $\underline{E}G$ and $X$ a finite proper $G$-CW complex. Suppose further that $f \colon E \to E'$ is a morphism of $K(n)$-local $\mathbb{E}_1$-algebras.  Then the induced map $f_* \colon \pi_0E \to \pi_0E'$ sends $\chi^E_G(X)$ to $\chi^{E'}_G(X)$. As a consequence we see that in the $K(n)$-local context, $\chi^{L_{K(n)}\mathbb{S}}_G(X)$ plays the role of the universal $G$-equivariant Euler characteristic in the following sense: for any $K(n)$-local $\mathbb{E}_1$-algebra $E$, the element $\chi^E_G(X)$ is the image of $\chi^{L_{K(n)}\mathbb{S}}_G(X)$ under the unit $L_{K(n)}\mathbb{S} \to E$. \end{rem}

We are particularly interested in the case $E=E_n$. Let us temporarily write $\mathbf{1} = L_{K(n)}\mathbb{S}$. The image of the unit map 
\[\pi_0\mathbf{1} \to \pi_0E_n\]
is equal to $\mathbb{Z}_p \subset \pi_0E_n$
and the kernel is equal to the nilradical (see e.g., \cite[Proposition 2.2.6]{CSYambheight}). This implies that $\chi^{\mathbf{1} }_G(X)$ coincides with $\chi^{E_n}_G(X) \in \mathbb{Z}_{(p)} \subset \pi_0\mathbf{1}$ up to nilpotents. (For $G$ finite this is already observed in \cite{Yan, BMCSY}.) In general, when $n \geq 2$, it is difficult to compute $\chi^{\mathbf{1}}_G(X)$ or $\chi^{\mathbf{1} }_{\orb}(G)$ even when $G$ is finite. This is because $\pi_0 \mathbf{1}$ is difficult to describe. At height $n=1$, one can say more. Recall that (see \cite{ravenellocalization})
\[\pi_0L_{K(1)}\mathbb{S} \cong \begin{cases}  \mathbb{Z}_p, & p \;\text{odd}, \\ \Z_2[\varepsilon]/(2\varepsilon, \varepsilon^2), & p=2.  \end{cases}\]
The following is an immediate consequence of the above observations about the unit $\pi_0L_{K(n)}\mathbb{S} \to \pi_0E_n$ \cite[Proposition 2.2.6]{CSYambheight}:
\begin{prop} \label{prop: general K(1)-local card} Let $G$ be a discrete group with finite $\underline{E}G$ and $X$ a finite proper $G$-CW complex.  Then for any odd prime $p$, one has  $\chi^{L_{K(1)}\mathbb{S}}_G(X)=\chi^{E_1}_G(X) \in \mathbb{Z}_{(p)} \subset \Z_p$. For $p=2$, we have
\[\chi^{L_{K(1)}\mathbb{S}}_G(X)=\chi^{E_1}_G(X)+\alpha_G(X)\varepsilon,\]
where $\alpha_G(X) \in \mathbb{Z}/2$ is an additive invariant. 
\end{prop}

When $X=\underline{E}G$, we denote $\alpha_G(X)$ just by $\alpha(G)$. We do not have a concrete formula for the invariant $\alpha_G(X)$. The main result of \cite{CYcard} (Theorem A) shows that 
\[\vert BC_2 \vert_{L_{K(1)}\Ss}=1+\varepsilon \in \pi_0 L_{K(1)}\Ss\]
and hence $\alpha(C_2)=1$. More generally, again using \cite[Theorem A]{CYcard}, it follows that if $G$ is a finite $2$-group, then 
\[\vert BG \vert_{L_{K(1)}\Ss}=1+(\log_2 {\vert G \vert} )\varepsilon \in \pi_0 L_{K(1)}\Ss.\]
This implies for example that $\alpha(C_4)=0$. Hence both possibilities for the invariant $\alpha$ can occur. Even more generally, for any finite $G$, the paper \cite{Li} gives a formula for $\vert BG \vert_{L_{K(1)}\Ss}$ at the prime $p=2$ in terms of the $2$-Sylow subgroups of $G$ (\cite[Corollary 5.2.9]{Li}):
\[\vert BG \vert_{L_{K(1)}\Ss}=\sum_{\{Q_{i_1},\dots,Q_{i_k}\}} (-1)^{k-1}\frac{\vert Q_{i_1} \cap \dots \cap Q_{i_k}  \vert}{\vert G \vert} (1+(\log_2\vert Q_{i_1} \cap \dots \cap Q_{i_k}  \vert)\varepsilon),\]
where $\{Q_1, \cdots Q_l\}$ is the set of 2-Sylow subgroups of $G$ and the sum runs over all its nonempty subsets. 
As a consequence it follows that for any finite $G$, one has
\[\alpha(G)=\sum_{\begin{aligned}\{Q_{i_1},\dots,Q_{i_k}\} &\\ \log_2\vert Q_{i_1} \cap \dots \cap Q_{i_k}  \vert\; &\text{is odd}\end{aligned}} (-1)^{k-1}\frac{\vert Q_{i_1} \cap \dots \cap Q_{i_k}  \vert }{\vert G \vert}.\]
Using \cite[Remark 4.3.19 and Theorem 5.2.5]{Li}, we see that the right-hand side of this formula in fact lives in $\Z_{(2)}$, thus  giving a non-trivial congruence in general. In particular, the identity makes sense modulo $2$. 
We do not have a general formula for $\alpha(G)$ for a group $G$ with a finite $\underline{E}G$, although we are able to calculate $\alpha(G)$ in some concrete non-trivial cases: in Example \ref{ex: Coxeter} below we compute $\alpha(W(L))$ and hence $\chi^{L_{K(1)}\mathbb{S}}_{\orb}(W(L))$ for any right-angled Coxeter group $W(L)$.

\section{Examples} \label{sec: examples}

In this section we apply our results to certain arithmetic groups and right-angled Coxeter groups. Any arithmetic group admits a finite $G$-CW complex model for $\underline{E}G$ by \cite{Ji2007}.  Also, right-angled Coxeter groups admit a finite $G$-CW complex model for $\underline{E}G$ by \cite{Dav}. Hence the results of the previous section apply. 


We first fix the following notation. Let $H$ be a finite abelian subgroup of $G$. Write
\[\Gen_n(H)=\{ (h_1, \dots, h_n) \in H^{n} \; \vert \; \langle h_1, \dots, h_n\rangle=H\}\]
for the set of generating $n$-tuples of $H$. Let $N(H)$ denote the normalizer of $H$ and $C(H)$ the centralizer of $H$ in $G$. The Weyl group $W(H)=N(H)/C(H)$ is finite and acts freely on $\Gen_n(H)$ by conjugation in each coordinate. 

The following is a simple combinatorial observation (see e.g., the proof of \cite[Lemma 4.1]{P24}):

\begin{lem} \label{lem: tuples bijection} Let $G$ be a group. Then there is a natural bijection 
\[\coprod_{(H)} \Gen_n(H)/W(H) \xrightarrow{\cong} G_{n,p}/G,\]
where the disjoint union runs over conjugacy classes of finite abelian $p$-subgroups. 
\end{lem}

As a consequence we obtain: 

\begin{lem} \label{lem: number of tuples} Suppose $G$ is a group such that any abelian $p$-subgroup of $G$ is either trivial or isomorphic to $\Z/p$. Suppose further that there are only finitely many conjugacy classes of elements of order $p$ in $G$. Then
\[\vert G_{n,p}/G \vert = 1+ \frac{p^n-1}{p-1} (\vert G_{1,p}/G \vert-1).\]
\end{lem}

\begin{proof} Below $(H)$ runs over the conjugacy classes of non-trivial abelian $p$-subgroups (all of which are isomorphic to $\Z/p$ by our assumption). Lemma \ref{lem: tuples bijection} yields the identity

\begin{align*}\vert G_{n,p}/G \vert=1+\sum_{(H)} \frac{\vert H \vert^n-1}{\vert N(H)/C(H) \vert}= 1+\sum_{(H)} \frac{\vert H \vert-1}{\vert N(H)/C(H) \vert} \cdot \frac{\vert H \vert^n-1}{\vert H \vert-1}= \\1+ \frac{p^n-1}{p-1}\sum_{(H)}\frac{\vert H \vert-1}{\vert N(H)/C(H) \vert}= 1+ \frac{p^n-1}{p-1} (\vert G_{1,p}/G \vert-1). \end{align*}
\end{proof}

\begin{rem} \label{rem: main thm normalizer} Lemma \ref{lem: tuples bijection} allows us to reformulate the main formula in Theorem \ref{thm: centraliser} as 
\[\chi_{\orb}^{E_n}(G)=\sum_{(H)} \vert \Gen_n(H) \vert \chi_{\orb}(N(H)),\]
where the sum runs over the conjugacy classes of all finite abelian $p$-subgroups. Note that only $p$-subgroups of rank at most $n$ can contribute to this sum. \end{rem}

\begin{ex} \label{ex: symplectic} Let $p$ be an odd prime and consider the symplectic group $\Sp_{p-1}(\Z)$. It follows from \cite[Section 2]{Ash89} and \cite{LevittNicolas98} that any non-trivial $p$-subgroup of $\Sp_{p-1}(\Z)$ is isomorphic to $\Z/p$. The paper \cite{SjerveYang97} shows that the conjugacy classes of elements of order $p$ in $\Sp_{p-1}(\Z)$ are in bijection with the set of equivalence classes of pairs $(I, a)$, where $I$ is an integral ideal in $\Z[\zeta_p]$ and $I \cdot \overline{I}=(a)$, where $a=\overline{a}$. It follows from \cite[Theorem 3]{SjerveYang97} that the number of such equivalence classes is equal to $2^{\frac{p-1}{2}}h_p^{-}$, where 
\[h_p^{-}=\frac{\vert \Cl(\mathbb{Q}(\zeta_p) \vert }{\vert \Cl(\mathbb{Q}(\zeta_p+\zeta_p^{-1}))\vert}\] 
is the relative class number. By \cite[Section 3.2]{Busch2002} and \cite{Brown74}, we know that for any element $A$ of order $p$, the centraliser of $A$ is finite and
\[C\langle A \rangle \cong \Z/p \times \Z/2.\] 
This gives $\chi_{\orb}(C\langle A \rangle )=\frac{1}{2p}$. Additionally, we also know from \cite{Hard} that
\[\chi_{\orb}(\Sp_{p-1}(\Z))=\zeta(-1) \zeta(-3) \cdots \zeta(2-p),\]
where $\zeta$ is the Riemann zeta function. Now using Lemma \ref{lem: number of tuples} and Theorem \ref{thm: centraliser}, we obtain
\begin{align*} \chi_{\orb}^{E_n}(\Sp_{p-1}(\Z)) & =\chi_{\orb}(\Sp_{p-1}(\Z))+ \frac{p^n-1}{p-1}\cdot 2^{\frac{p-1}{2}}h_p^{-} \cdot \chi_{\orb}(\Z/p \times \Z/2) \\ & =\zeta(-1) \zeta(-3) \cdots \zeta(2-p)+ \frac{p^n-1}{p-1}\cdot 2^{\frac{p-1}{2}}h_p^{-} \cdot \frac{1}{2p}\\ & =\zeta(-1) \zeta(-3) \cdots \zeta(2-p)+\frac{p^n-1}{p(p-1)} \cdot 2^{\frac{p-3}{2}}h_p^{-}.\end{align*}
The identity 
\[\chi_{\orb}^{E_n}(\Sp_{p-1}(\Z))=\zeta(-1) \zeta(-3) \cdots \zeta(2-p)+\frac{p^n-1}{p(p-1)} \cdot 2^{\frac{p-3}{2}}h_p^{-}\]
holds in $\Z_{(p)} \subset \Z_{p} \subset \pi_0E_n$ showing that the right-hand side of this equation is $p$-integral. This is a well-known congruence used in number theory to prove Kummer's regularity \cite{Kumregular, Brown74, Brwon75}. 
\end{ex}

\begin{rem} \label{rem: SL2} Using $\zeta(-1)=-\frac{1}{12}$ and $h_3^{-}=1$, it follows from the above formula that at the prime $p=3$ we have: 
\[\chi_{\orb}^{E_n}(\SL_2(\Z))=\frac{3^n-1}{6}-\frac{1}{12}=\frac{2 \cdot 3^{n-1}-1}{4}.\]
This can also be directly computed using Example \ref{ex: amalgamated1}. Indeed, to demonstrate this, we will do a calculation at the prime $p=2$. Let $G=\SL_2(\Z)$ for short. Then Bass--Serre theory implies that $G \cong C_4 \ast_{C_2} C_6$ \cite{Serre}.  Since $\chi^{E_n}_{G}$ is an additive invariant, the pushout in Example \ref{ex: amalgamated1} combined with Propositions \ref{prop: Quillen formula} and \ref{prop: card for finite} yield:
\begin{align*}\chi_{\orb}^{E_n}(\SL_2(\Z)) = \chi_G^{E_n}(G/C_4)+\chi_G^{E_n}(G/C_6)-\chi_G^{E_n}(G/C_2)=\\ \chi_{\orb}^{E_n}(C_4)+\chi_{\orb}^{E_n}(C_6)-\chi_{\orb}^{E_n}(C_2)=\\ \frac{4^n}{4}+\frac{2^n}{6}-\frac{2^n}{2}=\frac{3 \cdot 4^{n-1}-2^n}{3}.  \end{align*}
We leave it for the reader to check that a similar calculation at $p=3$ recovers the first formula above. 

\end{rem}

\begin{ex} \label{ex: general linear} Let $p \geq 5$ be a prime. It follows from Harder's calculation \cite{Hard} that $\chi_{\orb}(\SL_n(\Z))=0$ whenever $n \geq 3$. We also have $2\chi_{\orb}(\GL_n(\Z))=\chi_{\orb}(\SL_n(\Z))=0.$ In this example we offer a chromatic generalization of this fact when $n=p-1$. 

It follows from \cite[Proposition 1.1]{LevittNicolas98} that the only $p$-power order elements in $\GL_{p-1}(\Z)$ are elements of order $p$.
Any elementary abelian $p$-subgroup in $\GL_{p-1}(\Z)$ has rank at most $1$,
see e.g., \cite[Section 2]{Ash89}. It follows from \cite[Lemma 4]{Ash89} that $C \langle A \rangle$ for any $A$ of order $p$ is isomorphic to the group of units in $\Z[\zeta_p]$, which by the Dirichlet unit theorem is isomorphic to
\[\Z/p \times \Z/2 \times \Z^{\frac{p-3}{2}}.\]
The orbifold Euler characteristic of the latter is equal to $0$. Now Theorem \ref{thm: centraliser} implies that 
\[\chi_{\orb}^{E_n}(\GL_{p-1}(\Z))=0.\]
Analogously, one can check that $\chi_{\orb}^{E_n}(\SL_{p-1}(\Z))=0$. We leave the details to the reader (see also \cite[Proposition 5.26]{LPS2024}). 
    
\end{ex}

\begin{ex} \label{ex: totally real} Let $K$ be a totally real field, e.g., $K=\mathbb{Q}(\sqrt{d})$,
where $d$ is any squarefree positive integer. We let $\Oc_K$ denote the ring of integers of $K$.
It follows from \cite[p.\,453]{Hard} that
\[\chi_{\orb}(\SL_2(\Oc_K))=\zeta_K(-1),\]
where $\zeta_K$ is the Dedekind zeta function of $K$. When $K=\mathbb{Q}$ this specializes to the Riemann zeta function $\zeta$ and one recovers the well-known identity from above
\[\chi_{\orb}(\SL_2(\mathbb{Z}))=-\tfrac{1}{12}=\zeta(-1).\] 
Brown \cite[Lemma p.\,251]{Brown74} shows that any finite subgroup $H \leq \SL_2(\Oc_K)$ is cyclic.
Moreover, if $H$ is non-trivial and different from the center $\{\pm \Id\}$  of $\SL_2(\Oc_K)$, then there exists a unique maximal finite subgroup containing $H$. This maximal subgroup is the centralizer $C(H)$ and, moreover, it equals the normalizer $N(H)$. It follows from \cite[Remark 5.18]{LPS2024} that for $p$ an odd prime, one has
\[ \vert \SL_2(\Oc_K)_{n,p} / \SL_2(\Oc_K)  \vert-1=\sum_{(H)} ({\vert H_{(p)} \vert}^n-1).\]
where the sum runs over the conjugacy classes of maximal finite subgroups. Hence by Theorem \ref{thm: centraliser}, we obtain
\begin{align*}\chi_{\orb}^{E_n} (\SL_2(\Oc_K))=\chi_{\orb} (\SL_2(\Oc_K))+\sum_{(H)} ({\vert H_{(p)} \vert}^n-1)\chi_{\orb} (H)\\ =\zeta_K(-1)+\sum_{(H)} \frac{{(\vert H_{(p)} \vert}^n-1)}{\vert H \vert}.\end{align*}
By \cite[Theorem, p.198]{Hirz} and \cite{Prestel}, when $K=\mathbb{Q}(\sqrt{d})$ and $d$ is a square-free positive integer, there are explicit formulas for computing the numbers and types of conjugacy classes of maximal finite subgroups of $\SL_2(\Oc_K)$ in terms of the class numbers of imaginary quadratic fields. We leave the case $p=2$ to the reader (see e.g., \cite[Remark 5.19]{LPS2024}).
    
\end{ex}

\begin{ex} \label{ex: Coxeter} In this example we will work at the prime $p=2$ and apply the main results to right-angled Coxeter groups. This example allows us to compare various orbifold Euler characteristics and cardinalities defined here and in \cite{Yan, CSYambheight}. We will also examine the difference between the orbifold Euler characteristics with respect to $E_1$ and the $K(1)$-local sphere at the prime $2$. In particular, we will compute the invariant $\alpha$ from Proposition \ref{prop: general K(1)-local card} by relying on the computations of \cite{CYcard}.

Let $L$ be a finite graph with the set of vertices $S$ and set of edges $\mathcal{E}$. The right-angled  Coxeter group associated to $L$ is the group
\[W(L)=\langle s \in S \; \vert \; s^2=1, \;\text{for any}\; s \in S\; \text{and}\; (st)^2=1 \;\text{for any} \; \{s,t\} \in \mathcal{E} \rangle. \]
For a reference for Coxeter groups see \cite{Dav}. 
For any subset $T \subset S$, the subgroup generated by $T$ is denoted by $W_T$. This is again a right-angled  Coxeter group associated to the full subgraph spanned by $T$. The subset $T$ is called \emph{spherical} if the subgroup $W_T$ is finite. This is the case if and only if the elements of $T$ commute, i.e., the full subgraph spanned by $T$ is the complete graph. In this case
\[W_T \cong \prod_{\vert T \vert} C_2.\]

We will use a finite model for $\underline{E}W(L)$, known as the \emph{Davis complex} of $W(L)$ \cite[Chapter 7]{Dav}. We denote it by $\Sigma=\Sigma(W(L))$. We do not recall its construction here but we will now give a description of an equivariant cubical cell structure on it. 

In the case when $L$ has only one vertex, we have $W(L)=C_2$. The interval $I=[-1,1]$ has a $C_2$ action using the sign. More generally, given a finite spherical $T \subset S$, the subgroup $W_T=\prod_{\vert T \vert} C_2$ acts on $I_T=\prod_{i=1}^{\vert T \vert} [-1,1]$. This action is the cartesian power of the sign action on $I=[-1,1]$. The $C_2$-space $I=[-1,1]$ is a $C_2$-CW complex with one free $0$-cell $\{-1,1\}$ and one fixed $0$-cell $\{0\}$. It also has one free $1$-cell. This induces a $W_T$-CW structure on $I_T$.

In general, one has a finite increasing filtration $\bigcup_{m \geq 0}\Sigma^m =\Sigma$, where $\Sigma^{\vert S \vert}=\Sigma$ and for any $m \geq 0$, there is a pushout of $W(L)$-spaces (we denote $W(L)$ sometimes by $W$ for short)
\[\xymatrix{\coprod_{ \text{spherical} \;T \subset S, \; \vert T \vert =m} W \times_{W_T} \partial I_T \ar[r] \ar@{^{(}->}[d] & \Sigma^{m-1} \ar@{^{(}->}[d]  \\ \coprod_{ \text{spherical} \;T \subset S, \; \vert T \vert =m} W \times_{W_T} I_T \ar[r] & \Sigma^{m}, }\]
where $\partial I_T$ is the boundary of the cube $I_T$. Now we use Corollary \ref{cor: Quillen} rather than Theorem \ref{thm: centraliser} to compute $\chi_{\orb}^{E_n}(W(L))$. This is because in this case the cellular structure of $\underline{E}W(L)$ is easier to understand than computing the centralizers. Indeed, by Corollary \ref{cor: Quillen}, we have
\[\chi_{C_2}^{E_n}(I)- \chi_{C_2}^{E_n}(\partial I)=2^{n-1}-1,\]
and using the obvious identity $\chi_{G \times H}^{E_n}(X \times Y)=\chi^{E_n}_G(X)\chi^{E_n}_H(Y)$, we obtain
\[\chi_{W_T}^{E_n}(I_T)- \chi_{W_T}^{E_n}(\partial I_T)=(2^{n-1}-1)^{\vert T \vert}.\]
Now Corollary \ref{cor: Quillen}, an inductive argument using the above pushout, and additivity yield the formula
\[\chi_{\orb}^{E_n}(W(L))=\sum_{l=0}^{\vert S \vert} s(l)(2^{n-1}-1)^l,\]
where $s(l)$ is the number of spherical subsets $T \subset S$ of size $l$. We note that when $n=0$, this formula recovers a classical result 
\[\chi_{\orb}(W(L))=\sum_{l=0}^{\vert S \vert } s(l)\cdot \frac{(-1)^l}{2^l},\]
see for example \cite[(16.8)]{Dav}.

An analogous calculation can be done using the cardinalities with respect to the $K(1)$-local sphere. Taking the relations in $\pi_0L_{K(1)}\mathbb{S} \cong \Z_2[\varepsilon]/(2\varepsilon, \varepsilon^2)$ into account, exactly the same strategy as above allows us to compute
\[\chi_{\orb}^{L_{K(1)}\Ss}(W(L))=\sum_{l=0}^{\vert S \vert} s(l)(1+\varepsilon-1)^l=1+s(1)\varepsilon=1+\vert S \vert \varepsilon .\]
Note that up to nilpotence this agrees with $\chi_{\orb}^{E_1}(W(L))=1$. In particular, we have
\[\alpha(W(L))=\vert S \vert. \] 
Depending on the graph $L$ this number of course can be $0$ or $1$. 
\end{ex}

\begin{rem} \label{rem: compare} Let $G$ be a group with finite $\underline{E}G$, $p$ a prime, and $n \geq 1$. Suppose additionally that $G$ does not contain any elements of finite order other than a power of $p$. Then by Lemma \ref{lem:finiteEGandBG}, $BG$ is a finite colimit of $\pi$-finite $p$-spaces (i.e., a $p$-small space in the sense of \cite{Yan}). Additionally, for any finite subgroup $H \leq G$, we have (see \cite[Theorem B, Lemma 4.13]{HKR})
\[\chi_{\orb}^{E_n}(H)=\vert BH \vert_{E_n}=\frac{\vert H_{n,p}\vert}{\vert H \vert }=\vert H_{n-1,p}/H\vert=\chi^{K(n-1)}(BH),\]
where $\chi^{K(l)}$ is the Morava $K$-theory Euler characteristic. Now it follows from additivity (Corollary \ref{cor: Quillen} and \cite[Proposition 2.9]{LPS2024}) that 
\[\chi_{\orb}^{E_n}(G)=\chi^{K(n-1)}(BG).\]
As observed in \cite[Remark 4.10]{LPS2024}, the latter identity also makes sense when $n=0$. In this case $\chi^{K(-1)}(BG)$ is the classical orbifold Euler characteristic $\chi_{\orb}(G)$ which coincides with Yanovski's generalized homotopy cardinality of the $p$-finite space $BG$ for $p$ an odd prime. Following the introductions of \cite{Yan} and \cite{BMCSY}, we summarize that for a discrete group $G$ with a finite $\underline{E}G$ and only $p$-power torsion, the sequences
\[\chi_{\orb}(G), \chi^{\mathbb{Q}}(BG), \chi^{K(1)}(BG), \dots, \chi^{K(n)}(BG), \dots,\]
and 
\[\chi_{\orb}^{E_0}(G), \chi^{E_1}_{\orb}(G), \chi^{E_2}_{\orb}(G), \dots, \chi^{E_n}_{\orb}(G), \dots\]
coincide. The latter coincides up to nilpotents with 
\[\chi_{\orb}^{L_{K(0)}\mathbb{S}}(G), \chi^{L_{K(1)}\mathbb{S}}_{\orb}(G), \chi^{L_{K(2)}\mathbb{S}}_{\orb}(G), \dots, \chi^{L_{K(n)}\mathbb{S}}_{\orb}(G), \dots.\]
By definition $\chi_{\orb}(G)=\chi_{\orb}^{E_0}(G)=\chi_{\orb}^{L_{K(0)}\mathbb{S}}(G)$. The relationship among the higher invariants is nicely demonstrated in the case $p=2$ and $G=W(L)$ in Example \ref{ex: Coxeter}. By \cite[Proposition 5.3]{LPS2024}, 
\[\chi^{K(n-1)}(BW(L))=\sum_{l=0}^{\vert S \vert} s(l)(2^{n-1}-1)^l,\]
which coincides with the formula above for $\chi_{\orb}^{E_n}(W(L))$. Further, the invariants $\chi^{L_{K(1)}\mathbb{S}}_{\orb}(W(L))$ and $\chi_{\orb}^{E_1}(W(L))$ agree up to nilpotents. The concrete difference in this case is given by
\[\chi^{L_{K(1)}\mathbb{S}}_{\orb}(W(L))-\chi_{\orb}^{E_1}(W(L))=\alpha(W(L))\varepsilon=\vert S \vert \varepsilon.\]
    
\end{rem}

\section{Mapping class groups and a chromatic Harer--Zagier theorem} \label{section: MCG and Harer-Zagier}

Let $\Gamma_g^s$ denote the mapping class group of a closed oriented surface of genus $g$ with $s$ marked points. In particular, $\Gamma_g^0$ is the full mapping class group which will be denoted by $\Gamma_g$. By \cite{Broughton90, Har, Mislin2010, JiWolpert2010}, the group $\Gamma_g^s$ is virtually torsion-free and admits a finite $\underline{E}G$. In this section we will compute the Morava $E$-theory orbifold Euler characteristic of $\Gamma^1_g$ at all heights and primes, as well as that of $\Gamma_{\frac{(p-1)(p-2)}{2}}$ for primes $p \geq 5$. Finally, at height $n=1$ we calculate the Morava $E$-theory orbifold Euler characteristic of $\Gamma_g$ for all $g$ and at any prime. Our formulas extend the results of Harer--Zagier \cite{HZ} about the classical orbifold Euler characteristics of mapping class groups. 

By \cite{HZ}, one has 
\[\chi_{\orb}(\Gamma_g^1)=\zeta(1-2g)=-\frac{B_{2g}}{2g}, \; g \geq 1,\]
and
\[\chi_{\orb}(\Gamma_g)=\frac{\zeta(1-2g)}{2-2g}=\frac{B_{2g}}{2g(2g-2)},\; g \geq 2,\]
where $\zeta$ is the Riemann zeta function and $B_{2g}$ is the $2g$-th Bernoulli number. More generally, one has the following formulas from \cite{HZ}:
\[\chi_{\orb}(\Gamma^s_0) = \begin{cases} 1  \;\;\;\; s \leq 3 \\ (-1)^{s-3}(s-3)! \;\;\;\; s \geq 3, 
\end{cases} \]
\[\chi_{\orb}(\Gamma^s_1) = \begin{cases} -\frac{1}{12}  \;\;\;\; s \leq 1 \\ \frac{(-1)^{s}(s-1)!}{12} \;\;\;\; s \geq 1, 
\end{cases} \]
\[\chi_{\orb}(\Gamma^s_g)=(-1)^{s+1}\frac{(2g+s-3)!}{(2g-2)!}\zeta(1-2g) \;\;\;\;  g \geq 2,\; s \geq 0. \]

We begin with the proof of Theorem \ref{thm:introHZ} from the introduction, which states that 
\[
\chi^{E_n}_{\orb}(\Gamma_g^1) = \chi_{\mathrm{orb}}(\Gamma_g^1) + \sum_{m \geq 1, h \geq 0} \frac{|\mathrm{Sur}(\mathbb{Z}^n, \mathbb{Z}/p^m)|}{p^m} \chi_{\mathrm{orb}}(\mathcal{H}_{h}^{1}(\mathbb{Z}/p^m,g)).
\]
Let us first explain the numbers $\chi_{\mathrm{orb}}(\mathcal{H}_{h}^{1}(\mathbb{Z}/p^m,g))$, denoting the orbifold Euler characteristics of appropriate Hurwitz spaces, and then prove the theorem. Afterwards we will make the calculation of the theorem more explicit by providing a closed form for the right-hand side of the above expression in Theorem \ref{thm height n HZ}.

Let $H \leq \Gamma_g^1$ be a finite cyclic subgroup. Pick a generator $\gamma$ of $H$. Then the Nielsen realization \cite{Nielsen} guarantees that $\gamma$ can be represented by an automorphism $f_{\gamma}$ of $\Sigma_g$ of order $|H|$. In particular, associated to $f_{\gamma}$ we find a branched cover $\Sigma_g \to \Sigma_g/f_{\gamma}$, which we will somewhat abusively denote by
\[
\pi_H\colon \Sigma_g \to \Sigma_g/H.
\]
We will write $\Gamma^1(\pi_H)$ for the group of isotopy classes of automorphisms of this cover. This is the group denoted $\Gamma(F/f)$ by Harer--Zagier in \cite[Section 6]{HZ}. In fact, Lemma 6.3 of loc. cit. proves that $\Gamma^1(\pi_H) \cong N(H)/H$, with $N(H)$ the normalizer of $H$ in $\Gamma_g^1$. We now \emph{define} the following expression:
\[
\chi_{\mathrm{orb}}(\mathcal{H}_{h}^{1}(\mathbb{Z}/p^m,g)) := 
\sum_{\substack{
  (H) \\
  \operatorname{gen}(\Sigma_g/H) = h,\ |H| = p^m
}} \chi_{\mathrm{orb}}(\Gamma^1(\pi_H)).
\]
Here the sum is over conjugacy classes of cyclic subgroups $H \leq \Gamma_g^1$ such that $H$ is of order $p^m$ and the surface $\Sigma_g / H$ has genus $h$. This genus is determined by the Riemann--Hurwitz formula
\[
2g-1=p^m(2h-1+s)-l_1-\dots -l_s,
\]
where the cover has branch points $b_0, \ldots, b_s$ (with $b_0$ being the image of the fixed marked point $x_0$) and $l_i$  is the number of preimages of the branch point $b_i$.

\begin{rem}
\label{rem:Hurwitzspace}
Of course the definition of $\chi_{\mathrm{orb}}(\mathcal{H}_{h}^{1}(\mathbb{Z}/p^m,g))$ is motivated by geometry and can be interpreted as the Euler characteristic of an orbifold (although this is not logically necessary for our calculations). Indeed, when we speak of the Hurwitz space $\mathcal{H}_{h}^{1}(\mathbb{Z}/p^m,g)$, we mean the orbifold parametrizing the data of a branched cover $\Sigma_g \to B$ with $B$ of genus $h$, a marked point $x_0 \in \Sigma_g$, and deck group $\mathbb{Z}/p^m$ fixing $x_0$. This space breaks up as a disjoint union of components; indeed, an equivalence class of the specified data is equivalently a conjugacy class of subgroups $H \leq \Gamma_1^g$ with $H \cong \mathbb{Z}/p^m$. Recall that the moduli space of curves $\mathcal{M}^1_g$ can be thought of as the orbifold quotient of Teichm\"uller space $\mathcal{T}^1_g$ by $\Gamma^1_g$. Similarly, one can think of the component of $\mathcal{H}_{h}^{1}(\mathbb{Z}/p^m,g)$ corresponding to a conjugacy class of subgroups $H$ as the orbifold quotient of $(\mathcal{T}^1_g)^H$ by $N(H)/H \cong \Gamma^1(\pi_H)$. In particular, the orbifold Euler characteristic of such a component agrees with that of the group $\Gamma^1(\pi_H)$.
\end{rem}

\begin{proof}[Proof of Theorem \ref{thm:introHZ}]
According to Remark \ref{rem: main thm normalizer} we have
\[
\chi^{E_n}_{\orb}(\Gamma_g^1) = \sum_{(H)} \vert \Gen_n(H) \vert \chi_{\orb}(N(H)),
\]
where the sum runs over the conjugacy classes of finite abelian $p$-subgroups of $\Gamma_g^1$. Each finite subgroup $H \leq \Gamma_g^1$ is in fact cyclic \cite{Lu2001, HZ}. Moreover, the fact that $\Gamma^1(\pi_H)$ is isomorphic to the quotient $N(H)/H$ implies
\[
\chi_{\mathrm{orb}}(N(H)) = \frac{\chi_{\mathrm{orb}}(\Gamma^1(\pi_H))}{|H|}.
\]
Therefore we find
\begin{eqnarray*}
\chi^{E_n}_{\mathrm{orb}}(\Gamma_g^1) & = & \sum_{m \geq 0} \sum_{\substack{
  (H) \\
  |H| = p^m
}} \vert \Gen_n(H) \vert \chi_{\mathrm{orb}}(N(H)) \\
& = & \sum_{m \geq 0, h \geq 0} \frac{\vert \Gen_n(\mathbb{Z}/p^m) \vert}{p^m} \chi_{\mathrm{orb}}(\mathcal{H}_{h}^{1}(\mathbb{Z}/p^m,g)),
\end{eqnarray*}
which, up to rewriting, is the formula of the theorem.
\end{proof}

We will now work towards formulas for $|\Gen_n(H)|$ and $\chi_{\mathrm{orb}}(\mathcal{H}_{h}^{1}(\mathbb{Z}/p^m,g)$, so that the right-hand side of Theorem \ref{thm:introHZ} can be expanded into completely elementary terms and efficiently calculated. The following is straightforward:

\begin{lem}
\label{lem:GennH}
For $m \geq 1$ and $H \cong \mathbb{Z}/p^m$, we have $|\Gen_n(H)| = p^{nm} - p^{n(m-1)}$.
\end{lem}
\begin{proof}
Observe that the set of $n$-tuples of elements of $\mathbb{Z}/p^m$ splits as a disjoint union according to which subgroup an $n$-tuple generates, giving
\[
p^{mn} = \sum_{j=0}^m |\Gen_n(\mathbb{Z}/p^j)|.
\]
The claimed formula follows immediately by induction on $m$.  
\end{proof}

Our next proposition takes some more work, but the essential ideas are already contained in \cite{HZ}:

\begin{prop}
\label{prop:chiHurwitz}
For fixed $g$, $h$, and $m \geq 1$, we have that $\chi_{\mathrm{orb}}(\mathcal{H}_{h}^{1}(\mathbb{Z}/p^m,g))$ equals the sum
\[
\sum_{\tiny{\begin{aligned} s \geq 0 \hspace{2.2cm} \\ (m_1,\dots,m_s), \;\; 0 \leq m_i < m \hspace{1cm}  \\ 2g-1=p^m(2h-1+s)-p^{m_1}-\dots - p^{m_s} \end{aligned}}} p^{2mh} \frac{\chi_{\mathrm{orb}}(\Gamma^{s+1}_h)}{s!} N^1(p^m; p^{m_1}, \ldots, p^{m_s}),
\]
where $N^1(k;l_1,\dots,l_s)$ is the cardinality of the set
\[\{ (r_1, \dots r_s) \in (\Z/k)^{\times s} \;\; \vert \;\; 1+r_1 +\dots+r_s = 0, \; \gcd(r_i,k)=l_i\}.\]
\end{prop}
\begin{proof}
Fix a Riemann surface $B$ of genus $h$, equipped with a distinguished point $b_0$ and a further subset $\{b_1, \ldots, b_s\}$ that will serve as branch points. Write $B_0$ for the punctured surface $B - \{b_0, \ldots, b_s\}$ and denote by $\widehat{\Gamma}^s_h$ the group of isotopy classes of automorphisms of $B_0$ that fix the point $b_0$ and the set $\{b_1, \ldots, b_s\}$. (In other words, elements of this group are allowed to permute the $b_i$ for $i \neq 0$.) We write $\delta_i \in H_1(B_0)$ for the homology class of a small loop around $b_i$. We now consider the associated set $\Omega_B$ of \emph{branching data}, of which the elements are homomorphisms $\alpha\colon H_1(B_0) \to \mathbb{Z}/p^m$ so that $\alpha(\delta_0) = 1$ and the branched cover of $B$ determined by $\alpha$ has genus $g$. Concretely, this means that if we set $l_i = \mathrm{gcd}(\alpha(\delta_i), p^m)$ to be the size of the preimage of $b_i$, then these numbers must satisfy the Riemann--Hurwitz equation:
\[
2g-1=p^m(2h-1+s)-l_1-\dots - l_s.
\]
The set $\Omega_B$ is evidently acted on by $\widehat{\Gamma}^s_h$. We need the following two key facts, explained in \cite[Section 6]{HZ}:
\begin{itemize}
    \item[(1)] Let $H \leq \Gamma_g^1$ be a cyclic subgroup of order $p^m$ with quotient $\Sigma_g/H$ of genus $h$ and $s+1$ branch points (including the fixed base point). Sending $H$ to the branching data $\alpha_H$ determined by the cover  $\Sigma_g \to \Sigma_g/H$ determines a bijection between the set of conjugacy classes of such $H$ and the set of orbits $\Omega_B/\widehat{\Gamma}^s_h$. 
    \item[(2)] For $H$ as above, the stabilizer of $\alpha_H$ under the action of $\widehat{\Gamma}^s_h$ is isomorphic to $N(H)/H \cong \Gamma^1(\pi_H)$.
\end{itemize}
From these two facts we find
\[
\sum_{\substack{
  (H) \\
  \operatorname{gen}(\Sigma_g/H) = h,\ |H| = p^m
}} \chi_{\mathrm{orb}}(N(H)/H) = \sum_{s \geq 0} \sum_{[\alpha] \in \Omega_B/\widehat{\Gamma}^s_h} \chi_{\mathrm{orb}}(\mathrm{stab}({\alpha})).
\]
The stabilizer $\mathrm{stab}(\alpha)$ has finite index in $\widehat{\Gamma}^s_h$, so the right-hand side can be rewritten as $\sum_{s \geq 0} |\Omega_B| \cdot \chi_{\mathrm{orb}}(\widehat{\Gamma}^s_h)$. The group $\Gamma^{s+1}_h$ has index $s!$ in $\widehat{\Gamma}^s_h$, so $\chi_{\mathrm{orb}}(\widehat{\Gamma}^s_h) = \chi_{\mathrm{orb}}(\Gamma^{s+1}_h)/s!$. Thus, it remains to count the elements of the set $\Omega_B$, for fixed $s$. To do this, note that once we fix a tuple $(m_1, \ldots, m_s)$ with $0 \leq m_i < m$ with associated `reciprocal ramification indices' $l_i := p^{m_i}$, satisfying the Riemann--Hurwitz equation as specified above, it remains to count the homomorphisms $\alpha\colon H_1(B_0) \to \mathbb{Z}/p^m$ such that $\mathrm{gcd}(\alpha(\delta_i), p^m) = l_i$. The group $H_1(B_0)$ is isomorphic to the direct sum of $\mathbb{Z}^{2h}$, which is $H_1(B)$, with the abelian group generated by the $\delta_i$ modulo the relation $\sum_i \delta_i = 0$. This shows that
\[
|\Omega_B| = \sum_{\tiny{\begin{aligned} (m_1,\dots,m_s), \;\; 0 \leq m_i < m \hspace{1cm}  \\ 2g-1=p^m(2h-1+s)-p^{m_1}-\dots - p^{m_s} \end{aligned}}} p^{2mh}N^1(p^m; p^{m_1}, \ldots, p^{m_s}),
\]
completing the proof.
\end{proof}
 
Combining Theorem \ref{thm:introHZ} with Lemma \ref{lem:GennH} and Proposition \ref{prop:chiHurwitz} we arrive at the following:
 
\begin{thm} \label{thm height n HZ} Let $g \geq 1$. Then we have
\begin{align*}&\chi_{\orb}^{E_n}(\Gamma_g^1)=\chi_{\orb}(\Gamma_g^1)+\\ &\sum_{\tiny{\begin{aligned} m \geq 1 ,\;\;h \geq 0,\;\;s \geq 0 \hspace{0.9cm} \\ (m_1,\dots,m_s), \;\; 0 \leq m_i < m \;\; \hspace{0.5cm}  \\2g-1=p^m(2h-1+s)-p^{m_1}-\dots -p^{m_s}  \end{aligned}}} p^{2mh+n(m-1)-m}(p^n-1)\frac{\chi_{\orb}(\Gamma_h^{s+1})}{s!}N^1(p^m;p^{m_1},\dots,p^{m_s}).\end{align*}
\end{thm}
\begin{rem}
Using the formulas on p. 481 in \cite{HZ}, we know the values of $N^1(p^m;p^{m_1},\dots,p^{m_s})$ explicitly: 
\[N^1(p^m;p^{m_1},\dots,p^{m_s})=\frac{1}{p^m} \prod_{i=1}^s(p^{m-m_i}-p^{m-m_i-1})\Big (1-\Big(\frac{-1}{p-1}\Big)^{\nu_p} \Big ),\]
where $\nu_p$ is the number of indices $i$ with $m_i=0$. In particular, if $m_i \neq 0$ for all $i$, then $N^1(p^m;p^{m_1},\dots,p^{m_s})=0$. 
\end{rem}

One crucial point in the above calculation is that all finite subgroups of $\Gamma_g^1$ are cyclic. This is not the case for the group $\Gamma_g$ when $g \geq 2$, significantly complicating the calculation of $\chi_{\orb}^{E_n}(\Gamma_g)$. We are unable to provide a general formula for this orbifold Euler characteristic, but we do compute it for some specific values of $g$ in Example \ref{ex: height one full mcg}, and for all $g$ in the special case that the height $n$ is $1$. First, we unravel some special cases of Theorem \ref{thm height n HZ}.

\begin{ex} \label{ex: Gamma1} When the genus $g$ is not too far away from the prime $p$, we can make the formula in Theorem \ref{thm height n HZ} explicit. The main constraint we have on the indices is the Riemann--Hurwitz formula $2g-1=p^m(2h-1+s)-p^{m_1}-\dots -p^{m_s}$ which rules out a lot of possibilities. Below the genus $g$ will always be greater than or equal to $2$. The genus 1 case was covered in Example \ref{rem: SL2} since one has the isomorphisms (see e.g., \cite[Section 2.2.4]{FarbMarg})
\[\Gamma_1^1 \cong \Gamma_1 \cong \SL_2(\mathbb{Z}).\]

To understand the torsion in $\Gamma^1_g$ and $\Gamma_g$ it is useful to look at the short exact sequences (see e.g., \cite[Theorem 4.6 and Section 6.5]{FarbMarg})
\[1 \to \pi_1(\Sigma_g) \to \Gamma^1_g \to \Gamma_g \to 1\]
and
\[1 \to I_g \to \Gamma_g \to \Sp_{2g}(\Z) \to 1,\]
where $\pi_1(\Sigma_g)$ is the fundamental group of the surface $\Sigma_g$ and $I_g$ is the Torelli group. Both of these groups are torsion-free (see e.g., \cite[Theorem 6.8]{FarbMarg}). In particular, if $\gamma \in \Gamma^1_g$ has order $l$, then the homomorphism $\Gamma^1_g \to \Sp_{2g}(\Z)$ sends $\gamma$ to an element of order $l$. By a result of Levitt and Nicolas \cite{LevittNicolas98}, if $\gamma$ has order $p^m$ in $\GL_{2g}(\Z)$, then $2g \geq \varphi(p^m)=p^m-p^{m-1}$. 

Recall again that all finite subgroups of $\Gamma^1_g$ are cyclic \cite{Lu2001}. The first interesting special case is $g=\frac{p-1}{2}$ for $p \geq 5$. By the observation above, the group $\Gamma^1_{\frac{p-1}{2}}$ does not contain a subgroup isomorphic to $\Z/p^2$ and $g=\frac{p-1}{2}$ is the minimal genus where a nontrivial element of order $p$ occurs in $\Gamma^1_g$. The only non-trivial terms occurring in the formula of Theorem \ref{thm height n HZ} are $\chi_{\orb}(\Gamma_g^1)$ and the term indexed by $m=1, s=2, h=0, m_1=m_2=0$. (All other terms are excluded by elementary calculations using the Riemann--Hurwitz formula.) The values of quantities involved in the formula are
\[N^1(p;1,1)=\frac{1}{p}(p-1)^2\Big (1-\frac{1}{(p-1)^2}\Big)=\frac{(p-1)^2-1}{p}\]
and 
\[\frac{\chi_{\orb}(\Gamma_0^3)}{2}=\frac{1}{2}.\]
So all in all, we obtain
\[\chi_{\orb}^{E_n}\Big( \Gamma^1_{\frac{p-1}{2}} \Big)=\zeta(2-p)+(p^n-1)\frac{(p-1)^2-1}{2p^2}=\zeta(2-p)+\frac{(p^n-1)(p-2)}{2p}.\]
For $n \geq 1$, this equation holds in $\Z_{(p)} \subset \pi_0(E_n)$, showing that the right-hand side is $p$-integral. It is easy to see that this congruence is a special case of the von Staudt--Clausen theorem \cite{Staudt, Clausen}, which implies that
\[B_{p-1}+\frac{1}{p} \equiv 0 \mod \Z_{(p)}.\]

A more interesting special case is $g=\frac{p(p-1)}{2}$ for $p \geq 3$. In this case $\Gamma_g^1$ has a subgroup isomorphic to $\Z/p^2$ and the formulas at the heights $n=0,1,2$ are essentially different. We will see below that the formulas at heights $n \geq 1$ capture one additional special value of the Riemann $\zeta$-function. Once again using \cite{LevittNicolas98} we know that $\Gamma_{\frac{p(p-1)}{2}}^1$ does not contain a subgroup isomorphic to $\Z/p^3$ and $g=\frac{p(p-1)}{2}$ is the minimal genus where $p^2$-torsion shows up. 

Using the observation that $N^1(p^m;p^{m_1},\dots,p^{m_s})=0$ unless there is an index $i$ with $m_i=0$ and the Riemann--Hurwitz formula, we see that the only non-trivial terms in the formula of Theorem \ref{thm height n HZ} are indexed by

\rm{(i)} $m=0$;

\rm{(ii)} $m=1$, $s=1$, $h=\frac{p-1}{2}$, $m_1=0$;

\rm{(iii)} $m=1$, $s=p+1$, $h=0$, $m_1=\cdots=m_{p+1}=0$;

\rm{(iv)} $m=2$, $s=2$, $h=0$, $m_1=1$, $m_2=0$;

\rm{(v)} $m=2$, $s=2$, $h=0$, $m_1=0$, $m_2=1$. 

Additionally, we need the following values:
\[N^1(p;1)=\frac{1}{p}(p-1)\Big(1+\frac{1}{p-1}\Big)=1,\]
\[N^1(p; \underbrace{1, \dots, 1}_{p+1})=\frac{1}{p}(p-1)^{p+1}\Big(1-\frac{1}{(p-1)^{p+1}}\Big)=\frac{(p-1)^{p+1}-1}{p}\]
\[N^1(p^2;p,1)=N^1(p^2;1,p)=\frac{1}{p^2}(p^2-p)(p-1)\Big(1+\frac{1}{p-1} \Big)=p-1,\]
and
\[\frac{\chi_{\orb}(\Gamma^2_{\frac{p-1}{2}})}{1!}=(2-p)\zeta(2-p),\]
\[\frac{\chi_{\orb}(\Gamma_0^{p+2})}{(p+1)!}=\frac{1}{p(p+1)},\]
\[\frac{\chi_{\orb}(\Gamma_0^{3})}{2!}=\frac{1}{2}.\]
As a consequence we obtain the formula
\begin{align*}&\chi_{\orb}^{E_n}(\Gamma^1_{\frac{p(p-1)}{2}})=\zeta(1+p-p^2)+\\&(p^n-1)\Big(p^{p-2}(2-p)\zeta(2-p) +\frac{(p-1)^{p+1}-1}{p^3(p+1)}+p^{n-2}(p-1)\Big).\end{align*}
The formula shows that if the height $n$ is larger than $0$, then the formula for $\chi_{\orb}^{E_n}(\Gamma^1_{\frac{p(p-1)} {2}})$ involves two special values of the Riemann zeta function, $\zeta(1+p-p^2)$ and $\zeta(2-p)$. We also observe that the term $p^{n-2}(p-1)$ becomes integral when $n \geq 2$. Additionally, it is not hard to check that the corresponding congruences at height $n \geq 1$ can be obtained using the von Staudt--Clausen theorem and the congruence
\[B_{p(p-1)}+\frac{1}{p} \equiv 1 \mod p\Z_{(p)},\]
which is a special case of a theorem of Carlitz \cite[Theorem 3]{Carl}. 

\end{ex}

\begin{rem} More generally, one can check as in \cite[Remark 6.13]{P24} that if $p \geq 3$ and $p-1 \; \vert \; 2g$, then at height $n \geq 1$ the $p$-integrality of the right-hand side of Theorem \ref{thm height n HZ} (see Remark \ref{rem: congruence}) for $\Gamma_g^1$ is equivalent to the following general form of the congruence described above due to Carlitz \cite[Theorem 3]{Carl}:
\[B_{2g} +\frac{1}{p} \equiv 1 \mod p^r\Z_{(p)},\]
where $2g=zp^r(p-1)$, with $r \geq 0$ and $\gcd(z,p)=1$. We do not go into detail here and refer the interested reader to \cite{P24} to reproduce the argument in this case. For $p-1 \nmid  2g$ one recovers the congruence $\zeta(1-2g) \in \Z_{(p)}$, which is a consequence of Kummer's congruences \cite{Kummercong}. We note though that generally even if all subgroups of $G$ are cyclic, the congruences in Remark \ref{rem: congruence} are not formal consequences of a congruence at a fixed height. 
\end{rem}

\begin{ex} \label{ex: height one full mcg} With our current methods we are not able to give an explicit formula for $\chi_{\orb}^{E_n}(\Gamma_g)$ for general $g$ and $n$, since $\Gamma_g$ may contain any finite group as a subgroup. We will however provide some partial results. We only sketch the details below, as they are analogous to the above computations.  

We are able to calculate $\chi_{\orb}^{E_1}(\Gamma_g)$ using Theorem \ref{thm: centraliser} and \cite[Theorem 6.5]{HZ} (see also \cite[Proposition 6.1]{P24}). Indeed, at height one only the cyclic subgroups of $\Gamma_g$ will contribute and the same reasoning as before yields the following formulas:

\begin{align*}&\chi_{\orb}^{E_1}(\Gamma_g)=\chi_{\orb}(\Gamma_g)+\\ &\sum_{\tiny{\begin{aligned} m \geq 1 ,\;\;h \geq 0,\;\;s \geq 0 \hspace{0.75cm} \\ (m_1,\dots,m_s), \;\; 0 \leq m_i < m, \;\; \forall m_i \neq 0 \hspace{0.5cm}  \\2g-2=p^m(2h-2+s)-p^{m_1}-\dots -p^{m_s}  \end{aligned}}} \frac{1}{p^m}\frac{\chi_{\orb}(\Gamma_h^s)}{s!}p^{2mh}(1-p^{-2h})N(p^m;p^{m_1},\dots,p^{m_s})+\\ & \sum_{\tiny{\begin{aligned} m \geq 1 ,\;\;h \geq 0,\;\;s \geq 1 \hspace{0.75cm} \\  (m_1,\dots,m_s), \;\; 0 \leq m_i < m, \;\; \exists m_i = 0 \hspace{0.5cm}  \\2g-2=p^m(2h-2+s)-p^{m_1}-\dots -p^{m_s}  \end{aligned}}} \frac{1}{p^m}\frac{\chi_{\orb}(\Gamma_h^s)}{s!}p^{2mh}N(p^m;p^{m_1},\dots,p^{m_s}).\end{align*}
Here
\[N(p^m;p^{m_1},\dots,p^{m_s})=\frac{1}{p^m} \prod_{i=1}^s(p^{m-m_i}-p^{m-m_i-1})p^{\lambda_p} \Big (1-\frac{(-1)^{\mu_p-1}}{(p-1)^{\mu_p-1}} \Big ),\] 
where $\lambda_p=\min\{m_1, \dots, m_s, m-1\}$, and $\mu_p$ is the cardinality of the set
\[\{ m_i \; \vert \; p^{\lambda_p+1}  \nmid p^{m_i} \}.\]
In particular, we have
\[N(p; \underbrace{1, \dots, 1}_{s})=\frac{1}{p}(p-1)^s \Big (1-\frac{(-1)^{s-1}}{(p-1)^{s-1}} \Big ).\]
When the genus $g$ is not too far away from $p$, we can sometimes also give a formula for all heights. The first interesting special case is $g=\frac{p-1}{2}$ for $p \geq 5$. The group $\Gamma_{\frac{p-1}{2}}$ does not contain a subgroup isomorphic to $\Z/p^2$ or $\Z/p \times \Z/p$ and $g=\frac{p-1}{2}$ is the minimal genus where nontrivial $p$-torsion in $\Gamma_g$ occurs. This follows from the discussion at the beginning of Example \ref{ex: Gamma1} and \cite[Example 4.1]{Broughton}. The classification of order $p$ elements in $\Gamma_{\frac{p-1}{2}}$ and their normalizers and centralizers is given in \cite[Chapter III]{Xia} and \cite[Example 4.1]{Broughton}. Using Theorem \ref{thm: centraliser}
and Lemma \ref{lem: number of tuples} (or rather Remark \ref{rem: main thm normalizer}), for $p \geq 5$, one then obtains
\[\chi_{\orb}^{E_n}(\Gamma_{\frac{p-1}{2}})=\frac{\zeta(2-p)}{3-p}+ \frac{(p^n-1)(p-2)}{6p}.\]
A more interesting formula can be obtained in the case $g=\frac{(p-1)(p-2)}{2}$. By a result of Broughton \cite[Section 4]{Broughton} the group $\Gamma_{\frac{(p-1)(p-2)}{2}}$ has only one conjugacy class of subgroups isomorphic to $\Z/p \times \Z/p$ and does not have any higher rank elementary abelian $p$-subgroups. The genus $g=\frac{(p-1)(p-2)}{2}$ is the minimal genus for which a rank $2$ elementary abelian $p$-subgroup occurs. Additionally, again by the discussion at the beginning of Example \ref{ex: Gamma1}, the group $\Gamma_{\frac{(p-1)(p-2)}{2}}$ does not contain elements of order $p^2$. The paper \cite{Broughton} also identifies the centralizer and normalizer of the rank $2$ elementary abelian $p$-subgroup. In particular, the normalizer has order $6p^2$. 
We also need the formulas
\[\vert \Gen_n(\Z/p)\vert=p^n-1 \;\text{and}\; \vert \Gen_n(\Z/p\times \Z/p)\vert=p^{2n}-(1+p)p^n+p. \]
\noindent Now using Theorem \ref{thm: centraliser} and Remark \ref{rem: main thm normalizer} (and an argument for order $p$ elements given in the proof of Lemma \ref{lem: number of tuples}), we obtain
\begin{align*}&\chi_{\orb}^{E_n}(\Gamma_{\frac{(p-1)(p-2)}{2}})=
\\&\chi_{\orb}(\Gamma_{\frac{(p-1)(p-2)}{2}})+\frac{p^n-1}{p-1}\sum_{[\gamma]}\chi_{\orb}(C\langle \gamma \rangle)+(p^{2n}-(1+p)p^n+p)\cdot \frac{1}{6p^2},\end{align*}
where the sum runs over the conjugacy classes of order $p$ elements. Using \cite[Theorem 5]{HZ} (or equivalently the enumeration at the beginning of Example \ref{ex: height one full mcg}), we finally obtain 
\begin{align*}&\chi_{\orb}^{E_n}(\Gamma_{\frac{(p-1)(p-2)}{2}})=\frac{\zeta(3p-1-p^2)}{3p-p^2}+\\& \frac{1}{p}\cdot\frac{p^n-1
}{p-1}\Big(\frac{p^{p-1}-1}{3-p}\zeta(2-p) +\frac{(p-1)^{p-1}-1}{p^2(p-2)}\Big)+\frac{p^{2n}-(1+p)p^n+p}{6p^2}.\end{align*}
In more detail, the sum $\sum_{[\gamma]}\chi_{\orb}(C\langle \gamma \rangle)$ is equal to the part with $m=1$ of the big sum at the beginning of Example \ref{ex: height one full mcg}. Since $g=\frac{(p-1)(p-2)}{2}$, the Riemann--Hurwitz formula rules out all possibilities but $s=0$, $h=\frac{p-1}{2}$ and $s=p$, $h=0$. This explains the second term in the latter formula. Observe that this term is zero for $n=0$ and the last term is zero for $n=0,1$. The formula at height $n \geq 1$ recovers a congruence 
\[\frac{B_{2g}}{2g}- \frac{B_{2h}}{2h} \equiv \Big ( \frac{1}{2g}-\frac{1}{2h} \Big )\Big(1-{\frac{1}{p}}\Big) \mod p\Z_{(p)},\]
where $g=\frac{(p-1)(p-2)}{2}$ and $h=\frac{p-1}{2}$ (see \cite[Remark 6.12]{P24}). This is a special case of a congruence due to Cohen \cite[Proposition 11.4.4]{Cohen}, which is classically proved using the theory of $p$-adic $L$-functions.

\end{ex}

%






\part*{Appendix}
\addcontentsline{toc}{part}{Appendix}
\appendix

\section{Coherence for norm maps}
\label{appendix}

In this appendix we prove Lemma \ref{lem:Wirthmuller}, which states a certain coherence property for Wirthm\"uller isomorphisms. Consider the diagram
\[
\mathrm{res}\colon (\mathrm{Or}^G_{\mathrm{fin}})^{\mathrm{op}} \to \mathrm{Cat}_\infty\colon G/H \mapsto \Sp^H,
\]
with functoriality given by restrictions. The main results of \cite{CLL6functor} (applied in the case where $I = P = E$ consists of all morphisms in $\mathrm{Or}^G_{\mathrm{fin}}$) imply that this functor can be extended to a `six-functor formalism'. Let us explain precisely what aspect of this we need here. We write $\mathrm{Span}(\mathrm{Or}^G_{\mathrm{fin}})$ for the 2-category of spans in the category $\mathrm{Or}^G_{\mathrm{fin}}$. Its objects agree with those of $\mathrm{Or}^G_{\mathrm{fin}}$, the morphisms are given by spans
\[
G/H_0 \leftarrow G/K \rightarrow G/H_1
\]
and the 2-morphisms are spans of spans:
\[
\begin{tikzcd}
& G/K_0 \ar{dr}\ar{dl} & \\
G/H_0 & G/L \ar{u}\ar{d} & G/H_1. \\
& G/K_1 \ar{ur}\ar{ul} &
\end{tikzcd}
\]
There is an evident inclusion functor
\[
j\colon (\mathrm{Or}^G_{\mathrm{fin}})^{\mathrm{op}} \to \mathrm{Span}(\mathrm{Or}^G_{\mathrm{fin}})
\]
that is the identity on objects and sends a morphism $G/H \to G/K$ in $\mathrm{Or}^G_{\mathrm{fin}}$ to the span
\[
G/K \leftarrow G/H \xrightarrow{=} G/H.
\]



\begin{proof}[Proof of Lemma \ref{lem:Wirthmuller}]
Conjecture 2.24 of \cite{BenMosche}, proved in \cite{CLL6functor} as a consequence of Theorem A in loc. cit., implies that there is an essentially unique functor of $(\infty,2)$-categories
\[
R\colon \mathrm{Span}(\mathrm{Or}^G_{\mathrm{fin}}) \to \mathrm{Cat}_\infty
\]
such that the composition 
\[
(\mathrm{Or}^G_{\mathrm{fin}})^{\mathrm{op}} \xrightarrow{j} \mathrm{Span}(\mathrm{Or}^G_{\mathrm{fin}}) \xrightarrow{R} \mathrm{Cat}_\infty
\]
is the functor $\mathrm{res}$ with which we started this section. Therefore it suffices to observe that the left and right adjoints of the morphisms in $(\mathrm{Or}^G_{\mathrm{fin}})^{\mathrm{op}}$ already exist in the 2-category $\mathrm{Span}(\mathrm{Or}^G_{\mathrm{fin}})$ and that they agree coherently there. (Indeed, any functor of $(\infty,2)$-categories preserves adjunctions.) Note that the functor
\[
\mathrm{Or}^G_{\mathrm{fin}} \to \mathrm{Span}(\mathrm{Or}^G_{\mathrm{fin}})\colon (G/H \to G/K) \mapsto (G/H \xleftarrow{=} G/H \rightarrow G/K)
\]
provides a diagram that serves both as a system of left adjoints and as a system of right adjoints to the morphisms of the diagram $j$, completing the argument.
\end{proof}

\section{Relation to parametrized homotopy theory}
\label{app:parametrized}

In this appendix we compare our duality from Section \ref{sec: genuine dualizing complex} to the duality considered in \cite{HKK} and the Costenoble--Waner duality considered in \cite{Cno23}, \cite{MaySig}, and \cite{CosWan}. 

Let $G$ be a finite group. The $\infty$-category of $G$-spaces \[\mathcal{S}^G := \mathrm{Fun}((\mathrm{Or}^G)^{\mathrm{op}}, \mathcal{S})\]
is a presheaf category and hence an $\infty$-topos. Given a $G$-CW complex $X$, one can consider the \emph{parametrized $\infty$-category of $G$-spectra over $X$} as in \cite{BDGNS, Nardin, Shah} and in \cite{Martini, MW1, MW2, CLL1}. This is an $\infty$-categeory parametrized over the presheaf topos $\mathcal{S}^G$. We will only consider the global sections $\infty$-category $\Sp_X^G$ of this sheaf, which can be also constructed explicitly as follows. The functor
\[\underline{\Sp^G} \colon (\mathrm{Or}^G)^{\mathrm{op}} \to \Cat_{\infty},\]
sending the orbit $G/H$ to the $\infty$-category of genuine $H$-spectra $\Sp^H$ and morphisms in $\mathrm{Or}^G$ to restrictions, can be right Kan extended to a limit-preserving functor (sheaf)
\[\widehat{\underline{\Sp^G}} \colon (\mathcal{S}^G)^{\mathrm{op}} \to \Cat_{\infty}.\]
The value of this functor at $X$ is by definition $\Sp_X^G$. More informally, we can write
\[\Sp_X^G \simeq \lim_{G/H \to X} \Sp^H.\]
In particular, any colimit decomposition of $X$ in terms of representables gives an inverse limit formula for $\Sp_X^G$. Following \cite[Section 4]{Cno23} and \cite[Section 4.1]{HKK}, the restriction
\[X^* \colon \Sp^G \to \Sp_X^G\]
along $X \to \ast$ has a left adjoint $X_!$ and a right adjoint $X_*$. Moreover, analogously to Definition \ref{def:dualizingcomplex} and Definition \ref{def: relative norm}, one can define the (genuine) dualizing complex $D_X \in \Sp_X^G$ and the norm map in $\Sp^G$:
\[\mathrm{Nm}_X \colon X_{!}(F \otimes_X D_X) \to X_*(F).\]
The parametrized twisted ambidexterity (\cite[Theorem 4.8(5)]{Cno23} and \cite[Proposition 4.16]{HKK}) tells us that if $X$ is a finite $G$-CW complex, then this norm map is an equivalence. We will now explain how this duality is related to the results of our paper.

Let $G$ be a discrete group and $\Gamma \leq G$ a finite index torsion-free subgroup. Without loss of generality we can assume that $\Gamma$ is normal. The following is the key result that relates proper $G$-equivariant stable homotopy theory with parametrized $G/\Gamma$-equivariant stable homotopy theory. 

\begin{prop} \label{prop: proper vs parametrized} Let $G$ be a group and $\Gamma \leq G$ a finite index torsion-free normal subgroup. Choose any model for $\underline{E}G$. Then the $\infty$-category $\Sp^G_{\mathrm{prop}}$ is equivalent to the $\infty$-category $\Sp_{(\underline{E}G)/\Gamma}^{G/\Gamma}$. 
\end{prop}

\begin{proof} By Lemma \ref{lem:underbarEGcolim}, the $G/\Gamma$-space $(\underline{E}G)/\Gamma$ decomposes as a colimit 
\begin{align*}&\colim_{G/H \in \mathrm{Or}^G_{\mathrm{fin}}}  (G/\Gamma)/(H\Gamma/\Gamma)  \simeq \colim_{G/H \in \mathrm{Or}^G_{\mathrm{fin}}} G/H\Gamma \simeq  \\& \colim_{G/H \in \mathrm{Or}^G_{\mathrm{fin}}} (G/H)/\Gamma \simeq (\underline{E}G)/\Gamma.\end{align*}
Now applying the sheaf 
\[\widehat{\underline{\Sp^{G/\Gamma}}} \colon (\mathcal{S}^{G/\Gamma})^{\mathrm{op}} \to \Cat_{\infty}\]
to this colimit decomposition, we obtain a chain of equivalences
\[\Sp_{(\underline{E}G)/\Gamma}^{G/\Gamma} \simeq \lim_{G/H \in \mathrm{Or}^G_{\mathrm{fin}}} \Sp_{H\Gamma/\Gamma} \simeq \lim_{G/H \in \mathrm{Or}^G_{\mathrm{fin}}} \Sp_{H} \simeq \Sp_{\mathrm{prop}}^G.\]
Here the second equivalence comes from the fact that $\Gamma$ is torsion-free and hence $H\Gamma/\Gamma \cong H/H \cap \Gamma=H$ for any finite subgroup $H$. The last equivalence is \cite[Theorem 12.11]{LNP}. 
\end{proof}

For convenience, we now let $X$ denote the $G/\Gamma$-space $(\underline{E}G)/\Gamma$. Then the functors $X^*, X_*, X_!$ can be described in terms of proper equivariant stable homotopy theory by means of Proposition \ref{prop: proper vs parametrized}. Since $\Gamma \leq G$ is torsion-free, for any finite subgroup $H \leq G$, the composite
\[H \xrightarrow{\iota} G \xrightarrow{\pi} G/\Gamma,\]
where $\iota$ is the inclusion and $\pi$ is the projection, is injective. Hence we get a restriction functor
\[(\pi \iota)^*=\res^{G/\Gamma}_H \colon \Sp^{G/\Gamma} \to \Sp^H.\]
These restriction functors are compatible when $G/H$ varies over $\mathrm{Or}^G_{\mathrm{fin}}$ and hence induce a functor
\[\pi^* \colon \Sp^{G/\Gamma} \to \Sp^G_{\mathrm{prop}},\]
called the \emph{inflation}. A diagram chase shows that this functor corresponds to $X^*$ under the equivalence of Proposition \ref{prop: proper vs parametrized}. 
Furthermore, since the $\infty$-categories involved are presentable and $\pi^*=\res^{G/\Gamma}_H$ preserves both limits and colimits, we get left and right adjoints
\[\pi_!, \pi_* \colon \Sp^G_{\mathrm{prop}} \to  \Sp^{G/\Gamma},\]
respectively. One can think of $\pi_!$ as the \emph{genuine $\Gamma$-orbits} and $\pi_*$ as the \emph{genuine $\Gamma$-fixed points}. These correspond to $X_!$ and $X_*$, respectively, under the equivalence of Proposition \ref{prop: proper vs parametrized}.

\begin{rem} \label{rem: inflation} The adjunction
\[\xymatrix{ \Sp^{G/\Gamma}_{\mathrm{prop}} \ar@<0.5ex>[r]^-{\pi^*} & \Sp^{G}_{\mathrm{prop}} \ar@<0.5ex>[l]^-{\pi_*}}\]
exists for any normal subgroup $\Gamma \leq G$. This is because even if the composite $H \xrightarrow{\iota} G \xrightarrow{\pi} G/\Gamma$ is not injective in general, the inflation 
\[(\pi \iota)^*=\res^{G/\Gamma}_H \colon \Sp^{G/\Gamma}_{\mathrm{prop}} \to \Sp^H\]
still exists and preserves colimits. However, this functor does not necessarily preserve limits. Hence in general $\pi_{!} \colon \Sp^{G}_{\mathrm{prop}} \to \Sp^{G/\Gamma}_{\mathrm{prop}}$ might not exist. As we saw above, if $\Gamma$ is torsion-free, then $\pi_{!}$ does exist. 
\end{rem}

\begin{rem} \label{rem: new D} The inflation and genuine fixed points functor from the previous remark allow us to lift Proposition \ref{prop: dualizing module} to genuine equivariant homotopy theory. Indeed, let $\pi_1, \pi_2 \colon G \times G \to G$ denote the two projections. 
Then the duality functor $D \colon \Sp^G_{\mathrm{prop}} \to (\Sp^G_{\mathrm{prop}})^{\mathrm{op}}$ can be described as follows:
\[D(Y) \simeq {\pi_1}_*\mathrm{map}(\pi_2^*Y, \mathbb{S}[G]),\]
where $\mathbb{S}[G] =(\Delta_G)_{!}(\mathbb{S})$ is considered as a proper $G\times G$-spectrum with the usual action, and the mapping object is internal to $\Sp^{G \times G}_{\mathrm{prop}}$. Here $\Delta_G \colon G \to G \times G$ is the diagonal and $(\Delta_G)_{!}$ is left adjoint to the restriction $(\Delta_G)^* \colon \Sp^{G \times G}_{\mathrm{prop}} \to \Sp^{G}_{\mathrm{prop}}$. We do not need this description of the duality functor $D$ in this paper and leave the comparison to future work. The idea of the proof is similar to that of Proposition \ref{prop: dualizing module}, which is the underlying comparison in $\Fun(BG, \Sp)$. 
\end{rem}

Now we are ready to state the comparison result advertised at the beginning of this appendix. 

\begin{prop} \label{prop: app B} Let $G$ be a group with finite $\underline{E}G$ and $\Gamma \leq G$ a finite index torsion-free normal subgroup. Let $X$ denote the $G/\Gamma$-space $(\underline{E}G)/\Gamma$. Then under the equivalence $\Sp^G_{\mathrm{prop}} \simeq \Sp_{X}^{G/\Gamma}$, the duality functor $D$ on $\Sp^G_{\mathrm{prop}}$ corresponds to the Costenoble--Waner duality $D_{CW}$ \cite{CosWan, MaySig, Cno23, HKK} on $\Sp_{X}^{G/\Gamma}$. The functor $D_{CW}$ is given by
\[D_{CW}(F)={\pi_1}_*\mathrm{map_{X\times X}}({\pi_2}^*F, (\Delta_X)_!(\mathbb{S}_X)),\]
where $\pi_1, \pi_2 \colon X \times X \to X$ denote the projections (slightly abusing the notation), $\Delta_X \colon X \to X \times X$ is the diagonal, and $map_{X \times X}$ is the internal mapping object of  $\Sp_{X \times X}^{G/\Gamma}$. In particular, the dualizing complex $D_{X} \in \Sp_{X}^{G/\Gamma}$ of \cite{Cno23, HKK} and the dual of the sphere $D(\mathbb{S}_G) \in \Sp^G_{\mathrm{prop}}$ correspond to each other. \end{prop}

\begin{proof} This follows by going through relevant definitions from \cite{Cno23} and \cite{HKK} and chasing through the equivalence of Proposition \ref{prop: proper vs parametrized}. The crucial part is to observe that the diagram 
\[\xymatrix{\Sp_{X}^{G/\Gamma} \ar[d]^{\simeq} \ar[r]^{(\Delta_X)_{!}} & \Sp_{X \times X}^{G/\Gamma} \ar[d]^{(\Delta_{G/\Gamma})_{!}}   \ar[r]^{{\pi_1}_*} & \Sp_{X}^{G/\Gamma} \ar[d]^{\simeq}  \\ \Sp^G_{\mathrm{prop}} \ar[r]^{(\Delta_G)_{!}} & \Sp^{G \times G}_{\mathrm{prop}}  \ar[r]^{{\pi_1}_*} & \Sp^G_{\mathrm{prop}}}\]
commutes. Here $(\Delta_{G/\Gamma})_{!}$ denotes the composite
\[\Sp_{X \times X}^{G/\Gamma} \xrightarrow{(\Delta_{G/\Gamma})_{!}} \Sp_{X \times X}^{G/\Gamma \times G/\Gamma } \simeq \Sp^{G \times G}_{\mathrm{prop}}. \]
The left square commutes by definition. The right square commutes by a version of the parametrized Wirthm\"uller isomorphism, which implies that the left adjoint $(\Delta_{G/\Gamma})_{!}$ is naturally isomorphic to the right adjoint $(\Delta_{G/\Gamma})_{*}$ since $G/\Gamma$ is finite \cite[Theorem 19.5.2]{MaySig}. 
\end{proof}

\begin{rem} Suppose that the conditions of Proposition \ref{prop: app B} are satisfied. The twisted ambidexterity result of \cite{Cno23} and \cite{HKK} applied to $X=\underline{E}G/\Gamma$ implies that the norm gives an equivalence of genuine $G/\Gamma$-spectra
\[\mathrm{Nm}_{X} \colon X_{!} (F \otimes_X D_X) \xrightarrow{\simeq} X_*(F).\]
Equivalently, this implies that we have an equivalence
\[\pi_{!}(Y \otimes D(\mathbb{S}_G)) \xrightarrow{\simeq} \pi_*(Y).\]
By writing this in terms of orbits and fixed points, we get an equivalence
\[(Y \otimes D(\mathbb{S}_G))/\Gamma \xrightarrow{\simeq} Y^{\Gamma}.\]
The latter can be constructed directly using observations similar to those in Remark \ref{rem: norm limit}. More generally, one has an equivalence
\[(Y \otimes D(Z))/\Gamma \xrightarrow{\simeq} \mathrm{map}(Z,Y)^{\Gamma},\]
for any $Z \in (\Sp^G_{\mathrm{prop}})^{\omega}$. This gives a unit map $\mathbb{S}_{G/\Gamma} \to (Z \otimes D(Z))/\Gamma$ in $\Sp^{G/\Gamma}$ for a compact $Z$, witnessing the Costenoble--Waner dualizability.  In particular, the map $\mathbb{S}_{G/\Gamma} \to (D(\mathbb{S}_G))/\Gamma$ recovers a Spivak datum in the sense of \cite{HKK}. In general, we do not expect to have such a unit map for the duality functor $D$. Our duality is not a Costenoble--Waner type duality in general. 
    
\end{rem}

\bibliographystyle{alpha}
\bibliography{bib}

@article{AndrewPatchkoria,
    AUTHOR = {Andrew, N. and Patchkoria, I.},
     TITLE = {On the {F}arrell--{T}ate {$K$}-theory of ${O}ut({F}_n)$},
      YEAR = {2024}, 
Note ={arXiv:2505.21803, to appear in Forum of Mathematics, Sigma},
}

@article {Ash89,
    AUTHOR = {Ash, A.},
     TITLE = {Farrell cohomology of {${\rm GL}(n,{\bf Z})$}},
   JOURNAL = {Israel J. Math.},
  FJOURNAL = {Israel Journal of Mathematics},
    VOLUME = {67},
      YEAR = {1989},
    NUMBER = {3},
     PAGES = {327--336},
      ISSN = {0021-2172},
   MRCLASS = {11E72 (20J05)},
}

@article {BDP,
    AUTHOR = {B\'arcenas, N. and Degrijse, D. and Patchkoria,
              I.},
     TITLE = {Stable finiteness properties of infinite discrete groups},
   JOURNAL = {J. Topol.},
  FJOURNAL = {Journal of Topology},
    VOLUME = {10},
      YEAR = {2017},
    NUMBER = {4},
     PAGES = {1169--1196},
      ISSN = {1753-8416,1753-8424},
   MRCLASS = {55P91 (20J05 55P42)},
  MRNUMBER = {3743073},
MRREVIEWER = {Samik\ Basu},
       DOI = {10.1112/topo.12035},
       URL = {https://doi.org/10.1112/topo.12035},
}

@article{BenMosche,
    AUTHOR = {Ben-Moshe, S.},
     TITLE = {Higher semiadditivity in transchromatic homotopy theory},
      YEAR = {2024}, 
JOURNAL={arXiv:2411.00968},
}

@article {BMCSY,
    AUTHOR = {Ben-Moshe, S. and Carmeli, S. and Schlank, T. M. and
              Yanovski, L.},
     TITLE = {Chromatic cardinalities via redshift},
   JOURNAL = {Int. Math. Res. Not. IMRN},
  FJOURNAL = {International Mathematics Research Notices. IMRN},
      YEAR = {2024},
    NUMBER = {14},
     PAGES = {10918--10924},
      ISSN = {1073-7928,1687-0247},
   MRCLASS = {18F25 (19D99 55P43)},
  MRNUMBER = {4776197},
MRREVIEWER = {Markus\ Szymik},
       DOI = {10.1093/imrn/rnae109},
       URL = {https://doi.org/10.1093/imrn/rnae109},
}

@article {BestvinaBrady,
    AUTHOR = {Bestvina, M. and Brady, N.},
     TITLE = {Morse theory and finiteness properties of groups},
   JOURNAL = {Invent. Math.},
  FJOURNAL = {Inventiones Mathematicae},
    VOLUME = {129},
      YEAR = {1997},
    NUMBER = {3},
     PAGES = {445--470},
      ISSN = {0020-9910,1432-1297},
   MRCLASS = {20F36 (20J05 57M07)},
  MRNUMBER = {1465330},
MRREVIEWER = {John\ Meier},
       DOI = {10.1007/s002220050168},
       URL = {https://doi.org/10.1007/s002220050168},
}

@book {Bieri,
    AUTHOR = {Bieri, R.},
     TITLE = {Homological dimension of discrete groups},
    SERIES = {Queen Mary College Mathematics Notes},
   EDITION = {Second},
 PUBLISHER = {Queen Mary College, Department of Pure Mathematics, London},
      YEAR = {1981},
     PAGES = {iv+198},
   MRCLASS = {20J05 (18G20 57P10)},
}

@article {BieEck,
    AUTHOR = {Bieri, R. and Eckmann, B.},
     TITLE = {Groups with homological duality generalizing {P}oincar\'e{}
              duality},
   JOURNAL = {Invent. Math.},
  FJOURNAL = {Inventiones Mathematicae},
    VOLUME = {20},
      YEAR = {1973},
     PAGES = {103--124},
      ISSN = {0020-9910,1432-1297},
   MRCLASS = {20J05},
  MRNUMBER = {340449},
MRREVIEWER = {L.\ Ribes},
       DOI = {10.1007/BF01404060},
       URL = {https://doi.org/10.1007/BF01404060},
}

@article {BorelSerre,
    AUTHOR = {Borel, A. and Serre, J.-P.},
     TITLE = {Corners and arithmetic groups},
   JOURNAL = {Comment. Math. Helv.},
  FJOURNAL = {Commentarii Mathematici Helvetici},
    VOLUME = {48},
      YEAR = {1973},
     PAGES = {436--491},
      ISSN = {0010-2571,1420-8946},
   MRCLASS = {22E40},
  MRNUMBER = {387495},
MRREVIEWER = {M.\ S.\ Raghunathan},
       DOI = {10.1007/BF02566134},
       URL = {https://doi.org/10.1007/BF02566134},
}

@article {Broughton90,
    AUTHOR = {Broughton, S. A.},
     TITLE = {The equisymmetric stratification of the moduli space and the
              {K}rull dimension of mapping class groups},
   JOURNAL = {Topology Appl.},
  FJOURNAL = {Topology and its Applications},
    VOLUME = {37},
      YEAR = {1990},
    NUMBER = {2},
     PAGES = {101--113},
      ISSN = {0166-8641,1879-3207},
   MRCLASS = {57N05 (16P60 32G15)},
}

@incollection {Broughton,
    AUTHOR = {Broughton, S. A.},
     TITLE = {Normalizers and centralizers of elementary abelian subgroups
              of the mapping class group},
 BOOKTITLE = {Topology '90 ({C}olumbus, {OH}, 1990)},
    SERIES = {Ohio State Univ. Math. Res. Inst. Publ.},
    VOLUME = {1},
     PAGES = {77--89},
 PUBLISHER = {de Gruyter, Berlin},
      YEAR = {1992},
      ISBN = {3-11-012598-6},
   MRCLASS = {57M07 (20F32)},}

@article {Brwon75,
    AUTHOR = {Brown, K. S.},
     TITLE = {Euler characteristics of groups: the {$p$}-fractional part},
   JOURNAL = {Invent. Math.},
  FJOURNAL = {Inventiones Mathematicae},
    VOLUME = {29},
      YEAR = {1975},
    NUMBER = {1},
     PAGES = {1--5},
      ISSN = {0020-9910,1432-1297},
   MRCLASS = {22E40 (12A70 20G10)},}

@book{Brown,
    AUTHOR = {Brown, K. S.},
     TITLE = {Cohomology of groups},
    SERIES = {Graduate Texts in Mathematics},
    VOLUME = {87},
      NOTE = {Corrected reprint of the 1982 original},
 PUBLISHER = {Springer-Verlag, New York},
      YEAR = {1994},
     PAGES = {x+306},
      ISBN = {0-387-90688-6},
   MRCLASS = {20J05 (20-02)},
  MRNUMBER = {1324339},
}

@article {Brown82,
    AUTHOR = {Brown, K. S.},
     TITLE = {Complete {E}uler characteristics and fixed-point theory},
   JOURNAL = {J. Pure Appl. Algebra},
  FJOURNAL = {Journal of Pure and Applied Algebra},
    VOLUME = {24},
      YEAR = {1982},
    NUMBER = {2},
     PAGES = {103--121},
      ISSN = {0022-4049},
   MRCLASS = {20J05 (18G99 55M20)},}

@article {Busch2002,
    AUTHOR = {Busch, C.},
     TITLE = {The {F}arrell cohomology of {${\rm Sp}(p-1,\Bbb Z)$}},
   JOURNAL = {Doc. Math.},
  FJOURNAL = {Documenta Mathematica},
    VOLUME = {7},
      YEAR = {2002},
     PAGES = {239--254},
      ISSN = {1431-0635,1431-0643},
   MRCLASS = {20G10},
  MRNUMBER = {1938122},
MRREVIEWER = {Daniel\ K.\ Nakano},
}

@article {CYcard,
    AUTHOR = {Carmeli, S. and Yuan, A.},
     TITLE = {Higher semiadditive {G}rothendieck-{W}itt theory and the
              {$K(1)$}-local sphere},
   JOURNAL = {Commun. Am. Math. Soc.},
  FJOURNAL = {Communications of the American Mathematical Society},
    VOLUME = {3},
      YEAR = {2023},
     PAGES = {65--111},
      ISSN = {2692-3688},
   MRCLASS = {55P42 (11E70)},
  MRNUMBER = {4564543},
       DOI = {10.1090/cams/17},
       URL = {https://doi.org/10.1090/cams/17},
}

@article {Carl,
    AUTHOR = {Carlitz, L.},
     TITLE = {Some congruences for the {B}ernoulli numbers},
   JOURNAL = {Amer. J. Math.},
  FJOURNAL = {American Journal of Mathematics},
    VOLUME = {75},
      YEAR = {1953},
     PAGES = {163--172},
      ISSN = {0002-9327,1080-6377},
   MRCLASS = {10.0X},}

@unpublished {CCRY,
    AUTHOR = {Carmeli, S. and Cnossen, B. and Ramzi, M. and Yanovski, L.},
     TITLE = {Characters and transfer maps via categorified traces},
YEAR = {2025},
Note = {arXiv:2210.17364v1, to appear in Forum of Mathematics, Sigma},
}

@article {Clausen,
    AUTHOR = {Clausen, T.},
     TITLE = {Theorem},
   JOURNAL = {Astron. Nachr.},
  FJOURNAL = { Astronomische Nachrichten},
    VOLUME = {17},
      YEAR = {1840},
     PAGES = {351--352},
     NUMBER = {22},
      ISSN = {1465-3060},}

@article{Cno23,
    AUTHOR = {Cnossen, B.},
     TITLE = {Twisted ambidexterity in equivariant homotopy theory},
      YEAR = {2023}, 
note={arXiv:2303.00736, to appear in Journal of Topology},
}

@article {CLL,
    AUTHOR = {Cnossen, B. and Lenz, T. and Linskens, S.},
     TITLE = {The {A}dams isomorphism revisited},
   JOURNAL = {Math. Z.},
  FJOURNAL = {Mathematische Zeitschrift},
    VOLUME = {308},
      YEAR = {2024},
    NUMBER = {2},
     PAGES = {Paper No. 33, 32},
      ISSN = {0025-5874,1432-1823},
   MRCLASS = {18N60 (18A25 18N40 19K35 19L20 55P91)},
  MRNUMBER = {4798463},
MRREVIEWER = {Daniel\ Dugger},
       DOI = {10.1007/s00209-024-03582-w},
       URL = {https://doi.org/10.1007/s00209-024-03582-w},
}

@article {CLL1,
    AUTHOR = {Cnossen, B. and Lenz, T. and Linskens, S.},
     TITLE = {Parametrized stability and the universal property of global
              spectra},
   JOURNAL = {J. Topol.},
  FJOURNAL = {Journal of Topology},
    VOLUME = {18},
      YEAR = {2025},
    NUMBER = {4},
     PAGES = {Paper No. e70044, 119},
      ISSN = {1753-8416,1753-8424},
   MRCLASS = {55U35 (55P42 55P91)},
  MRNUMBER = {4990070},
MRREVIEWER = {David\ A.\ Blanc},
       DOI = {10.1112/topo.70044},
       URL = {https://doi.org/10.1112/topo.70044},
}

@book {Cohen,
    AUTHOR = {Cohen, H.},
     TITLE = {Number theory. {V}ol. {II}. {A}nalytic and modern tools},
    SERIES = {Graduate Texts in Mathematics},
    VOLUME = {240},
 PUBLISHER = {Springer, New York},
      YEAR = {2007},
     PAGES = {xxiv+596},
      ISBN = {978-0-387-49893-5},
   MRCLASS = {11-01 (11D61 11F80 11J86 11Mxx)},}

@book {CosWan,
    AUTHOR = {Costenoble, S. R. and Waner, S.},
     TITLE = {Equivariant ordinary homology and cohomology},
    SERIES = {Lecture Notes in Mathematics},
    VOLUME = {2178},
 PUBLISHER = {Springer, Cham},
      YEAR = {2016},
     PAGES = {xiv+294},
      ISBN = {978-3-319-50447-6; 978-3-319-50448-3},
   MRCLASS = {55N91 (55N25 55P42 55P91 55R91 57R91)},
  MRNUMBER = {3585352},
MRREVIEWER = {J.\ P. C. Greenlees},
       DOI = {10.1007/978-3-319-50448-3},
       URL = {https://doi.org/10.1007/978-3-319-50448-3},
}

@book {Dav,
    AUTHOR = {Davis, M. W.},
     TITLE = {The geometry and topology of {C}oxeter groups},
    SERIES = {London Mathematical Society Monographs Series},
    VOLUME = {32},
 PUBLISHER = {Princeton University Press, Princeton, NJ},
      YEAR = {2008},
     PAGES = {xvi+584},
      ISBN = {978-0-691-13138-2; 0-691-13138-4},
   MRCLASS = {20F55 (05B45 05C25 51-02 57M07)},}

@article {DL98,
    AUTHOR = {Davis, J. F. and Lück, W.},
     TITLE = {Spaces over a category and assembly maps in isomorphism
              conjectures in {$K$}- and {$L$}-theory},
   JOURNAL = {$K$-Theory},
  FJOURNAL = {$K$-Theory. An Interdisciplinary Journal for the Development,
              Application, and Influence of $K$-Theory in the Mathematical
              Sciences},
    VOLUME = {15},
      YEAR = {1998},
    NUMBER = {3},
     PAGES = {201--252},
      ISSN = {0920-3036,1573-0514},
   MRCLASS = {55N20 (19D10)},
  MRNUMBER = {1659969},
MRREVIEWER = {A.\ A.\ Ranicki},
       DOI = {10.1023/A:1007784106877},
       URL = {https://doi.org/10.1023/A:1007784106877},
}

@article {DHLPS,
    AUTHOR = {Degrijse, D. and Hausmann, M. and Lück, W. and
              Patchkoria, I. and Schwede, S.},
     TITLE = {Proper equivariant stable homotopy theory},
   JOURNAL = {Mem. Amer. Math. Soc.},
  FJOURNAL = {Memoirs of the American Mathematical Society},
    VOLUME = {288},
      YEAR = {2023},
    NUMBER = {1432},
     PAGES = {vi+142},
      ISSN = {0065-9266,1947-6221},
      ISBN = {978-1-4704-6704-3; 978-1-4704-7574-1},
   MRCLASS = {55P91},
  MRNUMBER = {4627088},
MRREVIEWER = {David\ Barnes},
       DOI = {10.1090/memo/1432},
       URL = {https://doi.org/10.1090/memo/1432},
}

@article {Elm,
    AUTHOR = {Elmendorf, A. D.},
     TITLE = {Systems of fixed point sets},
   JOURNAL = {Trans. Amer. Math. Soc.},
  FJOURNAL = {Transactions of the American Mathematical Society},
    VOLUME = {277},
      YEAR = {1983},
    NUMBER = {1},
     PAGES = {275--284},
      ISSN = {0002-9947,1088-6850},
   MRCLASS = {57S99 (55N25)},
  MRNUMBER = {690052},
MRREVIEWER = {J.\ P.\ May},
       DOI = {10.2307/1999356},
       URL = {https://doi.org/10.2307/1999356},
}

@book {FarbMarg,
    AUTHOR = {Farb, B. and Margalit, D.},
     TITLE = {A primer on mapping class groups},
    SERIES = {Princeton Mathematical Series},
    VOLUME = {49},
 PUBLISHER = {Princeton University Press, Princeton, NJ},
      YEAR = {2012},
     PAGES = {xiv+472},
      ISBN = {978-0-691-14794-9},
   MRCLASS = {57M50 (20F36 20F65 57M07 57N05)},}

@article {Far78,
    AUTHOR = {Farrell, F. T.},
     TITLE = {An extension of {T}ate cohomology to a class of infinite
              groups},
   JOURNAL = {J. Pure Appl. Algebra},
  FJOURNAL = {Journal of Pure and Applied Algebra},
    VOLUME = {10},
      YEAR = {1977/78},
    NUMBER = {2},
     PAGES = {153--161},
      ISSN = {0022-4049,1873-1376},
   MRCLASS = {20J05 (18H10 22E40)},}

@article {FR,
    AUTHOR = {Floyd, E. E. and Richardson, R. W.},
     TITLE = {An action of a finite group on an {$n$}-cell without
              stationary points},
   JOURNAL = {Bull. Amer. Math. Soc.},
  FJOURNAL = {Bulletin of the American Mathematical Society},
    VOLUME = {65},
      YEAR = {1959},
     PAGES = {73--76},
      ISSN = {0002-9904},
   MRCLASS = {55.00},
  MRNUMBER = {100848},
MRREVIEWER = {P.\ A.\ Smith},
       DOI = {10.1090/S0002-9904-1959-10282-2},
       URL = {https://doi.org/10.1090/S0002-9904-1959-10282-2},
}

@article {GM95,
    AUTHOR = {Greenlees, J. P. C. and May, J. P.},
     TITLE = {Generalized {T}ate cohomology},
   JOURNAL = {Mem. Amer. Math. Soc.},
  FJOURNAL = {Memoirs of the American Mathematical Society},
    VOLUME = {113},
      YEAR = {1995},
    NUMBER = {543},
     PAGES = {viii+178},
      ISSN = {0065-9266,1947-6221},
   MRCLASS = {55N15 (19L47 55P42 55Q91 55T25)},
  MRNUMBER = {1230773},
MRREVIEWER = {V.\ P.\ Snaith},
       DOI = {10.1090/memo/0543},
       URL = {https://doi.org/10.1090/memo/0543},
}

@article {GS96,
    AUTHOR = {Greenlees, J. P. C. and Sadofsky, H.},
     TITLE = {The {T}ate spectrum of {$v_n$}-periodic complex oriented
              theories},
   JOURNAL = {Math. Z.},
  FJOURNAL = {Mathematische Zeitschrift},
    VOLUME = {222},
      YEAR = {1996},
    NUMBER = {3},
     PAGES = {391--405},
      ISSN = {0025-5874,1432-1823},
   MRCLASS = {55N22},
  MRNUMBER = {1400199},
MRREVIEWER = {Mark\ Hovey},
       DOI = {10.1007/PL00004264},
       URL = {https://doi.org/10.1007/PL00004264},
}

@article {Hard,
    AUTHOR = {Harder, G.},
     TITLE = {A {G}auss-{B}onnet formula for discrete arithmetically defined
              groups},
   JOURNAL = {Ann. Sci. \'{E}c. Norm. Sup\'{e}r. (4)},
  FJOURNAL = {Annales Scientifiques de l'\'{E}cole Normale Sup\'{e}rieure. Quatri\`eme
              S\'{e}rie},
    VOLUME = {4},
      YEAR = {1971},
     PAGES = {409--455},
      ISSN = {0012-9593},
   MRCLASS = {20H10 (10D25 22E40 53C30 57F15)},}

@article {Har,
    AUTHOR = {Harer, J. L.},
     TITLE = {The virtual cohomological dimension of the mapping class group
              of an orientable surface},
   JOURNAL = {Invent. Math.},
  FJOURNAL = {Inventiones Mathematicae},
    VOLUME = {84},
      YEAR = {1986},
    NUMBER = {1},
     PAGES = {157--176},
      ISSN = {0020-9910,1432-1297},
   MRCLASS = {32G15 (20F38 57N05)},}

@article {HZ,
    AUTHOR = {Harer, J. and Zagier, D.},
     TITLE = {The {E}uler characteristic of the moduli space of curves},
   JOURNAL = {Invent. Math.},
  FJOURNAL = {Inventiones Mathematicae},
    VOLUME = {85},
      YEAR = {1986},
    NUMBER = {3},
     PAGES = {457--485},
      ISSN = {0020-9910},
   MRCLASS = {32G15 (14H15 57R20)},}

@book {Hirz,
    AUTHOR = {Hirzebruch, F.},
     TITLE = {Hilbert modular surfaces},
    SERIES = {S\'{e}rie des Conf\'{e}rences de l'Union Math\'{e}matique Internationale,
              No. 4},
 PUBLISHER = {Universit\'{e} de Gen\`eve, Secr\'{e}tariat de l'Enseignement
              Math\'{e}matique, Geneva},
      YEAR = {1973},
     PAGES = {103},
   MRCLASS = {14J15 (10D20 14J25)},}

@article {HKR,
    AUTHOR = {Hopkins, M. J. and Kuhn, N. J. and Ravenel, D.
              C.},
     TITLE = {Generalized group characters and complex oriented cohomology
              theories},
   JOURNAL = {J. Amer. Math. Soc.},
  FJOURNAL = {Journal of the American Mathematical Society},
    VOLUME = {13},
      YEAR = {2000},
    NUMBER = {3},
     PAGES = {553--594},
      ISSN = {0894-0347,1088-6834},
   MRCLASS = {55N22 (55N20 55N91 55R40)},
  MRNUMBER = {1758754},
MRREVIEWER = {Andrew\ J.\ Baker},
       DOI = {10.1090/S0894-0347-00-00332-5},
       URL = {https://doi.org/10.1090/S0894-0347-00-00332-5},
}

@unpublished {HL13,
    AUTHOR = {Hopkins, M. J. and Lurie, J.},
     TITLE = {Ambidexterity in {$K(n)$}-local stable homotopy theory.},
      YEAR = {2013},
      Note = {https://www.math.ias.edu/~lurie/papers/Ambidexterity.pdf},
}

@article {HS96,
    AUTHOR = {Hovey, M. and Sadofsky, H.},
     TITLE = {Tate cohomology lowers chromatic {B}ousfield classes},
   JOURNAL = {Proc. Amer. Math. Soc.},
  FJOURNAL = {Proceedings of the American Mathematical Society},
    VOLUME = {124},
      YEAR = {1996},
    NUMBER = {11},
     PAGES = {3579--3585},
      ISSN = {0002-9939,1088-6826},
   MRCLASS = {55P60 (55P42 55P91)},
  MRNUMBER = {1343699},
MRREVIEWER = {Richard\ John\ Steiner},
       DOI = {10.1090/S0002-9939-96-03495-8},
       URL = {https://doi.org/10.1090/S0002-9939-96-03495-8},
}

@article {Ji2007,
    AUTHOR = {Ji, L.},
     TITLE = {Integral {N}ovikov conjectures and arithmetic groups
              containing torsion elements},
   JOURNAL = {Comm. Anal. Geom.},
  FJOURNAL = {Communications in Analysis and Geometry},
    VOLUME = {15},
      YEAR = {2007},
    NUMBER = {3},
     PAGES = {509--533},
      ISSN = {1019-8385,1944-9992},
   MRCLASS = {22E40 (57R19)},
  MRNUMBER = {2379803},
MRREVIEWER = {Jean-Louis\ Tu},
       URL = {http://projecteuclid.org/euclid.cag/1201269327},
}

@incollection {JiWolpert2010,
    AUTHOR = {Ji, L. and Wolpert, S. A.},
     TITLE = {A cofinite universal space for proper actions for mapping
              class groups},
 BOOKTITLE = {In the tradition of {A}hlfors-{B}ers. {V}},
    SERIES = {Contemp. Math.},
    VOLUME = {510},
     PAGES = {151--163},
 PUBLISHER = {Amer. Math. Soc., Providence, RI},
      YEAR = {2010},
      ISBN = {978-0-8218-4732-9},
   MRCLASS = {57M07 (20F65)},
}

@article {Klein,
    AUTHOR = {Klein, J. R.},
     TITLE = {The dualizing spectrum of a topological group},
   JOURNAL = {Math. Ann.},
  FJOURNAL = {Mathematische Annalen},
    VOLUME = {319},
      YEAR = {2001},
    NUMBER = {3},
     PAGES = {421--456},
      ISSN = {0025-5831,1432-1807},
   MRCLASS = {55P91 (20J05 55N91 55P42 57P10)},
  MRNUMBER = {1819876},
MRREVIEWER = {Richard\ John\ Steiner},
       DOI = {10.1007/PL00004441},
       URL = {https://doi.org/10.1007/PL00004441},
}

@article {Klein02,
    AUTHOR = {Klein, J. R.},
     TITLE = {Axioms for generalized {F}arrell-{T}ate cohomology},
   JOURNAL = {J. Pure Appl. Algebra},
  FJOURNAL = {Journal of Pure and Applied Algebra},
    VOLUME = {172},
      YEAR = {2002},
    NUMBER = {2-3},
     PAGES = {225--238},
      ISSN = {0022-4049,1873-1376},
   MRCLASS = {55N35 (20J06 55P91)},
  MRNUMBER = {1906876},
MRREVIEWER = {Teimuraz\ Pirashvili},
       DOI = {10.1016/S0022-4049(01)00151-7},
       URL = {https://doi.org/10.1016/S0022-4049(01)00151-7},
}

@unpublished {HKK,
    AUTHOR = {Hilman, K. and Kirstein, D. and Kremer, C.},
     TITLE = {Parametrised {P}oincar\'e duality and equivariant fixed points methods},
      YEAR = {2024},
      Note = {arXiv:2405.17641},
}

@unpublished {BDGNS,
    AUTHOR = {Barwick, C. and Dotto, E. and Glasman, S. and Nardin, D. and Shah, J.},
     TITLE = {Parametrized higher category theory and higher algebra: Expos\'e {I} – Elements of
parametrized higher category theory},
      YEAR = {2016},
      Note = {arXiv:1608.03657},
}

@unpublished {Nardin,
    AUTHOR = {Nardin, D.},
     TITLE = {Parametrized higher category theory and higher algebra: Expos\'e {IV} –
{S}tability with respect to an orbital {$\infty$}-category},
      YEAR = {2016},
      Note = {arXiv:1608.07704},
}

@article{Kumregular,
 Author = {Kummer, E. E.},
 TITLE = {Allgemeiner {Beweis} des {Fermatschen} {Satzes}, da{{\ss}} die {Gleichung} {{\(x^\lambda+y^\lambda=z^\lambda\)}} durch ganze Zahlen unl{\"o}sbar ist, f{\"u}r alle diejenigen Potenz-Exponenten {{\(\lambda\)}}, welche ungerade Primzahlen sind und in den Z{\"a}hlern der ersten {{\(\frac 1{2}(\lambda-3)\)}} {Bernoullischen} {Zahlen} als {Factoren} nicht vorkommen.},
 FJOURNAL = {Journal f{\"u}r die Reine und Angewandte Mathematik},
 JOURNAL= {J. Reine Angew. Math.},
 ISSN = {0075-4102},
 VOLUME = {40},
 PAGES = {131--138},
 YEAR = {1850},}

@article{Kummercong,
 AUTHOR = {Kummer, E. E.},
 TITLE = {{\"U}ber eine allgemeine {Eigenschaft} der rationalen {Entwicklungsco{\"e}fficienten} einer bestimmten {Gattung} analytischer {Functionen}.},
 FJOURNAL = {Journal f{\"u}r die Reine und Angewandte Mathematik},
 JOURNAL = {J. Reine Angew. Math.},
 ISSN = {0075-4102},
 VOLUME = {41},
 PAGES = {368--372},
 YEAR = {1851},}

@article {LearyNucinkis,
    AUTHOR = {Leary, I. J. and Nucinkis, B. E. A.},
     TITLE = {Some groups of type {$VF$}},
   JOURNAL = {Invent. Math.},
  FJOURNAL = {Inventiones Mathematicae},
    VOLUME = {151},
      YEAR = {2003},
    NUMBER = {1},
     PAGES = {135--165},
      ISSN = {0020-9910,1432-1297},
   MRCLASS = {20J05 (20F34)},
  MRNUMBER = {1943744},
MRREVIEWER = {David\ Meier},
       DOI = {10.1007/s00222-002-0254-7},
       URL = {https://doi.org/10.1007/s00222-002-0254-7},
}

@unpublished {Li,
    AUTHOR = {Li, Y.},
     TITLE = {{C}ardinalities in Height 1},
      YEAR = {2025},
      Note = {Master thesis, University of Copenhagen, arXiv:2505.09150},
}

@incollection {LuckSurvey,
    AUTHOR = {Lück, W.},
     TITLE = {Survey on classifying spaces for families of subgroups},
 BOOKTITLE = {Infinite groups: geometric, combinatorial and dynamical
              aspects},
    SERIES = {Progr. Math.},
    VOLUME = {248},
     PAGES = {269--322},
 PUBLISHER = {Birkh\"auser, Basel},
      YEAR = {2005},
      ISBN = {978-3-7643-7446-4; 3-7643-7446-2},
   MRCLASS = {55R35 (19D55 19G24 20F65 57S25)},
  MRNUMBER = {2195456},
MRREVIEWER = {R.\ M.\ Vogt},
       DOI = {10.1007/3-7643-7447-0\_7},
       URL = {https://doi.org/10.1007/3-7643-7447-0_7},
}

@article {LO01,
    AUTHOR = {Lück, W. and Oliver, B.},
     TITLE = {The completion theorem in {$K$}-theory for proper actions of a
              discrete group},
   JOURNAL = {Topology},
  FJOURNAL = {Topology. An International Journal of Mathematics},
    VOLUME = {40},
      YEAR = {2001},
    NUMBER = {3},
     PAGES = {585--616},
      ISSN = {0040-9383},
   MRCLASS = {19L47 (55N15 55N91)},
  MRNUMBER = {1838997},
MRREVIEWER = {Kazuhisa\ Shimakawa},
       DOI = {10.1016/S0040-9383(99)00077-4},
       URL = {https://doi.org/10.1016/S0040-9383(99)00077-4},
}

@article {Luckfinite,
    AUTHOR = {Lück, W.},
     TITLE = {The type of the classifying space for a family of subgroups},
   JOURNAL = {J. Pure Appl. Algebra},
  FJOURNAL = {Journal of Pure and Applied Algebra},
    VOLUME = {149},
      YEAR = {2000},
    NUMBER = {2},
     PAGES = {177--203},
      ISSN = {0022-4049,1873-1376},
   MRCLASS = {55R35 (57S30)},
  MRNUMBER = {1757730},
MRREVIEWER = {Allan\ Edmonds},
       DOI = {10.1016/S0022-4049(98)90173-6},
       URL = {https://doi.org/10.1016/S0022-4049(98)90173-6},
}

@article {Mislin2010,
    AUTHOR = {Mislin, G.},
     TITLE = {Classifying spaces for proper actions of mapping class groups},
   JOURNAL = {M\"{u}nster J. Math.},
  FJOURNAL = {M\"{u}nster Journal of Mathematics},
    VOLUME = {3},
      YEAR = {2010},
     PAGES = {263--272},
      ISSN = {1867-5778,1867-5786},
   MRCLASS = {55R35 (20F38 57M99)},
}

@book {MaySig,
    AUTHOR = {May, J. P. and Sigurdsson, J.},
     TITLE = {Parametrized homotopy theory},
    SERIES = {Mathematical Surveys and Monographs},
    VOLUME = {132},
 PUBLISHER = {American Mathematical Society, Providence, RI},
      YEAR = {2006},
     PAGES = {x+441},
      ISBN = {978-0-8218-3922-5; 0-8218-3922-5},
   MRCLASS = {55P42 (19L99 55N20 55N22 55P91)},
  MRNUMBER = {2271789},
MRREVIEWER = {A.\ A.\ Ranicki},
       DOI = {10.1090/surv/132},
       URL = {https://doi.org/10.1090/surv/132},
}

@article {LevittNicolas98,
    AUTHOR = {Levitt, G. and Nicolas, J-L.},
     TITLE = {On the maximum order of torsion elements in {${\rm GL}(n,{\bf
              Z})$} and {${\rm Aut}(F_n)$}},
   JOURNAL = {J. Algebra},
  FJOURNAL = {Journal of Algebra},
    VOLUME = {208},
      YEAR = {1998},
    NUMBER = {2},
     PAGES = {630--642},
      ISSN = {0021-8693,1090-266X},
   MRCLASS = {20G30 (20F28)},
  MRNUMBER = {1655470},
MRREVIEWER = {Bruno\ P.\ Zimmermann},
       DOI = {10.1006/jabr.1998.7481},
       URL = {https://doi.org/10.1006/jabr.1998.7481},
}

@book {LMS,
    AUTHOR = {Lewis, Jr., L. G. and May, J. P. and Steinberger, M. and
              McClure, J. E.},
     TITLE = {Equivariant stable homotopy theory},
    SERIES = {Lecture Notes in Mathematics},
    VOLUME = {1213},
      NOTE = {With contributions by J. E. McClure},
 PUBLISHER = {Springer-Verlag, Berlin},
      YEAR = {1986},
     PAGES = {x+538},
      ISBN = {3-540-16820-6},
   MRCLASS = {55-02 (55Nxx 55Pxx 57S99)},
  MRNUMBER = {866482},
MRREVIEWER = {T.\ tom Dieck},
       DOI = {10.1007/BFb0075778},
       URL = {https://doi.org/10.1007/BFb0075778},
}

@article {LNP,
    AUTHOR = {Linskens, S. and Nardin, D. and Pol, L.},
     TITLE = {Global homotopy theory via partially lax limits},
   JOURNAL = {Geom. Topol.},
  FJOURNAL = {Geometry \& Topology},
    VOLUME = {29},
      YEAR = {2025},
    NUMBER = {3},
     PAGES = {1345--1440},
      ISSN = {1465-3060,1364-0380},
   MRCLASS = {55N91 (18N70 55P91)},
  MRNUMBER = {4918109},
MRREVIEWER = {Hemant\ Kumar\ Singh},
       DOI = {10.2140/gt.2025.29.1345},
       URL = {https://doi.org/10.2140/gt.2025.29.1345},
}

@unpublished {LPS2024,
    AUTHOR = {Lück, W. and Patchkoria, I. and Schwede, S.},
     TITLE = {{C}haracter theory and {E}uler characteristic for orbispaces and infinite groups},
      YEAR = {2024},
      Note = {arXiv:2410.14510},
}

@article {Lu2001,
    AUTHOR = {Lu, Q.},
     TITLE = {Periodicity of the punctured mapping class group},
   JOURNAL = {J. Pure Appl. Algebra},
  FJOURNAL = {Journal of Pure and Applied Algebra},
    VOLUME = {155},
      YEAR = {2001},
    NUMBER = {2-3},
     PAGES = {211--235},
      ISSN = {0022-4049,1873-1376},
   MRCLASS = {57M60 (20J15 57M20)},
  MRNUMBER = {1801416},
MRREVIEWER = {Denis\ Sjerve},
       DOI = {10.1016/S0022-4049(99)00112-7},
       URL = {https://doi.org/10.1016/S0022-4049(99)00112-7},
}

@book {LurieHTT,
    AUTHOR = {Lurie, J.},
     TITLE = {Higher topos theory},
    SERIES = {Annals of Mathematics Studies},
    VOLUME = {170},
 PUBLISHER = {Princeton University Press, Princeton, NJ},
      YEAR = {2009},
     PAGES = {xviii+925},
      ISBN = {978-0-691-14049-0; 0-691-14049-9},
   MRCLASS = {18-02 (18B25 18E35 18G30 18G55 55U40)},
  MRNUMBER = {2522659},
MRREVIEWER = {Mark\ Hovey},
       DOI = {10.1515/9781400830558},
       URL = {https://doi.org/10.1515/9781400830558},
}

@unpublished {LurieEl,
    AUTHOR = {Lurie, J.},
     TITLE = {Elliptic cohomology III: Tempered cohomology},
      YEAR = {2019},
      Note = {https://www.math.ias.edu/~lurie/papers/Elliptic-III-Tempered.pdf},
}

@book {LL2,
    AUTHOR = {Lück, W.},
     TITLE = {{$L^2$}-invariants: theory and applications to geometry and
              {$K$}-theory},
    SERIES = {Ergebnisse der Mathematik und ihrer Grenzgebiete. 3. Folge. A
              Series of Modern Surveys in Mathematics [Results in
              Mathematics and Related Areas. 3rd Series. A Series of Modern
              Surveys in Mathematics]},
    VOLUME = {44},
 PUBLISHER = {Springer-Verlag, Berlin},
      YEAR = {2002},
     PAGES = {xvi+595},
      ISBN = {3-540-43566-2},
   MRCLASS = {58J22 (19K56 46L80 57Q10 57R20 58J52)},
  MRNUMBER = {1926649},
MRREVIEWER = {Thomas\ Schick},
       DOI = {10.1007/978-3-662-04687-6},
       URL = {https://doi.org/10.1007/978-3-662-04687-6},
}

@unpublished {Martini,
    AUTHOR = {Martini, L.},
     TITLE = {Cocartesian fibrations and straightening internal to an {$\infty$}-topos},
      YEAR = {2022},
      Note = {arXiv:2204.00295},
}

@unpublished {MW1,
    AUTHOR = {Martini, L. and Wolf, S.},
     TITLE = {Presentability and topoi in internal higher category theory},
      YEAR = {2022},
      Note = {arXiv:2209.05103},
}

@article {MW2,
    AUTHOR = {Martini, L. and Wolf, S.},
     TITLE = {Colimits and cocompletions in internal higher category theory},
   JOURNAL = {High. Struct.},
  FJOURNAL = {Higher Structures},
    VOLUME = {8},
      YEAR = {2024},
    NUMBER = {1},
     PAGES = {97--192},
      ISSN = {2209-0606},
   MRCLASS = {18Nxx (18Axx 18B25 18F20 55Pxx)},
  MRNUMBER = {4752519},
MRREVIEWER = {Tobias\ Lenz},
       DOI = {10.21136/HS.2024.03},
       URL = {https://doi.org/10.21136/HS.2024.03},
}

@article {Nielsen,
    AUTHOR = {Nielsen, J.},
     TITLE = {Abbildungsklassen endlicher {O}rdnung},
   JOURNAL = {Acta Math.},
  FJOURNAL = {Acta Mathematica},
    VOLUME = {75},
      YEAR = {1943},
     PAGES = {23--115},
      ISSN = {0001-5962,1871-2509},
   MRCLASS = {56.0X},}

@article {NS,
    AUTHOR = {Nikolaus, T. and Scholze, P.},
     TITLE = {On topological cyclic homology},
   JOURNAL = {Acta Math.},
  FJOURNAL = {Acta Mathematica},
    VOLUME = {221},
      YEAR = {2018},
    NUMBER = {2},
     PAGES = {203--409},
      ISSN = {0001-5962,1871-2509},
   MRCLASS = {55U35 (16E40 18E30 19D99)},
  MRNUMBER = {3904731},
MRREVIEWER = {Geoffrey\ M. L. Powell},
       DOI = {10.4310/ACTA.2018.v221.n2.a1},
       URL = {https://doi.org/10.4310/ACTA.2018.v221.n2.a1},
}

@article {CSYambheight,
    AUTHOR = {Carmeli, S. and Schlank, T. M. and Yanovski, L.},
     TITLE = {Ambidexterity and height},
   JOURNAL = {Adv. Math.},
  FJOURNAL = {Advances in Mathematics},
    VOLUME = {385},
      YEAR = {2021},
     PAGES = {Paper No. 107763, 90},
      ISSN = {0001-8708,1090-2082},
   MRCLASS = {18N60 (55U35)},
  MRNUMBER = {4246977},
       DOI = {10.1016/j.aim.2021.107763},
       URL = {https://doi.org/10.1016/j.aim.2021.107763},
}

@article {CSYambidexterity,
    AUTHOR = {Carmeli, S. and Schlank, T. M. and Yanovski, L.},
     TITLE = {Ambidexterity in chromatic homotopy theory},
   JOURNAL = {Invent. Math.},
  FJOURNAL = {Inventiones Mathematicae},
    VOLUME = {228},
      YEAR = {2022},
    NUMBER = {3},
     PAGES = {1145--1254},
      ISSN = {0020-9910,1432-1297},
   MRCLASS = {55P43 (55P60)},
  MRNUMBER = {4419631},
MRREVIEWER = {Lennart\ Meier},
       DOI = {10.1007/s00222-022-01099-9},
       URL = {https://doi.org/10.1007/s00222-022-01099-9},
}

@unpublished {P24,
    AUTHOR = {Patchkoria, I.},
     TITLE = {Chromatic congruences and Bernoulli numbers},
      YEAR = {2024},
      Note = {arXiv:2406.17705},
}

@article {Prestel,
    AUTHOR = {Prestel, A.},
     TITLE = {Die {F}ixpunkte der symmetrischen {H}ilbertschen {M}odulgruppe
              zu einem reell-quadratischen {Z}ahlk\"{o}rper mit
              {P}rimzahldiskriminante},
   JOURNAL = {Math. Ann.},
  FJOURNAL = {Mathematische Annalen},
    VOLUME = {200},
      YEAR = {1973},
     PAGES = {123--139},
      ISSN = {0025-5831,1432-1807},
   MRCLASS = {10D05},}

@article {RV16,
    AUTHOR = {Reich, H. and Varisco, M.},
     TITLE = {On the {A}dams isomorphism for equivariant orthogonal spectra},
   JOURNAL = {Algebr. Geom. Topol.},
  FJOURNAL = {Algebraic \& Geometric Topology},
    VOLUME = {16},
      YEAR = {2016},
    NUMBER = {3},
     PAGES = {1493--1566},
      ISSN = {1472-2747,1472-2739},}

@book {Serre,
    AUTHOR = {Serre, J.-P.},
     TITLE = {Trees},
    SERIES = {Springer Monographs in Mathematics},
      NOTE = {Translated from the French original by John Stillwell,
              Corrected 2nd printing of the 1980 English translation},
 PUBLISHER = {Springer-Verlag, Berlin},
      YEAR = {2003},
     PAGES = {x+142},
      ISBN = {3-540-44237-5},
   MRCLASS = {20E08 (05C05 20E06 20G25)},
  MRNUMBER = {1954121},
}

@incollection {Serre1,
    AUTHOR = {Serre, J.-P.},
     TITLE = {Cohomologie des groupes discrets},
 BOOKTITLE = {S\'{e}minaire {B}ourbaki, 23\`eme ann\'{e}e (1970/1971), {E}xp. {N}o.
              399},
     PAGES = {337--350. Lecture Notes in Math., Vol. 244},
      YEAR = {1971},
   MRCLASS = {22E40},}

@article {Shah,
    AUTHOR = {Shah, J.},
     TITLE = {Parametrized higher category theory},
   JOURNAL = {Algebr. Geom. Topol.},
  FJOURNAL = {Algebraic \& Geometric Topology},
    VOLUME = {23},
      YEAR = {2023},
    NUMBER = {2},
     PAGES = {509--644},
      ISSN = {1472-2747,1472-2739},
   MRCLASS = {55U35 (55U10 55U40)},
  MRNUMBER = {4587313},
MRREVIEWER = {Philippe\ Gaucher},
       DOI = {10.2140/agt.2023.23.509},
       URL = {https://doi.org/10.2140/agt.2023.23.509},
}

@article {SjerveYang97,
    AUTHOR = {Sjerve, D. and Yang, Q.},
     TITLE = {Conjugacy classes of {$p$}-torsion in {${\rm Sp}_{p-1}(\bold
              Z)$}},
   JOURNAL = {J. Algebra},
  FJOURNAL = {Journal of Algebra},
    VOLUME = {195},
      YEAR = {1997},
    NUMBER = {2},
     PAGES = {580--603},
      ISSN = {0021-8693,1090-266X},
   MRCLASS = {20G30 (11R18 30F10)},
  MRNUMBER = {1469641},
MRREVIEWER = {Naohisa\ Shimomura},
       DOI = {10.1006/jabr.1997.7084},
       URL = {https://doi.org/10.1006/jabr.1997.7084},
}

@article{Staudt,
 AUTHOR = {Staudt, K. G. C.},
 TITLE = {Beweis eines {Lehrsatzes}, die {Bernoullischen} {Zahlen} betreffend.},
 FJOURNAL = {Journal f{\"u}r die Reine und Angewandte Mathematik},
 JOURNAL = {J. Reine Angew. Math.},
 ISSN = {0075-4102},
 VOLUME = {21},
 PAGES = {372--374},
 YEAR = {1840},}

@article {Brown74,
    AUTHOR = {Brown, K. S.},
     TITLE = {Euler characteristics of discrete groups and {$G$}-spaces},
   JOURNAL = {Invent. Math.},
  FJOURNAL = {Inventiones Mathematicae},
    VOLUME = {27},
      YEAR = {1974},
     PAGES = {229--264},
      ISSN = {0020-9910,1432-1297},
   MRCLASS = {22E40 (12A70 20G10 57D20)},
  MRNUMBER = {385007},
MRREVIEWER = {H.\ Bass},
       DOI = {10.1007/BF01390176},
       URL = {https://doi.org/10.1007/BF01390176},
}

@article {Stap,
    AUTHOR = {Stapleton, N.},
     TITLE = {Transchromatic generalized character maps},
   JOURNAL = {Algebr. Geom. Topol.},
  FJOURNAL = {Algebraic \& Geometric Topology},
    VOLUME = {13},
      YEAR = {2013},
    NUMBER = {1},
     PAGES = {171--203},
      ISSN = {1472-2747},
   MRCLASS = {55N20 (55N91)},
}

@book {tomD,
    AUTHOR = {tom Dieck, T.},
     TITLE = {Transformation groups},
    SERIES = {De Gruyter Studies in Mathematics},
    VOLUME = {8},
 PUBLISHER = {Walter de Gruyter \& Co., Berlin},
      YEAR = {1987},
     PAGES = {x+312},
      ISBN = {3-11-009745-1},
   MRCLASS = {57Sxx (57-02)},
  MRNUMBER = {889050},
MRREVIEWER = {Shmuel\ Weinberger},
       DOI = {10.1515/9783110858372.312},
       URL = {https://doi.org/10.1515/9783110858372.312},
}

@article {Wall,
    AUTHOR = {Wall, C. T. C.},
     TITLE = {Rational {E}uler characteristics},
   JOURNAL = {Proc. Cambridge Philos. Soc.},
  FJOURNAL = {Proceedings of the Cambridge Philosophical Society},
    VOLUME = {57},
      YEAR = {1961},
     PAGES = {182--184},
      ISSN = {0008-1981},
   MRCLASS = {18.20},}

@article {Wir,
    AUTHOR = {Wirthm\"uller, K.},
     TITLE = {Equivariant homology and duality},
   JOURNAL = {Manuscripta Math.},
  FJOURNAL = {Manuscripta Mathematica},
    VOLUME = {11},
      YEAR = {1974},
     PAGES = {373--390},
      ISSN = {0025-2611,1432-1785},
   MRCLASS = {55B25 (55C05)},
  MRNUMBER = {343260},
MRREVIEWER = {W.\ D.\ Neumann},
       DOI = {10.1007/BF01170239},
       URL = {https://doi.org/10.1007/BF01170239},
}

@incollection {Xia,
    AUTHOR = {Xia, Y.},
     TITLE = {The {$p$}-torsion of the {F}arrell-{T}ate cohomology of the
              mapping class group {$\Gamma_{(p-1)/2}$}},
 BOOKTITLE = {Topology '90 ({C}olumbus, {OH}, 1990)},
    SERIES = {Ohio State Univ. Math. Res. Inst. Publ.},
    VOLUME = {1},
     PAGES = {391--398},
 PUBLISHER = {de Gruyter, Berlin},
      YEAR = {1992},
      ISBN = {3-11-012598-6},
   MRCLASS = {57M07 (20J06 57N05)},}

@article {Yan,
    AUTHOR = {Yanovski, L.},
     TITLE = {Homotopy cardinality via extrapolation of {M}orava-{E}uler
              characteristics},
   JOURNAL = {Selecta Math. (N.S.)},
  FJOURNAL = {Selecta Mathematica. New Series},
    VOLUME = {29},
      YEAR = {2023},
    NUMBER = {5},
     PAGES = {Paper No. 81, 36},
      ISSN = {1022-1824,1420-9020},
   MRCLASS = {55P42},
  MRNUMBER = {4660289},
MRREVIEWER = {Birgit\ Richter},
       DOI = {10.1007/s00029-023-00886-3},
       URL = {https://doi.org/10.1007/s00029-023-00886-3},
}

@article{CLL6functor,
  title={Universality of span 2-categories and the construction of 6-functor formalisms},
  author={Cnossen, B. and Lenz, T. and Linskens, S.},
  note={arXiv:2505.19192},
  year={2025}
}

@article{kuhntate,
  title={Tate cohomology and periodic localization of polynomial functors},
  author={Kuhn, N. J.},
  journal={Inventiones mathematicae},
  volume={157},
  number={2},
  pages={345--370},
  year={2004},
  publisher={Springer}
}

@article{elmantohaugseng,
  title={On distributivity in higher algebra {I}: the universal property of bispans},
  author={Elmanto, E. and Haugseng, R.},
  journal={Compositio Mathematica},
  volume={159},
  number={11},
  pages={2326--2415},
  year={2023},
  publisher={London Mathematical Society}
}

@article{ravenellocalization,
  title={Localization with respect to certain periodic homology theories},
  author={Ravenel, D. C.},
  journal={American Journal of Mathematics},
  volume={106},
  number={2},
  pages={351--414},
  year={1984},
  publisher={JSTOR}
}
\end{document}